\documentclass[10pt,reqno]{amsart}

\usepackage[utf8]{inputenc}
\usepackage{amsfonts}
\usepackage{amsmath}
\usepackage{amscd}
\usepackage{mathrsfs}
\usepackage{amsthm}
\usepackage{amssymb}
\usepackage[top=30truemm,bottom=30truemm,left=30truemm,right=30truemm]{geometry}
\usepackage[bookmarks=false,draft=false,breaklinks,colorlinks]{hyperref}
\usepackage[dvipdfmx]{graphicx}
\usepackage[all]{xy}

\usepackage{color}
\usepackage{url}

\usepackage{comment}

\usepackage{fancyhdr}
\usepackage{enumitem}
\setenumerate{label=(\arabic*),nosep}
\setitemize{nosep}
\newlist{clist}{enumerate}{1}
\setlist*[clist]{label=(\roman*),nosep}

\theoremstyle{definition}
\newtheorem{Def}{Definition}[section]
\newtheorem{Thm}[Def]{Theorem}
\newtheorem{Prop}[Def]{Proposition}
\newtheorem{Lem}[Def]{Lemma}
\newtheorem{Cor}[Def]{Corollary}
\newtheorem{Rem}[Def]{Remark}

\newcommand{\lm}{\lambda} 

\title{Differential calculus for fully matricial functions I}
\author{Hyuga Ito}

\address{
Graduate School of Mathematics, Nagoya University, Furocho, Chikusaku, Nagoya, 464-8602, Japan
}
\email{hyuga.ito.e6@math.nagoya-u.ac.jp}

\date{\today}

\begin{document}

\maketitle
\begin{abstract}
    We will introduce a cyclic derivative for fully (stably) matricial functions and study its basic properties. In particular, we will show the Poincar\'{e} lemma for stably matricial functions of certain classes. We will also position Voiculescu's framework of fully matricial functions in the context of nc functions due to Kaliuzhnyi-Verbovetskyi and Vinnikov in order to clarify the relation between the present work and previous related works.
\end{abstract}
\allowdisplaybreaks{
\tableofcontents

\section{Introduction}
Voiculescu initiated ``free probability theory'' to analyze free group factors in the 1980s. Since then, many authors have established many analogues of classical probability theory in free probability theory. One of them is the introduction of analogues of entropy and Fisher's information measure (called free entropy and free Fisher's information measure, respectively), which were introduced by Voiculescu. In their study, Voiculescu introduced a certain non-commutative differential operator, which is called the \textit{free difference quotient}, to consider the non-commutative substitute of Hilbert transform \cite{v98}. 

Then, Voiculescu also introduced a certain algebraic structure related to the free difference quotient, which is the concept of \textit{generalized difference quotient rings (GDQ rings)} and is similar to the concept of Hopf algebras, and he used it to explain the mechanism of analytic subordination results of non-commutative Cauchy transforms \cite{v00b}. In the paper, Voiculescu also found the self-dual property of GDQ rings, and he went into the study of duality of GDQ rings \cite{v04,v10}. 

In these two papers, Voiculescu considered non-commutative analogues of those observations \cite[Appendix I]{v10}, using certain non-commutative functions on his affine space $M(B)$ and on his non-commutative Riemann sphere $Gr(B)$ (see bellow for precise definitions), which were named \textit{fully matricial functions}. Remark that the space of fully matricial functions has the GDQ ring structure, but is not a GDQ ring (i.e., in general, multivariable fully matricial functions are not tensor products of one-variable fully matricial functions). In \cite{v10}, Voiculescu also introduced a more general concept of non-commutative functions, which are called \textit{stably matricial functions}, and studied their series expansions, etc.

Moreover, Voiculescu has studied the \textit{cyclic derivative} related to free entropy (the concept of cyclic derivatives is more classical), and he established the Poincar\'{e} lemma and the existence of some exact sequence with respect to the cyclic derivative for $\mathbb{C}$-coefficients non-commutative polynomials \cite{v00a}. After that, Mai and Speicher generalized Voiculescu's work to general GDQ rings \cite{ms19}. 

The purpose of this paper is to introduce the cyclic derivative in the context of \cite{v04,v10}, and study its basic properties. In particular, we will show the Poincar\'{e} lemma for stably matricial analytic functions of certain classes. In addition, we will examine analogues of Voiculescu's work \cite{v00a} in the setting of $B$-coefficients non-commutative polynomials and of some special subalgebra of the algebra of fully matricial analytic functions.

There are various approaches to non-commutative functions. Thus, we will also position Voiculescu's fully matricial function theory in the context of \cite{kvv14} due to Kaliuzhnyi-Verbovetskyi and Vinnikov in order to clarify the relation between the present work and two previous works \cite{kvsv20,a20}.

\section{Preliminaries}
Let us recall some basic definitions and prepare some notations. Throughout this paper, let $E$ be a unital Banach algebra over $\mathbb{C}$ and $B$ a unital Banach subalgebra of $E$. We denote by $M(B)$ the family $(M_{n}(B))_{n\in\mathbb{N}}$, where $M_{n}(B)$ denotes the space of all $B$-coefficients $n$ by $n$ matrices.
\begin{Def}(\cite[Definition 6.1]{v04})\label{deffmBs}
    We say that $\Omega=(\Omega_{n})_{n\in\mathbb{N}}$ is an \textit{(affine) fully matricial $B$-set} if $\Omega$ satisfies the following conditions:
    \begin{enumerate}
        \item $\Omega_{n}\subset M_{n}(B)$ for each $n\in\mathbb{N}$.
        \item $\Omega_{n_{1}}\oplus\Omega_{n_{2}}=\Omega_{n_{1}+n_{2}}\cap(M_{n_{1}}(B)\oplus M_{n_{2}}(B))$ for any $n_{1},n_{2}\in\mathbb{N}$.
        \item $(\mathrm{Ad}(s)\otimes\mathrm{id}_{B})\Omega_{n}=\Omega_{n}$ for any $n\in\mathbb{N}$ and $s\in GL_{n}(\mathbb{C})$.
    \end{enumerate}
    
    An affine fully matricial $B$-set is \textit{open} if $\Omega_{n}$ is open for each $n\in\mathbb{N}$.
\end{Def}

\begin{Rem}
    In the definition of affine fully matricial $B$-sets, if we replace condition (3) with condition $(3)'$ bellow, then it becomes the definition of \textit{affine stably matricial $B$-sets} (\cite[Definition 11.1]{v10}):
    \begin{align*}
        (3)'\quad&\mbox{If $(\mathrm{Ad}(s)\otimes\mathrm{id}_{B})[\beta^{(1)}\oplus\beta^{(2)}]\in\Omega_{n_{1}+n_{2}}$ for some $s\in GL_{n_{1}+n_{2}}(\mathbb{C})$ and $\beta^{(i)}\in M_{n_{i}}(B)$, $i=1,2$},\\ &\mbox{then there exist $s_{i}\in GL_{n_{i}}(\mathbb{C})$, $i=1,2$, such that $(\mathrm{Ad}(s_{i})\otimes\mathrm{id}_{B})\beta^{(i)}\in\Omega_{n_{i}}$, $i=1,2$}.
    \end{align*}
\end{Rem}

For example, $M(B)=(M_{n}(B))_{n\in\mathbb{N}}$ is clearly an open affine fully matricial $B$-set. Also, when $B$ is a unital C*-algebra, the matricial disk $R\mathcal{D}_{0}(B)=(R\mathcal{D}_{0}(B)_{n})_{n\in\mathbb{N}}$ with $R\mathcal{D}_{0}(B)_{n}=\{RT\,|\,T\in M_{n}(B),\|T\|<1\}$ for any $R>0$ and the matricial unitary group $\mathcal{U}(B)=(\mathcal{U}(B)_{n})_{n\in\mathbb{N}}$ with $\mathcal{U}(B)_{n}=\{U\in M_{n}(B)\,|\,UU^*=I_{n}\otimes1=U^*U\}$ are both affine stably matricial $B$-sets ($R\mathcal{D}_{0}(B)$ is open), but are not fully matricial $B$-sets (see \cite[Proposition 11.2]{v10}). 

Voiculescu defined a certain non-commutative Riemann sphere $Gr(B)$ over $B$ as follows.
\begin{Def}(\cite[subsection 3.1]{v10})
    We define the equivalence relation $\widetilde{\lambda n}$ on $GL_{2}(M_{n}(B))$ as follows.
    \begin{align*}
        \begin{bmatrix}
            a&b\\
            c&d
        \end{bmatrix}
        \widetilde{\lambda n}
        \begin{bmatrix}
            a'&b'\\
            c'&d'
        \end{bmatrix}
        \mbox{ if and only if there exists $z\in GL_{n}(B)$ such that }
        b=b'z\mbox{ and }d=d'z
    \end{align*}
    for any $n\in\mathbb{N}$ and 
    $
    \left[
    \begin{smallmatrix}
        a&b\\
        c&d
    \end{smallmatrix}
    \right],
    \left[
    \begin{smallmatrix}
        a'&b'\\
        c'&d'
    \end{smallmatrix}
    \right]
    \in GL_{2}(M_{n}(B))
    $
    with $a,b,c,d,a',b',c',d'\in M_{n}(B)$.
    Then, we define $Gr(B)=(Gr_{n}(B))_{n\in\mathbb{N}}$ as follows.
    \begin{align*}
        Gr_{n}(B)=GL_{2}(M_{n}(B))/\widetilde{\lambda n}
    \end{align*}
    for any $n\in\mathbb{N}$.
\end{Def}
\begin{Rem}
    In the definitions of affine fully and stably matricial $B$-sets, if we replace $M(B)=(M_{n}(B))_{n\in\mathbb{N}}$ with $Gr(B)=(Gr_{n}(B))_{n\in\mathbb{N}}$, then they become the definitions of \textit{fully and stably matricial $B$-sets of $Gr(B)$} (see \cite[subsection 3.3 and Definition 11.1]{v10}). Here, note that the usual direct sum $\oplus$ and the adjoint action $\mathrm{Ad}(\cdot)$ by $GL_{n}(\mathbb{C})$ are replaced with
    \begin{align*}
        \left(
        \begin{bmatrix}
            a&b\\
            c&d
        \end{bmatrix}
        /\widetilde{\lambda n}
        \right)
        \oplus
        \left(
        \begin{bmatrix}
            a'&b'\\
            c'&d'
        \end{bmatrix}
        /\widetilde{\lambda n'}
        \right)
        =
        \begin{bmatrix}
            a\oplus a'&b\oplus b'\\
            c\oplus c'&d\oplus d'
        \end{bmatrix}
        /\widetilde{\lambda n+n'}
    \end{align*}
    and 
    \begin{align*}
        s\cdot
        \left(
        \begin{bmatrix}
            a&b\\
            c&d
        \end{bmatrix}
        /\widetilde{\lambda n}
        \right)
        &=
        \begin{bmatrix}
            (s\otimes1)a&(s\otimes1)b\\
            (s\otimes1)c&(s\otimes1)d
        \end{bmatrix}
        /\widetilde{\lambda n}\\
        &
        =
        \begin{bmatrix}
            (s\otimes1)a(s\otimes1)^{-1}&(s\otimes1)b(s\otimes1)^{-1}\\
            (s\otimes1)c(s\otimes1)^{-1}&(s\otimes1)d(s\otimes1)^{-1}
        \end{bmatrix}
        /\widetilde{\lambda n},
    \end{align*}
    respectively, for any 
    $
    \left[
    \begin{smallmatrix}
        a&b\\
        c&d
    \end{smallmatrix}
    \right]
    /\widetilde{\lambda n}\in Gr_{n}(B)
    $
    ,
    $
    \left[
    \begin{smallmatrix}
        a'&b'\\
        c'&d'
    \end{smallmatrix}
    \right]
    /\widetilde{\lambda n'}\in Gr_{n'}(B)
    $
    ,
    $s\in GL_{n}(\mathbb{C})$ and $n\in\mathbb{N}$ (see \cite[subsection 3.2]{v10}).
\end{Rem}
\begin{Rem}
    By definition, it is easy to see that every fully matricial $B$-set is stably matricial.
\end{Rem}
The following objects are the main ones that we will consider in this paper:
\begin{Def}(\cite[Defintion 6.5, subsection 7.8]{v04} and \cite[Definition 11.1]{v10})\label{defkariablefuuly}
    Let $\Omega$ be an affine fully (resp. stably) matricial $B$-set and $\mathcal{U}$ be a Banach space over $\mathbb{C}$. Then, we say that $f=(f_{n})_{n\in\mathbb{N}}$ is an \textit{(affine) fully (resp. stably) matricial $\mathcal{U}$-valued function} on $\Omega$ if $f$ satisfies the following conditions:
    \begin{enumerate}
        \item $f_{n}$ is a map from $\Omega_{n}$ to $M_{n}(\mathcal{U})=M_{n}(\mathbb{C})\otimes\mathcal{U}$ for each $n\in\mathbb{N}$.
        \item $f_{n_{1}+n_{2}}(\beta^{(1)}\oplus\beta^{(2)})=f_{n_{1}}(\beta^{(1)})\oplus f_{n_{2}}(\beta^{(2)})$ for any $n_{i}\in\mathbb{N}$ and $\beta^{(i)}\in \Omega_{n_{i}}$, $i=1,2$.
        \item If $(\mathrm{Ad}(s)\otimes\mathrm{id}_{B})\beta=\beta'$ for some $\beta,\beta'\in\Omega_{n}$ and $s\in GL_{n}(\mathbb{C})$, then we have $f_{n}(\beta')=(\mathrm{Ad}(s)\otimes\mathrm{id}_{\mathcal{U}})f_{n}(\beta)$.
    \end{enumerate}
    
    When $\Omega$ is open, we say that an affine fully (resp. stably) matricial $\mathcal{U}$-valued function $f=(f_{n})_{n\in\mathbb{N}}$ on $\Omega$ is \textit{analytic} if $f_{n}$ is analytic for each $n\in\mathbb{N}$. We denote by $A(\Omega)$ the set of all affine fully matricial analytic $\mathbb{C}$-valued functions on $\Omega$.
    
    Also, we say that $f=(f_{n_{1};\dots;n_{k}})_{n_{1},\dots,n_{k}\in\mathbb{N}}$ is a \textit{$k$-variable affine fully (resp. stably) matricial $\mathcal{U}$-valued function} on $\Omega$ if $f$ satisfies the following conditions:
    \begin{enumerate}
        \item $f_{n_{1};\dots;n_{k}}$ is a map from $\Omega_{n_{1}}\times\cdots\times\Omega_{n_{k}}$ to $M_{n_{1}}(\mathbb{C})\otimes\cdots\otimes M_{n_{k}}(\mathbb{C})\otimes\mathcal{U}$ for each $n_{1},\dots,n_{k}\in\mathbb{N}$.
        \item $\left(f_{n_{1};\dots;n_{j-1};n;n_{j+1};\dots;n_{k}}(\beta^{(1)};\dots;\beta^{(j-1)};(\cdot);\beta^{(j+1)};\dots;\beta^{(k)})\right)_{n\in\mathbb{N}}$ is an affine fully (resp. stably) matricial 
        \begin{align*}
            M_{n_{1}}(\mathbb{C})\otimes\cdots\otimes M_{n_{j-1}}(\mathbb{C})\otimes M_{n_{j+1}}(\mathbb{C})\otimes\cdots\otimes M_{n_{k}}(\mathbb{C})\otimes\mathcal{U}
        \end{align*}
        -valued function on $\Omega$ for each $n_{i}\in\mathbb{N}$, $1\leq j\leq k$, $i=1,2,\dots,j-1,j+1,\dots,k$ and $\beta^{(i)}\in\Omega_{n_{i}}$, $i=1,2,\dots,j-1,j+1,\dots,k$.
    \end{enumerate}
    
    When $\Omega$ is open, we say that a $k$-variable affine fully (resp. stably) matricial $\mathcal{U}$-valued function $f=(f_{n_{1};\dots;n_{k}})_{n_{1},\dots,n_{k}\in\mathbb{N}}$ is \textit{analytic} if $f_{n_{1};\dots;n_{k}}$ is analytic for each $n_{1},\dots,n_{k}\in\mathbb{N}$. We denote by $A(\Omega;\dots;\Omega)$ the set of all $k$-variable affine fully (resp. stably) matricial $\mathbb{C}$-valued functions on $\Omega$.
\end{Def}

\begin{Def}(\cite[subsection 3.4 and Definition 11.1]{v10})
    Let $\Omega$ be a fully (resp. stably) matricial $B$-set of $Gr(B)$ and $\mathcal{U}$ be a Banach space over $\mathbb{C}$. Then, we say that $f=(f_{n})_{n\in\mathbb{N}}$ is a \textit{fully (resp. stably) matricial $\mathcal{U}$-valued function} on $\Omega$ if $f$ satisfies the following conditions:
    \begin{enumerate}
        \item $f_{n}$ is a map from $\Omega_{n}$ to $M_{n}(\mathcal{U})=M_{n}(\mathbb{C})\otimes\mathcal{U}$ for each $n\in\mathbb{N}$.
        \item $f_{n_{1}+n_{2}}(\pi_{1}\oplus\pi_{2})=f_{n_{1}}(\pi_{1})\oplus f_{n_{2}}(\pi_{2})$ for any $n_{i}\in\mathbb{N}$ and $\pi_{i}\in \Omega_{n_{i}}$, $i=1,2$.
        \item If $s\cdot\pi=\pi'$ for some $\pi,\pi'\in\Omega_{n}$ and $s\in GL_{n}(\mathbb{C})$, then we have $f_{n}(\pi')=(\mathrm{Ad}(s)\otimes\mathrm{id}_{\mathcal{U}})f_{n}(\pi)$.
    \end{enumerate}
    
    When $\Omega$ is open, we say that a fully (resp. stably) matricial $\mathcal{U}$-valued function $f=(f_{n})_{n\in\mathbb{N}}$ on $\Omega$ is \textit{analytic} if $f_{n}$ is analytic for each $n\in\mathbb{N}$. We also denote by $A(\Omega)$ the set of all fully matricial analytic $\mathbb{C}$-valued functions on $\Omega$.
    
    Also, we say that $f=(f_{n_{1};\dots;n_{k}})_{n_{1},\dots,n_{k}\in\mathbb{N}}$ is a \textit{$k$-variable fully (resp. stably) matricial $\mathcal{U}$-valued function} on $\Omega$ if $f$ satisfies the following conditions:
    \begin{enumerate}
        \item $f_{n_{1};\dots;n_{k}}$ is a map from $\Omega_{n_{1}}\times\cdots\times\Omega_{n_{k}}$ to $M_{n_{1}}(\mathbb{C})\otimes\cdots\otimes M_{n_{k}}(\mathbb{C})\otimes\mathcal{U}$ for each $n_{1},\dots,n_{k}\in\mathbb{N}$.
        \item $\left(f_{n_{1};\dots;n_{j-1};n;n_{j+1};\dots;n_{k}}(\pi_{1};\dots;\pi_{j-1};(\cdot);\pi_{j+1};\dots;\pi_{k})\right)_{n\in\mathbb{N}}$ is a fully (resp. stably) matricial 
        \begin{align*}
            M_{n_{1}}(\mathbb{C})\otimes\cdots\otimes M_{n_{j-1}}(\mathbb{C})\otimes M_{n_{j+1}}(\mathbb{C})\otimes\cdots\otimes M_{n_{k}}(\mathbb{C})\otimes\mathcal{U}
        \end{align*}
        -valued function on $\Omega$ for each $n_{i}\in\mathbb{N}$, $1\leq j\leq k$, $i=1,2,\dots,j-1,j+1,\dots,k$ and $\pi_{i}\in\Omega_{n_{i}}$, $i=1,2,\dots,j-1,j+1,\dots,k$.
    \end{enumerate}
    
    When $\Omega$ is open, we say that a $k$-variable fully (resp. stably) matricial $\mathcal{U}$-valued function $f=(f_{n_{1};\dots;n_{k}})_{n_{1},\dots,n_{k}\in\mathbb{N}}$ is \textit{analytic} if $f_{n_{1};\dots;n_{k}}$ is analytic for each $n_{1},\dots,n_{k}\in\mathbb{N}$. We denote by $A(\Omega;\dots;\Omega)$ the set of all $k$-variable fully (resp. stably) matricial $\mathbb{C}$-valued functions on $\Omega$.
\end{Def}

The spaces $A(\Omega)$ and $A(\Omega;\Omega)$ are unital algebras by natural operations. Also, $A(\Omega;\Omega)$ has $A(\Omega)$-bimodule structure as follows.
\begin{align*}
    (f\cdot g)_{n_{1};n_{2}}(*_{1};*_{2})=(f_{n_{1}}(*_{1})\otimes I_{n_{2}})g_{n_{1};n_{2}}(*_{1};*_{2})
\end{align*}
and
\begin{align*}
    (g\cdot f)_{n_{1};n_{2}}(*_{1};*_{2})=g_{n_{1};n_{2}}(*_{1};*_{2})(I_{n_{1}}\otimes f_{n_{2}}(*_{2}))
\end{align*}
for any $f\in A(\Omega)$, $g\in A(\Omega;\Omega)$ and $*_{i}\in\Omega_{n_{i}}$, $i=1,2$, where $\Omega$ is an open affine fully matricial $B$-set (resp. open fully matricial $B$-set of $Gr(B)$).

Voiculescu constructed a derivation-comultiplication $\partial$ (resp. $\widetilde{\partial}$) for fully matricial analytic $\mathbb{C}$-valued functions
\begin{align*}
    \partial\mbox{\,(resp. $\widetilde{\partial}$)\,}:A(\Omega)\to A(\Omega;\Omega)
\end{align*}
with respect to $A(\Omega)$-bimodule structure, where $\Omega$ is an open affine fully matricial $B$-set (resp. open fully matricial $B$-set of $Gr(B)$) (see \cite[section 7]{v04} and \cite[section 5]{v10} for details). 

The operators $\partial\otimes\mathrm{id}$ and $\mathrm{id}\otimes\partial$ (resp. $\widetilde{\partial}\otimes\mathrm{id}$ and $\mathrm{id}\otimes\widetilde{\partial}$) from $A(\Omega;\Omega)$ to $A(\Omega;\Omega;\Omega)$ were also constructed and satisfies the coassociativity (see \cite[section 7]{v04} and \cite[section 5]{v10} for details):
\begin{align*}
    (\partial\otimes\mathrm{id})\circ\partial=(\mathrm{id}\otimes\partial)\circ\partial
    \quad
    \left(
    \mbox{resp. }
    (\widetilde{\partial}\otimes\mathrm{id})\circ\widetilde{\partial}=(\mathrm{id}\otimes\widetilde{\partial})\circ\widetilde{\partial}
    \right).
\end{align*}

\begin{Rem}\label{remstablypartial}
    When Voiculescu constructed the above operators, the similarity preserving property is important. However, every affine stably matricial $B$-set (resp. stably matricial $B$-set of $Gr(B)$) $\Omega$ has the smallest fully matricial extension $\widetilde{\Omega}=(\widetilde{\Omega}_{n})_{n\in\mathbb{N}}$ with $\widetilde{\Omega}_{n}=\bigcup_{s\in GL_{n}(\mathbb{C})}(s\otimes1)\Omega_{n}(s\otimes1)^{-1}$. Then, every stably matricial function $f=(f_{n})_{n\in\mathbb{N}}$ on $\Omega$ has a unique fully matricial extension $\widetilde{f}=(\widetilde{f}_{n})_{n\in\mathbb{N}}$ to $\widetilde{\Omega}$ (see \cite[Proposition 11.1 and Corollary 11.1]{v10}). Thus, the fully matricial extensions allow us to define the above operators for stably matricial functions.
\end{Rem}


\section{Cyclic derivative for fully matricial functions}

\subsection{Definition and basic properties}
The bilinear map $\#:(A\otimes A^{\mathrm{op}})\times A\ni((a\otimes c),b)\mapsto (a\otimes c)\#b=abc\in A$ often appears in non-commutative differential calculus in the context of free probability (see \cite{ms19}), where the multiplication of $A\otimes A^{\mathrm{op}}$ is given by $(a_{1}\otimes c_{1})\cdot(a_{2}\otimes c_{2})=a_{1}a_{2}\otimes c_{2}c_{1}$. Firstly, we give a fully matricial analogue of this map.

\begin{Lem}\label{sand}
    Let $\Omega$ be an open fully matricial $B$-set of $Gr(B)$. Choose arbitrary $f\in A(\Omega;\Omega)$ and $g\in A(\Omega)$. We define the collection of functions $f\#g=((f\#g)_{n})_{n\in\mathbb{N}}$ as follows.
    \begin{equation*}
        (f\#g)_{n}(\pi)
        :=f_{n;n}(\pi;\pi)\#g_{n}(\pi)
    \end{equation*}
    for any $n\in\mathbb{N}$ and any $\pi\in\Omega_{n}$, where $(A\otimes C)\#B=ABC$ for $A,\,B,\,C\in M_{n,n}(\mathbb{C})$. Then, $\#$ defines a bilinear map $A(\Omega;\Omega)\times A(\Omega)$ to $A(\Omega)$, and satisfies
    \begin{align*}
        (f\#g)_{n}(\pi)
        &=\sum_{1\leq b,c\leq n}\mathrm{Tr}_{n}\left(e^{(n)}_{b,c}(f_{n;n}(\pi;\pi)\#g_{n}(\pi))\right)e^{(n)}_{c,b}\\
        &=\sum_{1\leq b,c\leq n}\mathrm{Tr}_{n}\left((\sigma(f_{n;n}(\pi;\pi))\#e^{(n)}_{b,c})g_{n}(\pi)\right)e^{(n)}_{c,b},
    \end{align*}
    where $\sigma:M_{m}(\mathbb{C})\otimes M_{n}(\mathbb{C})\ni A\otimes B\mapsto B\otimes A\in M_{n}(\mathbb{C})\otimes M_{m}(\mathbb{C})$ and $e^{(n)}_{c,b}$ are matrix units of $M_{n}(\mathbb{C})$.
\end{Lem}
\begin{proof}
    The formulas are easily obtained by direct computation with the tracial property. The analyticity of $f\#g$ is clear. Hence, we have to show the direct sum and similarity preserving properties for $f\#g$. 
    
    \underline{Direct sum preserving property}: Choose arbitrary $\pi_{m}\in \Omega_{m}$ and $\pi_{n}\in\Omega_{n}$. By the definition of multivariable fully matricial functions, we observe that
    \begin{align*}
        \,&f_{m+n,m+n}(\pi_{m}\oplus\pi_{n};\pi_{m}\oplus\pi_{n})\\
        &=\sum_{\substack{1\leq a,b\leq m\\1\leq e,f\leq m}}
        (f_{m,m}(\pi_{m};\pi_{m}))_{(a,b)(e,f)}
        \begin{bmatrix}
            e^{(m)}_{a,b}&\\
            &0
        \end{bmatrix}
        \otimes
        \begin{bmatrix}
            e^{(m)}_{e,f}&\\
            &0
        \end{bmatrix}\\
        &\quad+
        \sum_{\substack{1\leq a,b\leq m\\1\leq g,h\leq n}}
        (f_{m,n}(\pi_{m};\pi_{n}))_{(a,b)(g,h)}
        \begin{bmatrix}
            e^{(m)}_{a,b}&\\
            &0
        \end{bmatrix}
        \otimes
        \begin{bmatrix}
            0&\\
            &e^{(n)}_{g,h}
        \end{bmatrix}\\
        &\quad+
        \sum_{\substack{1\leq c,d\leq n\\1\leq e,f\leq m}}
        (f_{n,m}(\pi_{n};\pi_{m}))_{(c,d)(e,f)}
        \begin{bmatrix}
            0&\\
            &e^{(n)}_{c,d}
        \end{bmatrix}
        \otimes
        \begin{bmatrix}
            e^{(m)}_{e,f}&\\
            &0
        \end{bmatrix}\\
        &\quad+
        \sum_{\substack{1\leq c,d\leq n\\1\leq g,h\leq n}}
        (f_{n,n}(\pi_{n};\pi_{n}))_{(c,d)(g,h)}
        \begin{bmatrix}
            0&\\
            &e^{(n)}_{c,d}
        \end{bmatrix}
        \otimes
        \begin{bmatrix}
            0&\\
            &e^{(n)}_{g,h}
        \end{bmatrix}.
    \end{align*}
    Hence, we have
    \begin{align*}
        \,&(f\#g)_{m+n}(\pi_{m}\oplus\pi_{n})\\
        &=f_{m+n;m+n}(\pi_{m}\oplus\pi_{n})\#g_{m+n}(\pi_{m}\oplus\pi_{n})\\
        &=\sum_{\substack{1\leq a,b\leq m\\1\leq e,f\leq m}}
        (f_{m;m}(\pi_{m};\pi_{m}))_{(a,b)(e,f)}
        \begin{bmatrix}
            e^{(m)}_{a,b}&\\
            &0
        \end{bmatrix}
        \begin{bmatrix}
            g_{m}(\pi_{m})&\\
            &g_{n}(\pi_{n})
        \end{bmatrix}
        \begin{bmatrix}
            e^{(m)}_{e,f}&\\
            &0
        \end{bmatrix}\\
        &\quad+
        \sum_{\substack{1\leq a,b\leq m\\1\leq g,h\leq n}}
        (f_{m;n}(\pi_{m};\pi_{n}))_{(a,b)(g,h)}
        \begin{bmatrix}
            e^{(m)}_{a,b}&\\
            &0
        \end{bmatrix}
        \begin{bmatrix}
            g_{m}(\pi_{m})&\\
            &g_{n}(\pi_{n})
        \end{bmatrix}
        \begin{bmatrix}
            0&\\
            &e^{(n)}_{g,h}
        \end{bmatrix}\\
        &\quad+
        \sum_{\substack{1\leq c,d\leq n\\1\leq e,f\leq m}}
        (f_{n;m}(\pi_{n};\pi_{m}))_{(c,d)(e,f)}
        \begin{bmatrix}
            0&\\
            &e^{(n)}_{c,d}
        \end{bmatrix}
        \begin{bmatrix}
            g_{m}(\pi_{m})&\\
            &g_{n}(\pi_{n})
        \end{bmatrix}
        \begin{bmatrix}
            e^{(m)}_{e,f}&\\
            &0
        \end{bmatrix}\\
        &\quad+
        \sum_{\substack{1\leq c,d\leq n\\1\leq g,h\leq n}}
        (f_{n;n}(\pi_{n};\pi_{n}))_{(c,d)(g,h)}
        \begin{bmatrix}
            0&\\
            &e^{(n)}_{c,d}
        \end{bmatrix}
        \begin{bmatrix}
            g_{m}(\pi_{m})&\\
            &g_{n}(\pi_{n})
        \end{bmatrix}
        \begin{bmatrix}
            0&\\
            &e^{(n)}_{g,h}
        \end{bmatrix}\\
        &=\sum_{\substack{1\leq a,b\leq m\\1\leq e,f\leq m}}
        (f_{m;m}(\pi_{m};\pi_{m}))_{(a,b)(e,f)}
        \begin{bmatrix}
            e^{(m)}_{a,b}g_{m}(\pi_{m})e^{(m)}_{e,f}&\\
            &0
        \end{bmatrix}\\
        &\quad+\sum_{\substack{1\leq c,d\leq n\\1\leq g,h\leq n}}
        (f_{n;n}(\pi_{n};\pi_{n}))_{(c,d)(e,f)}
        \begin{bmatrix}
            0&\\
            &e^{(n)}_{c,d}g_{n}(\pi_{n})e^{(n)}_{e,f}
        \end{bmatrix}\\
        &=f_{m;m}(\pi_{m};\pi_{m})\#g_{m}(\pi_{m})\oplus f_{n;n}(\pi_{n};\pi_{n})\#g_{n}(\pi_{n})\\
        &=(f\#g)_{m}(\pi_{m})\oplus(f\#g)_{n}(\pi_{n}).
    \end{align*}
    Thus, $f\#g$ preserves direct sums. 
    
    \underline{Similarity preserving property}: 
    Choose arbitrary $\pi\in\Omega_{n}$ and $s\in GL_{n}(\mathbb{C})$. We have
    \begin{align*}
        (f\#g)_{n}(s\cdot\pi)
        &=f_{n;n}(s\cdot\pi;s\cdot\pi)\#g_{n}(s\cdot\pi)\\
        &=(\mathrm{Ad}(s)\otimes \mathrm{Ad}(s))[f_{n;n}(\pi;\pi)]\#\mathrm{Ad}(s)[g_{n}(\pi)]\\
        &=s(f_{n;n}(\pi;\pi)\#g_{n}(\pi))s^{-1}\\
        &=\mathrm{Ad}(s)[(f\#g)_{n}(\pi)].
    \end{align*}
    Thus, $f\#g$ preserves similarity.
\end{proof}
\begin{Cor}
    We have
    \begin{equation*}
        (f_{1}\otimes f_{2})\#g=f_{1}gf_{2}
    \end{equation*}
    for any $f_{1},f_{2},g\in A(\Omega)$, where $(f_{1}\otimes f_{2})_{n_{1};n_{2}}(\pi_{1};\pi_{2})=f_{1,n_{1}}(\pi_{1})\otimes f_{2,n_{2}}(\pi_{2})$ for any $f_{1},f_{2}\in A(\Omega)$ and $\pi_{1}\in\Omega_{n_{1}}$, $\pi_{2}\in\Omega_{n_{2}}$, $n_{1},n_{2}\in\mathbb{N}$.
\end{Cor}

\begin{Def}
    Let $\Omega$ be an open fully matricial $B$-set of $Gr(B)$ and $f$ be an analytic fully matricial function on $\Omega$. We define the collection of functions $\widetilde{D}f=((\widetilde{D}f)_{n})_{n\in\mathbb{N}}$ as follows.
    \begin{equation*}
        (\widetilde{D}f)_{n}(\pi)=(m\circ\sigma)\left((\eta(\widetilde{\partial}f)_{n}(\pi)\right)
    \end{equation*}
    for each $n\in\mathbb{N}$ and $\pi\in\Omega_{n}$, where $m:M_{n}(\mathbb{C})\otimes M_{n}(\mathbb{C})\ni A\otimes B\mapsto AB\in M_{n}(\mathbb{C})$, $\eta:A(\Omega;\Omega)\to A(\Omega)$ given by $(\eta(g))_{n}(\pi)=g_{n;n}(\pi;\pi)$ for any $g\in A(\Omega;\Omega)$ and $\pi\in\Omega_{n}$. We call $\widetilde{D}$ the \textit{fully matricial cyclic derivative}.
\end{Def}

\begin{Cor}\label{cyclic}
    The fully matricial cyclic derivative $\widetilde{D}$ defines a linear map from $A(\Omega)$ to $A(\Omega)$ and satisfies 
    \begin{align*}
        (\widetilde{D}f)_{n}(\pi)
        =
        \sum_{1\leq b,c\leq n}
        \mathrm{Tr}_{n}\left((\widetilde{\partial}f)_{n;n}(\pi;\pi)\#e^{(n)}_{b,c}\right)e^{(n)}_{c,b}
        =
        \left(((\sigma\circ\widetilde{\partial})f)\#1_{A(\Omega)}\right)_{n}(\pi)
    \end{align*}
    for any $f\in A(\Omega)$ and $\pi\in \Omega_{n}$.
\end{Cor}
\begin{proof}
    This follows from Lemma \ref{sand}.
\end{proof}

The cyclic derivatives $\delta_{i}:=\delta_{X_{i}:\mathbb{C}\langle \check{i}\rangle}$, $i=1,\dots,n$ for $\mathbb{C}\langle X_{1},\dots,X_{n}\rangle$ are not derivations, where $\mathbb{C}\langle\check{i}\rangle:=\mathbb{C}\langle X_{1},\dots,X_{i-1},X_{i+1},\dots,X_{n}\rangle$. However, they satisfy the following relation (see \cite[comment after Definition 3.3]{ms19}): 
\begin{equation*}
    \delta_{i}[a_{1}a_{2}]=(\sigma\circ\partial_{i})[a_{1}]\# a_{2}+(\sigma\circ\partial_{i})[a_{2}]\# a_{1}
\end{equation*}
for any $a_{1},a_{2}\in \mathbb{C}\langle X_{1},\dots,X_{n}\rangle$, where $\partial_{i}:=\partial_{X_{i}:\mathbb{C}\langle \check{i}\rangle}$ are the free difference quotients for $\mathbb{C}\langle X_{1},\dots,X_{n}\rangle$. We have a fully matricial analogue of this formula as follows.
\begin{Cor}
    We have
    \begin{equation*}
        \widetilde{D}[fg]
        =(\sigma\circ\widetilde{\partial})[f]\#g+(\sigma\circ\widetilde{\partial})[g]\#f
    \end{equation*}
    for any $f,g\in A(\Omega)$.
\end{Cor}
\begin{proof}
    This follows from Corollary \ref{cyclic} and the fact that $\widetilde{\partial}$ is a derivation.
\end{proof}

Voiculescu \cite[section 8]{v10} introduced the grading operator $\Lambda$ on $A(\Omega)$ as follows.
\begin{equation*}
    (\Lambda f)_{n}(\pi)
    :=\frac{d}{dt}
    \left.
    e^t
    f_{n}
    \left(
    C
    \left(
    \begin{bmatrix}
        1&0\\
        0&e^t
    \end{bmatrix}
    \right)
    \pi
    \right)
    \right|_{t=0}
\end{equation*}
for any $f\in A(\Omega)$ and any $\pi\in \Omega_{n}$, where 
$C
\left(
\left[
\begin{smallmatrix}
    1&0\\
    0&e^t
\end{smallmatrix}
\right]
\right)\pi=
\left[
\begin{smallmatrix}
    1^{\oplus n}&0\\
    0&(e^t)^{\oplus n}
\end{smallmatrix}
\right]\pi$ if $\pi\in\Omega_{n}$. Also, we have
\begin{equation*}
    ((\Lambda-\mathrm{id})f)_{n}(\pi)
    =
    \frac{d}{dt}
    \left.
    f_{n}
    \left(
    C
    \left(
    \begin{bmatrix}
        1&0\\
        0&e^t
    \end{bmatrix}
    \right)
    \pi
    \right)
    \right|_{t=0}.
\end{equation*}
Hence, $\Lambda-\mathrm{id}$ is a derivation from $A(\Omega)$ to $A(\Omega)$ and $\Lambda$ a coderivation from $A(\Omega)$ to $A(\Omega)$, i.e., $\widetilde{\partial}\circ\Lambda=(\Lambda\otimes\mathrm{id}+\mathrm{id}\otimes\Lambda)\circ\widetilde{\partial}$. Let us define linear maps $\Lambda\otimes\mathrm{id}$ and $\mathrm{id}\otimes\Lambda$ from $A(\Omega;\Omega)$ to $A(\Omega;\Omega)$ as follows.
\begin{equation*}
    ((\Lambda\otimes\mathrm{id})f)_{m;n}(\pi_{m};\pi_{n})
    :=\frac{d}{dt}
    \left.
    e^t
    f_{m,n}
    \left(
    C
    \left(
    \begin{bmatrix}
        1&0\\
        0&e^t
    \end{bmatrix}
    \right)
    \pi_{m}
    ;
    \pi_{n}
    \right)
    \right|_{t=0},
\end{equation*}
\begin{equation*}
    ((\mathrm{id}\otimes\Lambda)f)_{m;n}(\pi_{m};\pi_{n})
    :=\frac{d}{dt}
    \left.
    e^t
    f_{m,n}
    \left(
    \pi_{m}
    ;
    C
    \left(
    \begin{bmatrix}
        1&0\\
        0&e^t
    \end{bmatrix}
    \right)
    \pi_{n}
    \right)
    \right|_{t=0}
\end{equation*}
for any $\pi_{i}\in\Omega_{i}$ and any $f\in A(\Omega;\Omega)$. The maps $(\Lambda-\mathrm{id})\otimes\mathrm{id}$ and $\mathrm{id}\otimes(\Lambda-\mathrm{id})$ are also defined in the same way. We have the following relation between $\widetilde{D}$, $\Lambda$ and $\Lambda-\mathrm{id}$:

\begin{Lem}
    We have $\widetilde{D}\circ(\Lambda-\mathrm{id})=\Lambda\circ\widetilde{D}$.
\end{Lem}
\begin{proof}
    Choose an arbitrary $f\in A(\Omega)$. Recall that
    \begin{align*}
        \widetilde{\partial}\circ(\Lambda-\mathrm{id})=
            ((\Lambda-\mathrm{id})\otimes\mathrm{id}+\mathrm{id}\otimes(\Lambda-\mathrm{id})+\mathrm{id}_{A^(\Omega;\Omega)})\circ\widetilde{\partial}
    \end{align*}
    holds (see \cite[section 8]{v10}). 
    For any $n\in\mathbb{N}$ and $\pi\in\Omega_{n}$, we observe that
    \begin{align*}
        \,&\left((\widetilde{D}\circ(\Lambda-\mathrm{id}))f\right)_{n}(\pi)\\
        &=
        (m\circ\sigma)
        \left[
        \eta
        \left(
        (
        \widetilde{\partial}\circ(\Lambda-\mathrm{id})
        )
        f
        \right)_{n}(\pi)
        \right]\\
        &=
        (m\circ\sigma)
        \left[
        \eta
        \left(
        (
        ((\Lambda-\mathrm{id})\otimes\mathrm{id}+\mathrm{id}\otimes(\Lambda-\mathrm{id})+\mathrm{id}_{A(\Omega;\Omega)})\circ\widetilde{\partial}
        )
        f
        \right)_{n}(\pi)
        \right]\\
        &=
        (m\circ\sigma)
        \left[
        \left(((\Lambda-\mathrm{id})\otimes\mathrm{id})\widetilde{\partial}f\right)_{n;n}(\pi;\pi)
        \right]
        +
        (m\circ\sigma)
        \left[
        \left((\mathrm{id}\otimes(\Lambda-\mathrm{id}))\widetilde{\partial}f\right)_{n;n}(\pi;\pi)
        \right]\\
        &\quad+(m\circ\sigma)
        \left[
        (
        \widetilde{\partial}f
        )_{n;n}(\pi;\pi)
        \right]\\
        &=
        (m\circ\sigma)
        \left[
        \left.
        \frac{d}{dt}
        (\widetilde{\partial}f)_{n;n}
        \left(
        C
        \left(
        \begin{bmatrix}
            1&0\\
            0&e^t
        \end{bmatrix}
        \right)\pi;\pi
        \right)
        \right|_{t=0}
        \right]\\
        &\quad+
        (m\circ\sigma)
        \left[
        \left.
        \frac{d}{dt}
        (\widetilde{\partial}f)_{n;n}
        \left(
        \pi
        ;
        C
        \left(
        \begin{bmatrix}
            1&0\\
            0&e^t
        \end{bmatrix}
        \right)\pi
        \right)
        \right|_{t=0}
        \right]
        +(\widetilde{D}f)_{n}(\pi).
    \end{align*}
    
    On the other hand, we have
    \begin{align*}
        \,&\left((\widetilde{\Lambda}\circ\widetilde{D})f\right)_{n}(\pi)\\
        &=
        \left.
        \frac{d}{dt}
        e^t
        (\widetilde{D}f)_{n}
        \left(
        C
        \left(
        \begin{bmatrix}
            1&0\\
            0&e^t
        \end{bmatrix}
        \right)\pi
        \right)
        \right|_{t=0}\\
        &=
        (\widetilde{D}f)_{n}(\pi)
        +
        \left.
        \frac{d}{dt}
        (\widetilde{D}f)_{n}
        \left(
        C
        \left(
        \begin{bmatrix}
            1&0\\
            0&e^t
        \end{bmatrix}
        \right)\pi
        \right)
        \right|_{t=0}\\
        &=
        (\widetilde{D}f)_{n}(\pi)
        +
        \left.
        \frac{d}{dt}
        (m\circ\epsilon)
        \left[
        (\widetilde{\partial}f)_{n;n}
        \left(
        C
        \left(
        \begin{bmatrix}
            1&0\\
            0&e^t
        \end{bmatrix}
        \right)\pi
        ;
         C
        \left(
        \begin{bmatrix}
            1&0\\
            0&e^t
        \end{bmatrix}
        \right)\pi
        \right)
        \right]
        \right|_{t=0}\\
        &=
        (\widetilde{D}f)_{n}(\pi)
        +
        (m\circ\epsilon)
        \left[
         \left.
        \frac{d}{dt}
        (\widetilde{\partial}f)_{n;n}
        \left(
        C
        \left(
        \begin{bmatrix}
            1&0\\
            0&e^t
        \end{bmatrix}
        \right)\pi
        ;
         C
        \left(
        \begin{bmatrix}
            1&0\\
            0&e^t
        \end{bmatrix}
        \right)\pi
        \right)
        \right|_{t=0}
        \right]\\
        \intertext{\hspace{4cm}(by the chain rule for two-variable functions on Banach spaces)}
        &=
        (\widetilde{D}f)_{n}(\pi)\\
        &\quad+
        (m\circ\sigma)
        \Biggl[
         \left.
        \frac{d}{dt}
        (\widetilde{\partial}f)_{n;n}
        \left(
        C
        \left(
        \begin{bmatrix}
            1&0\\
            0&e^t
        \end{bmatrix}
        \right)\pi
        ;
         \pi
        \right)
        \right|_{t=0}\\
        &\quad
        +
        \left.
        \frac{d}{dt}
        (\widetilde{\partial}f)_{n;n}
        \left(
        \pi
        ;
         C
        \left(
        \begin{bmatrix}
            1&0\\
            0&e^t
        \end{bmatrix}
        \right)\pi
        \right)
        \right|_{t=0}
        \Biggr]\\
        &=
        (\widetilde{D}f)_{n}(\pi)\\
        &\quad+
        (m\circ\sigma)
        \left[
        \left.
        \frac{d}{dt}
        (\widetilde{\partial}f)_{n;n}
        \left(
        C
        \left(
        \begin{bmatrix}
            1&0\\
            0&e^t
        \end{bmatrix}
        \right)\pi;\pi
        \right)
        \right|_{t=0}
        \right]\\
        &\quad+
        (m\circ\sigma)
        \left[
        \left.
        \frac{d}{dt}
        (\widetilde{\partial}f)_{n;n}
        \left(
        \pi
        ;
        C
        \left(
        \begin{bmatrix}
            1&0\\
            0&e^t
        \end{bmatrix}
        \right)\pi
        \right)
        \right|_{t=0}
        \right].
    \end{align*}
    Hence, we have obtained the desired equality.
\end{proof}

The above lemma is nothing but a fully matricial analogue of \cite[Lemma 3.5]{ms19}.


\subsection{Duality of cyclic derivatives}
Let us recall some definitions and facts in \cite[sections 4 and 6]{v10}. We fix an element 
$\pi=
\left[
\begin{smallmatrix}
    a&b\\
    c&d
\end{smallmatrix}
\right]/\widetilde{\lm1}\in Gr_{1}(E)$. The \textit{$n$-th Grassmannian $B$-resolvent set} of $\pi$ is 
\begin{equation*}
    \widetilde{\rho}_{n}(\pi;B)
    =\left\{
    \sigma=
    \left[
    \begin{smallmatrix}
        \alpha&\beta\\
        \gamma&\delta
    \end{smallmatrix}
    \right]/\widetilde{\lm n}\in Gr_{n}(B)\,\middle|\,
    \left[
    \begin{smallmatrix}
        b^{\oplus n}&\beta\\
        d^{\oplus n}&\delta
    \end{smallmatrix}
    \right]\in GL_{2}(M_{n}(E))
    \right\},
\end{equation*}
and we call $\widetilde{\rho}(\pi;B)=(\widetilde{\rho}_{n}(\pi;B))_{n\in\mathbb{N}}$ the \textit{full Grassmannian $B$-resolvent set} of $\pi$. This $\widetilde{\rho}(\pi;B)$ is an open fully matricial $B$-set of $Gr(B)$. The \textit{$n$-th Grassmannian $B$-resolvent} of $\pi$ is
\begin{equation*}
    \widetilde{\mathcal{R}}_{n}(\pi;B)(\sigma)=\beta\zeta
\end{equation*}
for $\sigma=
\left[
\begin{smallmatrix}
    \alpha&\beta\\
    \gamma&\delta
\end{smallmatrix}
\right]/\widetilde{\lm n}\in\widetilde{\rho}_{n}(\pi;B)$, where $\zeta\in M_{n}(E)$ satisfies 
$\left[
\begin{smallmatrix}
    b^{\oplus n}&\beta\\
    d^{\oplus n}&\delta
\end{smallmatrix}
\right]^{-1}
=
\left[
\begin{smallmatrix}
    *&*\\
    *&\zeta
\end{smallmatrix}
\right]$, and we call $\widetilde{\mathcal{R}}(\pi;B)=(\widetilde{\mathcal{R}}_{n}(\pi;B))_{n\in\mathbb{N}}$ the \textit{full Grassmannian $B$-resolvent} of $\pi$. This $\widetilde{\mathcal{R}}(\pi;B)$ is an analytic fully matricial $E$-valued function on $\widetilde{\rho}(\pi;B)$. 

Let $\mathcal{CR}(\pi;B)$ denote the set of all matrix coefficients of $\widetilde{\mathcal{R}}_{n}(\pi;B)(\sigma)$, $n\in\mathbb{N}$, $\sigma\in \widetilde{\rho}_{n}(\pi;B)$, and $\mathcal{LR}(\pi;B)$ the linear span of $\mathcal{CR}(\pi;B)$. Remark that $\mathcal{CR}(\pi;B)$ is closed under multiplication since resolvents are corepresentations (group-like elements), and hence $\mathcal{LR}(\pi;B)$ is a subalgebra of $E$ (see \cite[Lemma 6.2]{v10}). We denote the norm closure of $\mathcal{LR}(\pi;B)$ in $E$ by $E_{1}$. 

Then, $E_{1}^d$ denotes the topological dual of $E_{1}$. The duality $\mathcal{U}$-transform is a map
\begin{equation*}
    \mathcal{U}:E_{1}^d\ni\varphi\mapsto\mathcal{U}(\varphi)=(\mathcal{U}(\varphi)_{n}(\cdot))_{n\in\mathbb{N}}\in A(\widetilde{\rho}(\pi;B))
\end{equation*}
such that 
\begin{equation*}
    \mathcal{U}(\varphi)_{n}(\sigma)=(\varphi\otimes\mathrm{id}_{M_{n}(\mathbb{C})})(\widetilde{\mathcal{R}}_{n}(\pi;B)(\sigma))
\end{equation*}
for any $n\in\mathbb{N}$ and any $\sigma\in\widetilde{\rho}_{n}(\pi;B)$.

Assume that there exists a derivation-comultiplication
\begin{equation*}
    \partial_{\pi}:\mathcal{LR}(\pi;B)\to\mathcal{LR}(\pi;B)^{\otimes2}
\end{equation*}
such that 
\begin{align*}
    (\partial_{\pi}\otimes\mathrm{id}_{M_{n}(\mathbb{C})})
    \left(
    \widetilde{\mathcal{R}}_{n}(\pi;B)(\sigma)
    \right)
    =
    \widetilde{\mathcal{R}}_{n}(\pi;B)(\sigma)\otimes_{M_{n}(\mathbb{C})}\widetilde{\mathcal{R}}_{n}(\pi;B)(\sigma),
\end{align*}
or equivalently,
\begin{align*}
    \partial_{\pi}
    \left[
    (\widetilde{\mathcal{R}}_{n}(\pi;B)(\sigma))_{(a,c)}
    \right]
    =\sum_{1\leq b\leq n}
    (\widetilde{\mathcal{R}}_{n}(\pi;B)(\sigma))_{(a,b)}
    \otimes
    (\widetilde{\mathcal{R}}_{n}(\pi;B)(\sigma))_{(b,c)}
\end{align*}
for any $1\leq a,c\leq n$ and $\sigma\in\widetilde{\rho}_{n}(\pi;B)$, $n\in\mathbb{N}$, where $(A\otimes e)\otimes_{M_{n}(\mathbb{C})}(A'\otimes e')=e\otimes e'\otimes AA'$ for any $e,e'\in E$ and $A,A'\in M_{n}(\mathbb{C})$.
(In fact, Voiculescu constructed some examples of derivations which satisfy this assumption \cite[section 12]{v10}.) 
Here, we can also consider the cyclic derivative $\delta_{\pi}=\mu\circ\sigma\circ\partial_{\pi}$ from $\mathcal{LR}(\pi;B)$ to $\mathcal{LR}(\pi;B)$. Voiculescu proved some duality results on $\widetilde{\partial}$ and $\mathcal{U}$. We also have the following duality property of $\widetilde{D}$ and $\mathcal{U}$:
\begin{Prop}
    If $\varphi\circ\delta_{\pi}\in E_{1}^d$ with $\varphi\in E_{1}^d$, then $\mathcal{U}(\varphi\circ\delta_{\pi})=-\widetilde{D}(\mathcal{U}(\varphi))$.
\end{Prop}
\begin{proof}
    Choose an arbitrary $\sigma\in\widetilde{\rho}_{n}(\pi;B)$. Note that
    \begin{align*}
        (\widetilde{\partial}[\mathcal{U}(\varphi)])_{n;n}(\sigma;\sigma)
        &=-(\varphi\otimes\mathrm{id}_{M_{n}(\mathbb{{C}})}\otimes\mathrm{id}_{M_{n}(\mathbb{C})})(\widetilde{\mathcal{R}}_{n}(\pi;B)(\sigma)\otimes_{E}\widetilde{\mathcal{R}}_{n}(\pi;B)(\sigma))\\
        &=-\sum_{\substack{1\leq a,b\leq n\\1\leq c,d\leq n}}\varphi\left((\widetilde{\mathcal{R}}_{n}(\pi;B)(\sigma))_{(a,b)}(\widetilde{\mathcal{R}}_{n}(\pi;B)(\sigma))_{(c,d)}\right)e^{(n)}_{a,b}\otimes e^{(n)}_{c,d},
    \end{align*}
    where $\otimes_{E}$ means that $(\alpha\otimes e)\otimes_{E}(\alpha'\otimes e')=ee'\otimes \alpha\otimes\alpha'$ for any $e,e'\in E$ and $\alpha,\alpha'\in M_{n}(\mathbb{C})$, $n\in\mathbb{N}$
    (see \cite[Proposition 6.3]{v10}).
    We have
    \begin{align*}
        (\widetilde{D}(\mathcal{U}(\varphi)))_{n}(\sigma)
        &=\sum_{1\leq b',c'\leq n}\mathrm{Tr}_{n}
        \left(
        (\widetilde{\partial}[\mathcal{U}(\varphi)]_{n;n})(\sigma;\sigma)\#e^{(n)}_{b',c'}
        \right)e^{(n)}_{c',b'}\\
        &=-\sum_{\substack{1\leq b',c'\leq n\\1\leq a,b\leq n\\1\leq c,d\leq n}}
        \varphi\left((\widetilde{\mathcal{R}}_{n}(\pi;B)(\sigma))_{(a,b)}(\widetilde{\mathcal{R}}_{n}(\pi;B)(\sigma))_{(c,d)}\right)\mathrm{Tr}_{n}
        \left(
        e^{(n)}_{a,b}e^{(n)}_{b',c'}e^{(n)}_{c,d}
        \right)
        e^{(n)}_{c',b'}\\
        &=-\sum_{\substack{1\leq a,b\leq n\\1\leq c,d\leq n}}
        \varphi\left((\widetilde{\mathcal{R}}_{n}(\pi;B)(\sigma))_{(a,b)}(\widetilde{\mathcal{R}}_{n}(\pi;B)(\sigma))_{(c,d)}\right)\mathrm{Tr}_{n}
        \left(
        e^{(n)}_{a,d}
        \right)
        e^{(n)}_{c,b}\\
        &=-\sum_{\substack{1\leq a,b,c\leq n}}
        \varphi\left((\widetilde{\mathcal{R}}_{n}(\pi;B)(\sigma))_{(a,b)}(\widetilde{\mathcal{R}}_{n}(\pi;B)(\sigma))_{(c,a)}\right)
        e^{(n)}_{c,b}.
    \end{align*}
    
    On the other hand, we observe that
    \begin{align*}
        \mathcal{U}(\varphi\circ\delta_{\pi})_{n}(\sigma)
        &=\left((\varphi\circ\delta_{\pi})\otimes\mathrm{id}_{M_{n}(\mathbb{C})}\right)(\widetilde{\mathcal{R}}_{n}(\pi;B)(\sigma))\\
        &=\sum_{1\leq b,c\leq n}\left(\varphi\circ\delta_{\pi}\right)((\widetilde{\mathcal{R}}_{n}(\pi;B)(\sigma))_{(c,b)})e^{(n)}_{c,b}\\
        &=\sum_{1\leq b,c\leq n}\sum_{a=1}^{n}(\varphi\circ\mu\circ\sigma)\left((\widetilde{\mathcal{R}}_{n}(\pi;B)(\sigma))_{(c,a)}\otimes(\widetilde{\mathcal{R}}_{n}(\pi;B)(\sigma))_{(a,b)}\right)e^{(n)}_{c,b}\\
        &=\sum_{1\leq a,b,c\leq n}\varphi\left((\widetilde{\mathcal{R}}_{n}(\pi;B)(\sigma))_{(a,b)}(\widetilde{\mathcal{R}}_{n}(\pi;B)(\sigma))_{(c,a)}\right)e^{(n)}_{c,b}.
    \end{align*}
    Thus, we have obtained $\mathcal{U}(\varphi\circ\delta_{\pi})=-\widetilde{D}(\mathcal{U}(\varphi))$.
\end{proof}


\subsection{Affine case}
In the affine setting (that is, $\Omega_{n}\subset M_{n}(B)$), we can also define a fully matricial cyclic derivative in the same way, and denote it by $D$. The facts in the previous section also hold for $D$. Here, in particular, consider the case when $\Omega=M(B)$. Recall that $M(B)$ is clearly an open affine fully matricial $B$-set. 

Note that $A(M(B))$ also has a special subalgebra, which was found and named the \textit{polynomial sub-bialgebra} $\mathcal{Z}(B^d)$ by Voiculescu. Here, $B^d$ is the topological dual of $B$ (see \cite[section 7]{v10}). We can regard $\varphi\in B^d$ as an element $z(\varphi)$ of $A(M(B))$ as follows.
\begin{equation*}
    z(\varphi)_{n}
    =
    \varphi\otimes\mathrm{id}_{M_{n}}.
\end{equation*}

Remark that $\mathcal{Z}(B^d)$ coincides with the unital subalgebra of $A(M(B))$ generated by $\{1_{A(M(B))}\}$ and $\{z(\varphi)\,|\,\varphi\in B^d\}$. Here, we have
\begin{equation*}
    \partial[z(\varphi)]=\varphi(1_{B})1_{A(M(B))}\otimes1_{A(M(B))}
\end{equation*}
for any $\varphi\in B^d$. It follows that the restriction of $\partial$ to $\mathcal{Z}(B^d)$ defines a derivation-comultiplication from $\mathcal{Z}(B^d)$ to $\mathcal{Z}(B^d)^{\otimes2}$. Hence, $(\mathcal{Z}(B^d),\partial|_{\mathcal{Z}(B^d)})$ is a GDQ ring. We can also consider the cyclic derivative $\delta_{\mathcal{Z}(B^d)}=\mu\circ\sigma\circ\partial|_{\mathcal{Z}(B^d)}$ from $\mathcal{Z}(B^d)$ to $\mathcal{Z}(B^d)$. Here, we observe the following lemma:
\begin{Lem}
    We have $D|_{\mathcal{Z}(B^d)}=\delta_{\mathcal{Z}(B^d)}$.
\end{Lem}
\begin{proof}
    It suffices to confirm the desired identity for each monomial $z(\varphi_{1})\cdots z(\varphi_{k})$. We have
    \begin{align*}
        \delta_{\mathcal{Z}(B^d)}[z(\varphi_{1})\cdots z(\varphi_{k})]
        &=(\mu\circ\sigma)\left(\partial[z(\varphi_{1})\cdots z(\varphi_{k})]\right)\\
        &=\sum_{i=1}^{k}\varphi_{i}(1_{B})(\mu\circ\sigma)[z(\varphi_{1})\cdots z(\varphi_{i-1})\otimes z(\varphi_{i+1})\cdots z(\varphi_{k})]\\
        &=\sum_{i=1}^{k}\varphi_{i}(1_{B})z(\varphi_{i+1})\cdots z(\varphi_{k})z(\varphi_{1})\cdots z(\varphi_{i-1}).
    \end{align*}
    
    On the other hand, we have
    \begin{align*}
        D[z(\varphi_{1})\cdots z(\varphi_{k})]_{n}(\beta)
        &=\sum_{1\leq b,c\leq n}\mathrm{Tr}_{n}
        \left(
        \partial[z(\varphi_{1})\cdots z(\varphi_{k})]_{n;n}(\beta;\beta)\#e^{(n)}_{b,c}
        \right)
        e^{(n)}_{c,b}\\
        &=\sum_{\substack{1\leq b,c\leq n\\1\leq i\leq k}}
        \varphi_{i}(1_{B})\mathrm{Tr}_{n}
        \left(
        (z(\varphi_{1})\cdots z(\varphi_{i-1}))_{n}(\beta)e^{(n)}_{b,c}(z(\varphi_{i+1})\cdots z(\varphi_{k}))_{n}(\beta)
        \right)e^{(n)}_{c,b}\\
        &=\sum_{\substack{1\leq a,b,c\leq n\\1\leq i\leq k}}
        \varphi_{i}(1_{B})
        (z(\varphi_{1})\cdots z(\varphi_{i-1}))_{n}(\beta)_{(a,b)}(z(\varphi_{i+1})\cdots z(\varphi_{k}))_{n}(\beta)_{(c,a)}
        e^{(n)}_{c,b}\\
        &=\sum_{\substack{1\leq b,c\leq n\\1\leq i\leq k}}
        \varphi_{i}(1_{B})
        \Bigl(
        (z(\varphi_{i+1})\cdots z(\varphi_{k}))_{n}(\beta)(z(\varphi_{1})\cdots z(\varphi_{i-1}))_{n}(\beta)
        \Bigr)_{(c,b)}e^{(n)}_{c,b}\\
        &=\left(
        \sum_{1\leq i\leq k}\varphi_{i}(1_{B})z(\varphi_{i+1})\cdots z(\varphi_{k})z(\varphi_{1})\cdots z(\varphi_{i-1})
        \right)_{n}(\beta)
    \end{align*}
    for any $\beta\in M_{n}(B)$. Thus, we have confirmed the desired $D|_{\mathcal{Z}(B^d)}=\delta_{\mathcal{Z}(B^d)}$.
\end{proof}

\section{The continuity of operators $\widetilde{\partial}$, $\widetilde{D}$ and $\#$}
In this section, we will show the continuity of derivation-comultiplication $\widetilde{\partial}$, fully matricial cyclic derivative $\widetilde{D}$ and operator $\#$ with respect to the uniform convergence. At first, we prepare some norms for fully (stably) matricial functions:

\begin{Def}
    Let $\Omega$ be an affine stably matricial $B$-set (resp. stably matricial $B$-set of $Gr(B)$) and $f$ be a one-variable stably matricial function. Then, we define the uniform norm $\|f\|_{n,\Omega}$ for each $n\in\mathbb{N}$ as follows.
    \begin{align*}
        \|f\|_{n,\Omega}=\sup\{\|f_{n}(*)\|\,|\,*\in\Omega_{n}\}.
    \end{align*}
    
    More generally, let $f$ be a $k$-variable stably matricial function. Then, we define the uniform norm $\|f\|_{n_{1};\dots;n_{k},\Omega}$ for each $n_{1},\dots,n_{k}\in\mathbb{N}$ as follows.
    \begin{align*}
        \|f\|_{n_{1};\dots;n_{k},\Omega}
        =
        \sup\{\|f_{n_{1};\dots;n_{k},R}(*_{1};\dots;*_{k})\|\,|\,*_{1}\in\Omega_{n_{1}},\dots,*_{k}\in\Omega_{n_{k}}\}.
    \end{align*}
    In particular, we write $\|f\|_{n,R}=\|f\|_{n,R\mathcal{D}_{0}(B)}$ and $\|f\|_{n_{1};\dots;n_{k},R}=\|f\|_{n_{1};\dots;n_{k},R\mathcal{D}_{0}(B)}$ for any $R>0$.
\end{Def}

\begin{Lem}
    Let $\Omega$ be an open fully matricial $B$-set of $Gr(B)$ and $f\in A(\Omega)$. Then, we have 
    \begin{align*}
        \,&\|(\widetilde{\nabla}f)_{n_{1};n_{2}}(\pi_{1};\pi_{2})(e^{(n_{1},n_{2})}_{b,c})\|\leq
        \|f\|_{n_{1}+n_{2},\Omega}
    \end{align*}
    for any $n_{1},n_{2}\in\mathbb{N}$ and any $\pi_{1}=\beta^{(1)}/\widetilde{\lambda n_{1}}\in\Omega_{n_{1}}$, $\pi_{2}=\beta^{(2)}/\widetilde{\lambda n_{2}}\in\Omega_{n_{2}}$ (see \cite[section 7]{v04} and \cite[section 5]{v10} for details on $\widetilde{\nabla}$).
\end{Lem}
\begin{proof}
    This follows from the following: By definition, we have
    \begin{align*}
        \,&(\widetilde{\nabla}f)_{n_{1};n_{2}}(\pi_{1};\pi_{2})(e^{(n_{1},n_{2})}_{b,c})\\
        &=
        \begin{bmatrix}
            I_{n_{1}}&0_{n_{1},n_{2}}
        \end{bmatrix}
        f_{n_{1}+n_{2}}
        \left(
        \begin{bmatrix}
            \beta^{(1)}_{(1,1)}&0&\beta^{(1)}_{(1,2)}&0\\
            0&\beta^{(2)}_{(1,1)}&0&\beta^{(2)}_{(1,2)}\\
            \beta^{(1)}_{(2,1)}&0&\beta^{(1)}_{(2,2)}&(e^{(n_{1},n_{2})}_{b,c}\otimes1)\beta^{(2)}_{(1,2)}\\
            0&\beta^{(2)}_{(2,1)}&0&\beta^{(2)}_{(2,2)}
        \end{bmatrix}
        /\widetilde{\lambda n_{1}+n_{2}}
        \right)
        \begin{bmatrix}
            0_{n_{1},n_{2}}\\
            I_{n_{2}}
        \end{bmatrix}
    \end{align*}
    for any $n_{1},n_{2}\in\mathbb{N}$, any $\pi_{1}=\beta^{(1)}/\widetilde{\lambda n_{1}}\in\Omega_{n_{1}}$, $\pi_{2}=\beta^{(2)}/\widetilde{\lambda n_{2}}\in\Omega_{n_{2}}$ and any $1\leq b\leq n_{1}$, $1\leq c\leq n_{2}$.
\end{proof}
Using this lemma, we can easily show the following propositions: 

\begin{Prop}
    Let $\Omega$ be an open fully matricial $B$-set of $Gr(B)$ and $f\in A(\Omega)$. Then, we have
    \begin{align*}
        \,&\|(\widetilde{\partial}f)_{n_{1};n_{2}}(\pi_{1};\pi_{2})\|\leq
        2(n_{1}+n_{2})\|f\|_{n_{1}+n_{2},\Omega}
    \end{align*}
    for any $\pi_{1}=\beta^{(1)}/\widetilde{\lambda n_{1}}\in\Omega_{n_{1}}$ and $\pi_{2}=\beta^{(2)}/\widetilde{\lambda n_{2}}\in\Omega_{n_{2}}$. In particular, if $f_{n}$ is bounded, then so are $(\widetilde{\partial}f)_{n_{1};n_{2}}$ for any $n_{1},n_{2}\in\mathbb{N}$ with $n=n_{1}+n_{2}$.
\end{Prop}

\begin{Prop}\label{propcyclicineq}
    Let $\Omega$ be an open fully matricial $B$-set of $Gr(B)$ and $f\in A(\Omega)$. Then, we have
    \begin{align*}
        \,&\|(\widetilde{D}f)_{n}(\pi)\|\leq2n\|f\|_{2n,\Omega}
    \end{align*}
    for any $n\in\mathbb{N}$ and $\pi=\beta/\widetilde{\lambda n}\in\Omega_{n}$. In particular, if $f$ is separately bounded (that is, $\|f\|_{n,\Omega}<\infty$ for each $n\in\mathbb{N}$), then so is $\widetilde{D}f$.
\end{Prop}

Hence, we obtain the following two corollaries:

\begin{Cor}\label{corcontifmfdq}
    Let $\Omega$ be an open fully matricial $B$-set of $Gr(B)$ and $f, f^{(L)}$, $L\in\mathbb{N}$, be elements of $A(\Omega)$. If $\{f^{(L)}_{n}\}_{L\in\mathbb{N}}$ converges to $f_{n}$ in the uniform norm $\|\cdot\|_{n,\Omega}$, then $\{(\widetilde{\partial}f^{(L)})_{n_{1};n_{2}}\}$ also converges to $(\widetilde{\partial}f)_{n_{1};n_{2}}$ in the uniform norm $\|\cdot\|_{n_{1};n_{2},\Omega}$.
\end{Cor}

\begin{Cor}\label{corcontifmcg}
     Let $\Omega$ be an open fully matricial $B$-set of $Gr(B)$ and $f, f^{(L)}$, $L\in\mathbb{N}$, be elements of $A(\Omega)$. If $\{f^{(L)}\}_{L\in\mathbb{N}}$ converges to $f$ in uniform norms $\{\|\cdot\|_{n,\Omega}\}_{n\in\mathbb{N}}$ (that is, $\{f^{(L)}_{n}\}_{L\in\mathbb{N}}$ converges to $f_{n}$ in $\|\cdot\|_{n,\Omega}$ for each $n\in\mathbb{N}$), then $\{\widetilde{D}f^{(L)}\}_{L\in\mathbb{N}}$ also converges to $\widetilde{D}f$ in uniform norms $\{\|\cdot\|_{n,\Omega}\}_{n\in\mathbb{N}}$.
\end{Cor}

It is also easy to see the following lemma with respect to the continuity of operator $\#$:

\begin{Lem}\label{lemcontisharp}
    Let $\Omega$ be an open fully matricial $B$-set of $Gr(B)$ and $f,f^{(L)}$, $L\in\mathbb{N}$, be elements of $A(\Omega;\Omega)$. Let $g$ also be an element of $A(\Omega)$. If $\{f^{(L)}_{n;n}\}_{L\in\mathbb{N}}$ converges to $f_{n;n}$ in the uniform norm $\|\cdot\|_{n;n,\Omega}$ and $\|g\|_{n,\Omega}<\infty$, then $\{(f^{(L)}\#g)_{n}\}_{L\in\mathbb{N}}$ also converges to $(f\#g)_{n}$ in the uniform norm $\|\cdot\|_{n,\Omega}$.
\end{Lem}

\begin{Rem}
    We can also show the above facts for operators $\widetilde{\partial}\otimes\mathrm{id}$, $\mathrm{id}\otimes\widetilde{\partial}$, $\#_{1,2}$ and $\#_{2,3}$. Moreover, the same statements clearly hold in the affine fully matricial setting.
\end{Rem}

\begin{Rem}
    The above facts can also work for stably matricial analytic functions. In fact, for a stably matricial analytic function $f$ on some open affine stably matricial $B$-set $\Omega$, we can approximate $f$ by Voiculescu's series expansion $\{f^{(L)}\}_{L\in\mathbb{N}}$ (see \cite[section 13]{v10} and also Appendix \ref{appendixpolyapproximate}). This series expansion $\{f^{(L)}\}_{L\in\mathbb{N}}$ converges to $f$ uniformly on each compact subset $\Omega$. Here, the series is exactly the restriction of series expansion $\{\widetilde{f}^{(L)}\}_{L\in\mathbb{N}}$ of the fully matricial extension $\widetilde{f}$ of $f$ to the smallest fully matricial extension $\widetilde{\Omega}$ of $\Omega$. 
    Hence, for example, if the series expansion $\{f^{(L)}\}_{L\in\mathbb{N}}$ converges uniformly on the whole $\Omega$, then we have $\|\widetilde{\partial}(f^{(L)}-f)\|_{n_{1};n_{2},\Omega}=\|\widetilde{\partial}(\widetilde{f}^{(L)}-\widetilde{f})\|_{n_{1};n_{2},\Omega}\|\leq2(n_{1}+n_{2})\|\widetilde{f}^{(L)}-\widetilde{f}\|_{n_{1}+n_{2},\Omega}=2(n_{1}+n_{2})\|f^{(L)}-f\|_{n_{1}+n_{2},\Omega}\to0$ (see also Remark \ref{remstablypartial}).
\end{Rem}

\section{Divergence and cyclic divergence operators for affine fully matricial functions}
Voiculescu \cite{v00a} proved the Poincar\'{e} lemma for cyclic gradients of the non-commutative polynomials $\mathbb{C}\langle X_{1},\dots,X_{n}\rangle$. Mai and Speicher \cite{ms19} generalized Voiculescu's work to general GDQ rings with some assumptions. In those works, the divergence and the cyclic divergence operators play an important role.

\begin{Def}(\cite[Definition 3.6, 3.11]{ms19})\label{div}
    Let $(A,\mu,\partial=(\partial_{1},\dots,\partial_{n}))$ be a multivariable GDQ ring. We call a tuple $\partial^*=(\partial_{1}^*,\dots,\partial_{n}^*)$ a \textit{divergence operator} for $(A,\mu,\partial)$ if the $\partial_{i}^*:A\otimes A\to A$, $i=1,\dots,n$ are all linear and satisfy the following equalities:
    \begin{equation*}
        \partial_{i}\circ\partial_{j}^*
        =(\partial_{j}^*\otimes\mathrm{id}_{A})\circ(\mathrm{id}_{A}\otimes\partial_{i})+(\mathrm{id}_{A}\otimes\partial_{j}^*)\circ(\partial_{i}\otimes\mathrm{id}_{A})+\delta_{i,j}\mathrm{id}_{A}^{\otimes2}
    \end{equation*}
    for all $i,j=1,\dots,n$. 
    
    Moreover, with a divergence operator $\partial^*$ for $(A,\mu,\partial)$, we can define a tuple $\mathcal{D}^*=(\mathcal{D}_{1}^*,\dots,\mathcal{D}_{n}^*)$, called a \textit{cyclic divergence operator} for $(A,\mu,\partial)$ (compatible with $\partial^*$), in such a way that the following equalities hold:
    \begin{equation*}
        \mathcal{D}_{i}\circ\mathcal{D}_{j}^*=\partial_{j}^*\circ\sigma\circ\partial_{i}+\delta_{i,j}\mathrm{id}_{A}
    \end{equation*}
    for all $i,j=1,2,\dots,n$.
\end{Def}

If there is an element $a_{i}\in A$ for each $i=1,\dots,n$ such that $\partial_{i}[a_{j}]=\delta_{i,j}1\otimes1$, then we can find a divergence operator as follows.
\begin{equation*}
    \partial_{j}^*[u]:=u\#a_{j}
\end{equation*}
for any $u\in A\otimes A$ (see \cite[Remark 3.7]{ms19}). Then, we can also obtain a cyclic divergence operator compatible with $\partial^*$ as follows.
\begin{equation*}
    \mathcal{D}_{j}^*[a]:=\partial_{j}^*[a\otimes1] \mbox{\quad or\quad}
    \mathcal{D}_{j}^*[a]:=\partial_{j}^*[1\otimes a]
\end{equation*}
for any $a\in A$ (see \cite[Lemma 3.13]{ms19}). We remark that, in general, their existence is non-trivial at all.

In the affine $M(B)$ setting, recall that $\partial[z(\theta)]=1_{A(M(B))}\otimes1_{A(M(B))}$ for $\theta\in B^d$ with $\theta(1_{B})=1$.
\begin{Def}\label{divertheta}
    For any $\theta\in B^d$ with $\theta(1_{B})=1$ and any $f\in A(M(B);M(B))$, we define $\partial_{\theta}^*[f]\in A(M(B))$ as follows.
    \begin{equation*}
        \partial_{\theta}^*[f]:=f\#z(\theta).
    \end{equation*}
    The above $\partial_{\theta}^*$ defines a linear map from $A(M(B);M(B))$ to $A(M(B))$, where Lemma \ref{sand} is still valid in the affine setting. 
    
    Also, for each $f\in A(M(B);M(B);M(B))$ we define $(\partial_{\theta}^*\otimes\mathrm{id})[f]$ and $(\mathrm{id}\otimes\partial_{\theta}^*)[f]$ in $A(M(B);M(B))$ as follows.
    \begin{equation*}
        (\partial_{\theta}^*\otimes\mathrm{id})[f]_{m,n}(\beta_{m};\beta_{n}):=(f\#_{1,2}z(\theta))_{m,n}(\beta_{m};\beta_{n}):=f_{m,m,n}(\beta_{m};\beta_{m};\beta_{n})\#_{1,2}z(\theta)_{m}(\beta_{m}),
    \end{equation*}
    \begin{equation*}
        (\mathrm{id}\otimes\partial_{\theta}^*)[f]_{m,n}(\beta_{m};\beta_{n}):=(f\#_{2,3}z(\theta))_{m,n}(\beta_{m};\beta_{n}):=f_{m,n,n}(\beta_{m};\beta_{n};\beta_{n})\#_{2,3}z(\theta)_{n}(\beta_{n})
    \end{equation*}
    for any $\beta_{m}\in M_{m}(B)$ and $\beta_{n}\in M_{n}(B)$, where $(A\otimes B\otimes C)\#_{1,2}X:=AXB\otimes C$ and $(A\otimes B\otimes C)\#_{2,3}X:=A\otimes BXC$. 
    
    We can show that $\partial_{\theta}^*\otimes\mathrm{id}$ and $\mathrm{id}\otimes\partial_{\theta}^*$ define linear maps
    \[
    A(M(B);M(B);M(B))\to A(M(B);M(B))
    \]
    in a similar fashion to Lemma \ref{sand}.
\end{Def}

\begin{Lem}\label{dividen}
    The linear map $\partial_{\theta}^*$ satisfies the following identity:
    \begin{equation*}
        \partial\circ\partial_{\theta}^*
        =(\partial_{\theta}^*\otimes\mathrm{id})\circ(\mathrm{id}\otimes\partial)+(\mathrm{id}\otimes\partial_{\theta}^*)\circ(\partial\otimes\mathrm{id})+\mathrm{id}_{A(M(B);M(B))}.
    \end{equation*}    
    Hence, $\partial_{\theta}^*$ is a divergence operator in the sense of Definition \ref{div}.
\end{Lem}
\begin{proof}
    Note that
    \begin{align*}
        \,&
        \left.
        \frac{d}{d\epsilon}
        f_{m+n;l}
        \left(
        \begin{bmatrix}
            \beta_{m}&\epsilon(\gamma\otimes1)\\
            &\beta_{n}
        \end{bmatrix}
        ;
        \beta
        \right)
        \right|_{\epsilon=0}\\
        &
        =
        \sum_{\substack{1\leq b\leq m\\1\leq c\leq n}}
        \gamma_{(b,c)}
        \left.
        \frac{d}{d\epsilon}
        f_{m+n;l}
        \left(
        \begin{bmatrix}
            \beta_{m}&\epsilon(e^{(m,n)}_{b,c}\otimes1)\\
            &\beta_{n}
        \end{bmatrix}
        ;
        \beta
        \right)
        \right|_{\epsilon=0}\\
        &
        =
        \sum_{\substack{1\leq a,b\leq m\\1\leq c,d\leq n\\1\leq e,f\leq l}}
        \gamma_{(b,c)}
        \left(
        \left.
        \frac{d}{d\epsilon}
        f_{m+n;l}
        \left(
        \begin{bmatrix}
            \beta_{m}&\epsilon(e^{(m,n)}_{b,c}\otimes1)\\
            &\beta_{n}
        \end{bmatrix}
        ;
        \beta
        \right)
        \right|_{\epsilon=0}
        \right)_{(a,m+d)}
        \begin{bmatrix}
            0&e^{(m,n)}_{a,d}\\
            &0
        \end{bmatrix}
        \otimes e^{(l)}_{e,f}
    \end{align*}
    and 
    \begin{align*}
        \,&((\partial\otimes\mathrm{id})f)_{m;n;l}(\beta_{m};\beta_{n};\beta)\\
        &=
        \sum_{\substack{1\leq a,b\leq m\\1\leq c,d\leq n\\1\leq e,f\leq l}}
        \left(
        \left.
        \frac{d}{d\epsilon}
        f_{m+n;l}
        \left(
        \begin{bmatrix}
            \beta_{m}&\epsilon(e^{(m,n)}_{b,c}\otimes1)\\
            &\beta_{n}
        \end{bmatrix}
        ;
        \beta
        \right)
        \right|_{\epsilon=0}
        \right)_{(a,m+d)}
        e^{(m)}_{a,b}\otimes e^{(n)}_{c,d}\otimes e^{(l)}_{e,f}
    \end{align*}
    for any $\beta_{m}\in M_{m}(B)$, $\beta_{n}\in M_{n}(B)$, $\beta\in M_{l}(B)$, $\gamma=[\gamma_{(b,c)}]\in M_{m,n}(\mathbb{C})$, $m,n,l\in\mathbb{N}$.
    We also have the same formulas of 
    \begin{align*}
        \left.
        \frac{d}{d\epsilon}
        f_{l;m+n}
        \left(
        \beta;
        \begin{bmatrix}
            \beta_{m}&\epsilon(\gamma\otimes1)\\
            &\beta_{n}
        \end{bmatrix}
        \right)
        \right|_{\epsilon=0}
        \mbox{ and }
        ((\mathrm{id}\otimes\partial)f)_{l;m;n}(\beta;\beta_{m};\beta_{n})
    \end{align*}
    for any $\beta_{m}\in M_{m}(B)$, $\beta_{n}\in M_{n}(B)$, $\beta\in M_{l}(B)$, $\gamma=[\gamma_{(b,c)}]\in M_{m,n}(\mathbb{C})$, $m,n,l\in\mathbb{N}$.
    Then, we have
    \begin{align*}
        \,&\begin{bmatrix}
            0&\nabla(\partial_{\theta}^*[f])_{m;n}(\beta_{m};\beta_{n})(\gamma)\\
            &0
        \end{bmatrix}\\
        &=\begin{bmatrix}
            0&\nabla(f\#z(\theta))_{m;n}(\beta_{m};\beta_{n})(\gamma)\\
            &0
        \end{bmatrix}\\
        &=\quad\frac{d}{d\epsilon}
        \left.
        (f\#z(\theta))_{m+n}
        \left(
        \begin{bmatrix}
            \beta_{m}&\epsilon(\gamma\otimes1)\\
            &\beta_{n}
        \end{bmatrix}
        \right)
        \right|_{\epsilon=0}\\
        &=\sum_{\substack{1\leq b',c'\leq m+n}}
        \mathrm{Tr}_{m+n}
        \left(
        \sigma
        (
        f_{m+n;m+n}
        (
        \left[
        \begin{smallmatrix}
            \beta_{m}&\\
            &\beta_{n}
        \end{smallmatrix}
        \right]
        ;
        \left[
        \begin{smallmatrix}
            \beta_{m}&\\
            &\beta_{n}
        \end{smallmatrix}
        \right]
        )
        )
        \#e^{(m+n)}_{b',c'}
        \begin{bmatrix}
            0&\gamma\\
            &0
        \end{bmatrix}
        \right)
        e^{(m+n)}_{c',b'}\\
        &\quad+\sum_{\substack{1\leq b',c'\leq m+n}}
        \mathrm{Tr}_{m+n}
        \Biggl(
        \sigma
        (
        \frac{d}{d\epsilon}
        \left.
        f_{m+n;m+n}
        (
        \left[
        \begin{smallmatrix}
            \beta_{m}&\epsilon(\gamma\otimes1)\\
            &\beta_{n}
        \end{smallmatrix}
        \right]
        ;
        \left[
        \begin{smallmatrix}
            \beta_{m}&\\
            &\beta_{n}
        \end{smallmatrix}
        \right]
        )
        \right|_{\epsilon=0}
        )
        \#e^{(m+n)}_{b',c'}\\
        &\hspace{10cm}
        \times\left[
        \begin{smallmatrix}
            z(\theta)_{m}(\beta_{m})&\\
            &z(\theta)_{n}(\beta_{n})
        \end{smallmatrix}
        \right]
        \Biggr)
        e^{(m+n)}_{c',b'}\\
        &\quad+\sum_{\substack{1\leq b',c'\leq m+n}}
        \mathrm{Tr}_{m+n}
        \Biggl(
        \sigma
        (
        \frac{d}{d\epsilon}
        \left.
        f_{m+n;m+n}
        (
        \left[
        \begin{smallmatrix}
            \beta_{m}&\\
            &\beta_{n}
        \end{smallmatrix}
        \right]
        ;
        \left[
        \begin{smallmatrix}
            \beta_{m}&\epsilon(\gamma\otimes1)\\
            &\beta_{n}
        \end{smallmatrix}
        \right]
        )
        \right|_{\epsilon=0}
        )
        \#e^{(m+n)}_{b',c'}\\
        &\hspace{10cm}
        \times\left[
        \begin{smallmatrix}
            z(\theta)_{m}(\beta_{m})&\\
            &z(\theta)_{n}(\beta_{n})
        \end{smallmatrix}
        \right]
        \Biggr)
        e^{(m+n)}_{c',b'}\\
        &=\begin{bmatrix}
            0&\alpha_{m,n}\left(f_{m;n}(\beta_{m};\beta_{n})\right)(\gamma)\\
            &0
        \end{bmatrix}\\
        &\quad+
        \begin{bmatrix}
            0&\alpha_{m,n}\left(\left((\mathrm{id}\otimes\partial_{\theta}^*)[(\partial\otimes\mathrm{id})f]\right)_{m;n}(\beta_{m};\beta_{n})\right)(\gamma)\\
            &0
        \end{bmatrix}\\
        &\quad+
        \begin{bmatrix}
            0&\alpha_{m,n}\left(\left((\partial_{\theta}^*\otimes\mathrm{id})[(\mathrm{id}\otimes\partial)f]\right)_{m;n}(\beta_{m};\beta_{n})\right)(\gamma)\\
            &0
        \end{bmatrix},
    \end{align*}
    where Lemma \ref{sand} is used in the 3rd equality and where $\alpha_{m,n}:M_{m}(\mathbb{C})\otimes M_{n}(\mathbb{C})\ni A\otimes B\mapsto (A\otimes B)\#(\cdot)\in B(M_{m,n}(\mathbb{C}))$ is the natural isomorphism. Thus, we obtain that 
    \begin{align*}
        \,&\partial[\partial_{\theta}^*[f]]_{m;n}(\beta_{m};\beta_{n})\\
        &=(\mathrm{id}\otimes\partial_{\theta}^*)[(\partial\otimes\mathrm{id})f]_{m;n}(\beta_{m};\beta_{n})
        +
        (\partial_{\theta}^*\otimes\mathrm{id})[(\mathrm{id}\otimes\partial)f]_{m;n}(\beta_{m};\beta_{n})
        +
        f_{m,n}(\beta_{m};\beta_{n})
    \end{align*}
    for any $f\in A(M(B);M(B))$ and any $\beta_{i}\in M_{i}(B)$ via $\alpha_{m,n}^{-1}$.
\end{proof}

\begin{Def}
    For each $f\in A(M(B))$ we define $D_{\theta}^*[f]\in A(M(B))$ as follows.
    \begin{equation*}
        D_{\theta}^*[f]:=\partial_{\theta}^*[f\otimes1]=fz(\theta).
    \end{equation*}
    It is clear that $D_{\theta}^*$ defines a linear map from $A(M(B))$ to $A(M(B))$.
\end{Def}

\begin{Lem}\label{lemcyclicdivergencedivergence}
    The linear map $D_{\theta}^*$ satisfies the following identity:
    \begin{equation*}
        D\circ D_{\theta}^*=\partial_{\theta}^*\circ\sigma\circ\partial+\mathrm{id}_{A(M(B))}.
    \end{equation*}
    Hence, $D_{\theta}^*$ is a cyclic divergence operator in the sense of Definition \ref{div}.
\end{Lem}
\begin{proof}
    Choose arbitrary $f\in A(M(B)$ and $\beta\in M_{n}(B)$. By definition, we have
    \begin{align*}
        (D\circ D_{\theta}^*)[f]_{n}(\beta)
        &=D[fz(\theta)]_{n}(\beta)\\
        &=\sum_{1\leq b,c\leq n}
        \mathrm{Tr}_{n}
        \left(
        \partial[fz(\theta)]_{n;n}(\beta;\beta)\#e^{(n)}_{b,c})
        \right)e^{(n)}_{c,b}\\
        &=\sum_{1\leq b,c\leq n}
        \mathrm{Tr}_{n}
        \left(
        f_{n}(\beta)
        \partial[z(\theta)]_{n;n}(\beta;\beta)\#e^{(n)}_{b,c}
        \right)e^{(n)}_{c,b}\\
        &
        \quad
        +
        \sum_{1\leq b,c\leq n}
        \mathrm{Tr}_{n}
        \left(
        (\partial[f]_{n;n}(\beta;\beta)\#e^{(n)}_{b,c})
        z(\theta)_{n}(\beta)
        \right)e^{(n)}_{c,b}\\
        &=\sum_{1\leq b,c\leq n}
        \mathrm{Tr}_{n}
        \left(
        f_{n}(\beta)e^{(n)}_{b,c}
        \right)e^{(n)}_{c,b}\\
        &\quad+\sum_{1\leq b,c\leq n}
        \mathrm{Tr}_{n}
        \left(
        e^{(n)}_{b,c}
        \sigma(\partial[f]_{n;n}(\beta;\beta))\#z(\theta)_{n}(\beta)
        \right)e^{(n)}_{c,b}\\
        &=f_{n}(\beta)
        +\sigma(\partial[f]_{n;n}(\beta;\beta))\#z(\theta)_{n}(\beta)\\
        &=f_{n}(\beta)+\partial_{\theta}^*(\sigma(\partial[f]))_{n}(\beta).
    \end{align*}
    Thus, we have obtained that $D\circ D_{\theta}^*
        =
        \mathrm{id}_{A(M(B))}+\partial_{\theta}^*\circ\sigma\circ\partial$.
\end{proof}

\section{Grading and number operators for affine fully matricial functions}
In Mai and Speicher's work, grading and number operators are also important in the Poincar\'{e} lemma for GDQ rings. We will consider their affine fully matricial analogues. Their constructions are the same as \cite[Lemma 3.8]{ms19}.
Let us set $N_{\theta}:=\partial_{\theta}^*\circ\partial:A(M(B))\to A(M(B))$, where $\theta\in B^d$ with $\theta(1)=1$.
\begin{Lem}
    We have
    \begin{equation*}
        N_{\theta}\circ\mu=\mu\circ(N_{\theta}\otimes\mathrm{id}_{A(M(B))}+\mathrm{id}_{A(M(B))}\otimes N_{\theta}),
    \end{equation*}
    where $\mu$ is the multiplication map of $A(M(B))$, that is, $N_{\theta}$ is a derivation from $A(M(B))$ to $A(M(B))$.
\end{Lem}
\begin{proof}
    Choose arbitrary $f^{(1)},f^{(2)}\in A(M(B))$ and $\beta\in M_{n}(B)$. We have
    \begin{align*}
        (N_{\theta}\circ\mu)[f^{(1)}\otimes f^{(2)}]_{n}(\beta)
        &=N_{\theta}[f^{(1)}f^{(2)}]_{n}(\beta)\\
        &=\partial_{\theta}^*[\partial[f^{(1)}f^{(2)}]]_{n}(\beta)\\
        &=\partial[f^{(1)}f^{(2)}]_{n;n}(\beta;\beta)\#z(\theta)_{n}(\beta)\\
        &=f^{(1)}_{n}(\beta)(\partial[f^{(2)}]_{n;n}(\beta;\beta)\#z(\theta)_{n}(\beta))
        +(\partial[f^{(1)}]_{n;n}(\beta;\beta)\#z(\theta)_{n}(\beta))f^{(2)}_{n}(\beta)\\
        &=f^{(1)}_{n}(\beta)\partial_{\theta}^*[\partial[f^{(2)}]]_{n}(\beta)
        +\partial_{\theta}^*[\partial[f^{(1)}]]_{n}(\beta)f^{(2)}_{n}(\beta)\\
        &=\mu(f^{(1)}\otimes\partial_{\theta}^*[\partial[f^{(2)}]]+\partial_{\theta}^*[\partial[f^{(1)}]]\otimes f^{(2)})_{n}(\beta)\\
        &=(\mu\circ(\mathrm{id}_{A(M(B))}\otimes N_{\theta}+N_{\theta}\otimes\mathrm{id}_{A(M(B))}))[f^{(1)}\otimes f^{(2)}]_{n}(\beta).
    \end{align*}
    Hence, we have obtained the desired identity.
\end{proof}

We define $N_{\theta}\otimes\mathrm{id},\,\mathrm{id}\otimes N_{\theta}:A(M(B);M(B))\to A(M(B);M(B))$ as follows (see also Definition \ref{divertheta}).
\begin{equation*}
    N_{\theta}\otimes\mathrm{id}=(\partial_{\theta}^*\otimes\mathrm{id})\circ(\partial\otimes\mathrm{id}),
\end{equation*}
\begin{equation*}
    \mathrm{id}\otimes N_{\theta}=(\mathrm{id}\otimes\partial_{\theta}^*)\circ(\mathrm{id}\otimes\partial).
\end{equation*}
In particular, we have $(N_{\theta}\otimes\mathrm{id})|_{A(M(B))^{\otimes2}}=N_{\theta}\otimes\mathrm{id}_{A(M(B))}$, $(\mathrm{id}\otimes N_{\theta})|_{A(M(B))^{\otimes2}}=\mathrm{id}_{A(M(B))}\otimes N_{\theta}$.

\begin{Lem}\label{lemnumtheta}
    We have
    \begin{equation*}
        \partial\circ N_{\theta}
        =(N_{\theta}\otimes\mathrm{id}+\mathrm{id}\otimes N_{\theta}+\mathrm{id}_{A(M(B);M(B))})\circ\partial.
    \end{equation*}
\end{Lem}
\begin{proof}
    By Lemma \ref{dividen} and the coassociativity of $\partial$, we have
    \begin{align*}
        \partial\circ N_{\theta}
        &=(\partial\circ\partial_{\theta}^*)\circ\partial\\
        &=\left((\partial_{\theta}^*\otimes\mathrm{id})\circ(\mathrm{id}\otimes\partial)+(\mathrm{id}\otimes\partial_{\theta}^*)\circ(\partial\otimes\mathrm{id})
        +\mathrm{id}_{A(M(B);M(B))}\right)\circ\partial\\
        &=(\partial_{\theta}^*\otimes\mathrm{id})\circ(\mathrm{id}\otimes\partial)\circ\partial+(\mathrm{id}\otimes\partial_{\theta}^*)\circ(\partial\otimes\mathrm{id})\circ\partial
        +\mathrm{id}_{A(M(B);M(B))}\circ\partial\\
        &=(\partial_{\theta}^*\otimes\mathrm{id})\circ(\partial\otimes\mathrm{id})\circ\partial+(\mathrm{id}\otimes\partial_{\theta}^*)\circ(\mathrm{id}\otimes\partial)\circ\partial
        +\mathrm{id}_{A(M(B);M(B))}\circ\partial\\
        &=\left(N_{\theta}\otimes\mathrm{id}+\mathrm{id}\otimes N_{\theta}+\mathrm{id}_{A(M(B);M(B))}\right)\circ\partial.
    \end{align*}
    Hence, we are done.
\end{proof}

We set $L_{\theta}:=N_{\theta}+\mathrm{id}_{A(M(B))}$ and define $L_{\theta}\otimes\mathrm{id},\,\mathrm{id}\otimes L_{\theta}:A(M(B);M(B))\to A(M(B);M(B))$ as follows.
\begin{equation*}
    L_{\theta}\otimes\mathrm{id}:=N_{\theta}\otimes\mathrm{id}+\mathrm{id}_{A(M(B);M(B))},
\end{equation*}
\begin{equation*}
    \mathrm{id}\otimes L_{\theta}:=\mathrm{id}\otimes N_{\theta}+\mathrm{id}_{A(M(B);M(B))}.
\end{equation*}
In particular, we have $(L_{\theta}\otimes\mathrm{id})|_{A(M(B))^{\otimes2}}=L_{\theta}\otimes\mathrm{id}_{A(M(B))}$, $(\mathrm{id}\otimes L_{\theta})|_{A(M(B))^{\otimes2}}=\mathrm{id}_{A(M(B))}\otimes L_{\theta}$.

\begin{Lem}
    The above map $L_{\theta}$ is a coderivation with respect to $\partial$, that is,
    \begin{equation*}
        \partial\circ L_{\theta}=(L_{\theta}\otimes\mathrm{id}+\mathrm{id}\otimes L_{\theta})\circ\partial.
    \end{equation*}
\end{Lem}
\begin{proof}
    By Lemma \ref{lemnumtheta}, we have
    \begin{align*}
        \partial\circ L_{\theta}
        &=\partial\circ(N_{\theta}+\mathrm{id}_{A(M(B))})\\
        &=\partial\circ N_{\theta}+\partial\\
        &=(N_{\theta}\otimes\mathrm{id}+\mathrm{id}\otimes N_{\theta}+\mathrm{id}_{A(M(B);M(B))})\circ\partial+\partial\\
        &=(L_{\theta}\otimes\mathrm{id}+\mathrm{id}\otimes L_{\theta}-\mathrm{id}_{A(M(B);M(B))})\circ\partial+\partial\\
        &=(L_{\theta}\otimes\mathrm{id}+\mathrm{id}\otimes L_{\theta})\circ\partial.
    \end{align*}
    Hence, we are done.
\end{proof}

Thus, $L_{\theta}$ is a \textit{grading operator}, and $N_{\theta}$ is a \textit{number operator} in the sense of \cite[Definition 3.2]{ms19}. 

\section{The affine fully matricial analogues of the Poincar\'{e} lemma}

\subsection{The case of the polynomial sub-bialgebra $\mathcal{Z}(B^d)$}
Let us consider the above operators on $\mathcal{Z}(B^d)$. Let us take a $\theta\in B^d$ with $\theta(1_{B})=1$. According to \cite[section 7]{v10}, $\mathcal{Z}(1^{\perp})$ and $z(\theta)$ are algebraically free and we can regard $\left(\mathcal{Z}(B^d),\partial|_{\mathcal{Z}(B^d)}\right)$ as $\left(\mathcal{Z}(1^{\perp})\langle z(\theta)\rangle,\partial_{z(\theta):\mathcal{Z}(1^{\perp})}\right)$, where $\mathcal{Z}(1^{\perp})$ is a subalgebra of $\mathcal{Z}(B^d)$ generated by $1^{\perp}$ and $\{1_{A(M(B))}\}$, where $1^{\perp}=\{\varphi\in B^d\,|\,\varphi(1_{B})=0\}$ (recall that $\partial[z(\theta)]=1_{A(M(B))}\otimes1_{A(M(B))}$). (Hence, we can study $(\mathcal{Z}(B^d),\partial|_{\mathcal{Z}(B^d)})$ as a counterpart of $(B\langle X\rangle,\partial_{X:B})$.) Here, let us set $N_{\theta,2}:=N_{\theta}\otimes\mathrm{id}+\mathrm{id}\otimes N_{\theta}+\mathrm{id}_{A(M(B);M(B))}:A(M(B);M(B))\to A(M(B);M(B))$.
Then, we have the following lemma:

\begin{Lem}\label{lemnum1}
    The operators $N_{\theta}$, $N_{\theta,2}$ and $L_{\theta}$ have eigenvectors. In particular, 
        \begin{equation*}
            N_{\theta}[c_{0}z(\theta)c_{1}\cdots z(\theta)c_{n}]=n\cdot c_{0}z(\theta)c_{1}\cdots z(\theta)c_{n},
        \end{equation*}
        \begin{equation*}
            L_{\theta}[c_{0}z(\theta)c_{1}\cdots z(\theta)c_{n}]=(n+1)\cdot c_{0}z(\theta)c_{1}\cdots z(\theta)c_{n}
        \end{equation*}
        and
        \begin{equation*}
            N_{\theta,2}[c_{0}z(\theta)c_{1}\cdots z(\theta)c_{n}\otimes c'_{0}z(\theta)c'_{1}\cdots z(\theta)c'_{m}]=(n+m+1)\cdot c_{0}z(\theta)c_{1}\cdots z(\theta)c_{n}\otimes c'_{0}z(\theta)c'_{1}\cdots z(\theta)c'_{m}
        \end{equation*}
        for any $n\in\mathbb{N}$ and $c_{i},c'_{j}\in\mathcal{Z}(1^{\perp})$.
\end{Lem}
\begin{proof}
    These formulas are confirmed by direct calculations.
\end{proof}

\begin{Lem}\label{lemnum2}
    Letting $\mathcal{Z}(B^d)_{\langle n\rangle}:=\{p\in\mathcal{Z}(B^d)\,|\,N_{\theta}[p]=n\cdot p\}$, $n\in\mathbb{N}$, we have
        \begin{equation*}
            \mathcal{Z}(B^d)=\mathcal{Z}(1^{\perp})\oplus\bigoplus_{n\geq1}\mathcal{Z}(B^d)_{\langle n\rangle}\quad\mbox{and}\quad
            \mathrm{ran}(N_{\theta}|_{\mathcal{Z}(B^d)})=\bigoplus_{n\geq1}\mathcal{Z}(B^d)_{\langle n\rangle}.
        \end{equation*}
    Similarly, letting $(\mathcal{Z}(B^d)^{\otimes2})_{\langle n\rangle}:=\{\xi\in\mathcal{Z}(B^d)^{\otimes2}\,|\,N_{\theta,2}[\xi]=n\cdot\xi\}$, $n\in\mathbb{N}$, we have
    \begin{equation*}
        \mathcal{Z}(B^d)^{\otimes2}
        =\mathcal{Z}(1^{\perp})^{\otimes2}\oplus\bigoplus_{n\geq2}(\mathcal{Z}(B^d)^{\otimes2})_{\langle n\rangle}.
    \end{equation*}
\end{Lem}
\begin{proof}
    Firstly, it is clear that $\mathcal{Z}(B^d)=\mathcal{Z}(1^{\perp})+\sum_{n\geq1}\mathcal{Z}(B^d)_{\langle n\rangle}$. If $p\in\mathcal{Z}(B^d)_{\langle n\rangle}\cap\mathcal{Z}(B^d)_{\langle m\rangle}$, $m>n\geq0$, then we have $(m-n)\cdot p=0$, that is, $p=0$. Thus, we have obtained that $\mathcal{Z}(B^d)=\mathcal{Z}(1^{\perp})\oplus\bigoplus_{n\geq1}\mathcal{Z}(B^d)_{\langle n\rangle}$. Now, the second identity is clear. The third identity also follows in the same way.
\end{proof}

By Lemmas \ref{lemnum1} and \ref{lemnum2}, we have the next corollary, which contains necessary assumptions of the Poincar\'{e} lemma for GDQ rings due to Mai and Speicher.

\begin{Cor}\label{corinjinc}
    The operators $L_{\theta}|_{\mathcal{Z}(B^d)}$ and $N_{\theta,2}|_{\mathcal{Z}(B^d)^{\otimes2}}$ are injective. Also, we have 
    \begin{align*}
        \mathrm{ran}(\partial^*_{\theta}|_{\mathcal{Z}(B^d)^{\otimes2}})\subset \mathrm{ran}(N_{\theta}|_{\mathcal{Z}(B^d)}) \mbox{\quad and\quad} \mathrm{ran}(D^*_{\theta}|_{\mathcal{Z}(B^d)})\subset \mathrm{ran}(N_{\theta}|_{\mathcal{Z}(B^d)}).
    \end{align*}
\end{Cor}

\begin{Rem}
    We can show the injectivity of $L_{\theta}$ and $N_{\theta,2}$ for general analytic fully matricial functions on the stably matricial disk $R\mathcal{D}_{0}(M_{k}(\mathbb{C}))$ (see Appendix \ref{appendixinjenctions}).
\end{Rem}

Therefore, we have the the following facts by \cite[Theorems 4.1 and 4.5]{ms19}:

\begin{Thm}\label{poic}
    For any $p\in\mathcal{Z}(B^d)=(\mathcal{Z}(1^{\perp}))\langle z(\theta)\rangle$, where $\theta\in B^d$ with $\theta(1)=1$, the following conditions are equivalent:
    \begin{enumerate}
        \item there exists a $q\in\mathcal{Z}(B^d)$ such that $D[q]=p$.
        \item $\partial[p]=(\sigma\circ\partial)[p]$.
        \item $(D\circ D^*_{\theta})[p]=L_{\theta}[p]$.
    \end{enumerate}
\end{Thm}
\begin{Thm}\label{poif}
    For any $\xi\in\mathcal{Z}(B^d)^{\otimes2}=(\mathcal{Z}(1^{\perp}))\langle z(\theta)\rangle^{\otimes2}$, where $\theta\in B^d$ with $\theta(1)=1$, the following conditions are equivalent:
    \begin{enumerate}
        \item there exists a $q\in\mathcal{Z}(B^d)$ such that $\partial[q]=\xi$.
        \item $(\partial\otimes\mathrm{id})[\xi]=(\mathrm{id}\otimes\partial)[\xi]$.
        \item $(\partial\circ\partial_{\theta}^*)[\xi]=N_{\theta,2}[\xi]$.
    \end{enumerate}
\end{Thm}

\begin{Rem}\label{bofx}
    Let $B$ be a unital algebra over $\mathbb{C}$. Define $B\langle X\rangle$ as the algebraic free product $B*_{\mathbb{C}}\mathbb{C}\langle X\rangle$. Then, we have the free difference quotient $\partial_{X:B}$ given by
    \begin{equation*}
        \partial_{X:B}[b_{0}Xb_{1}\cdots Xb_{n}]
        :=\sum_{i=0}^{n}b_{0}Xb_{1}X\cdots Xb_{i-1}\otimes b_{i}X\cdots Xb_{n}
    \end{equation*}
    for each monomial $b_{0}Xb_{1}X\cdots Xb_{n}$. Remark that $\partial_{X:B}[X]=1\otimes 1$. Hence, we can define divergence, cyclic divergence, grading and number operators similarly to $\partial_{\theta}^*$, $D^*_{\theta}$, $L_{\theta}$ and $N_{\theta}$, respectively. Thus, the Poincar\'{e} lemma also holds for $(B\langle X\rangle,\mu,\partial_{X:B})$.
\end{Rem}

\subsection{The case of stably matricial analytic functions in the case of $B=\mathbb{C}$}
    We will show the Poincar\'{e} lemma of $D$ and $\partial$ for the stably matricial analytic functions $A(R\mathcal{D}_{0}(\mathbb{C}))$, where $\mathcal{D}_{0}(\mathbb{C})=(\mathcal{D}_{0}(\mathbb{C})_{n})_{n\in\mathbb{N}}$ with $\mathcal{D}_{0}(\mathbb{C})_{n}=\{\omega\in M_{n}(\mathbb{C})\,|\,\|\omega\|<1\}$.

\begin{Thm}\label{thmpoincarecyclic}
    Assume $R>0$. For any $f\in A(R\mathcal{D}_{0}(\mathbb{C}))$, the following conditions are equivalent:
    \begin{enumerate}
        \item there exists a $g\in A(R\mathcal{D}_{0}(\mathbb{C}))$ such that $Dg=f$.
        \item $(\partial f)_{n;n}(\omega;\omega)=\sigma((\partial f)_{n;n}(\omega;\omega))$ for any $n\in\mathbb{N}$ and $\omega\in R\mathcal{D}_{0}(\mathbb{C})$.
        \item $(D\circ D^*_{\theta})f=L_{\theta}f$.
    \end{enumerate}
\end{Thm}
\begin{proof}
    Basically, the proof is done in a similar fashion to \cite[Theorem 4.1]{ms19}.
    
    (1)$\Rightarrow$(2): By \cite[Theorem 13.8]{v10}, we can approximate each element of $A(R\mathcal{D}_{0}(\mathbb{C}))$ by a sequence of $\mathcal{Z}((\mathbb{C})^d)$ in $\{\|\cdot\|_{n,R'}\}_{n\in\mathbb{N},R'<R}$, where $(\mathbb{C})^d$ is the topological dual of $\mathbb{C}$ (in this case, $(\mathbb{C})^d=\mathbb{C}\cdot\theta$ with $\theta(1)=1$). 
    Hence, there exists a sequence $\{g^{(L)}\}\subset\mathcal{Z}((\mathbb{C})^d)$ such that $g^{(L)}_{n}\to g_{n}$ as $L\to\infty$ in $\|\cdot\|_{n,R'}$ for any $0<R'<R$ and $n\in\mathbb{N}$. By condition (1), Corollaries \ref{corcontifmfdq} and \ref{corcontifmcg}, we observe that $(\partial Dg^{(L)})_{n_{1};n_{2}}\to(\partial f)_{n_{1};n_{2}}$ in $\|\cdot\|_{n_{1}+n_{2},R'}$ for any $n_{1},n_{2}\in\mathbb{N}$ and $0<R'<R$. 
    By \cite[Lemma 3.4]{ms19}, we have $(\partial Dg^{(L)})_{n;n}(\omega;\omega)=\sigma(\partial Dg^{(L)})_{n;n}(\omega;\omega))$. Thus, we observe that
    \begin{align*}
        (\partial f)_{n;n}(\omega;\omega)
        &=\lim_{L\to\infty}(\partial Dg^{(L)})_{n;n}(\omega;\omega)\\
        &=\lim_{L\to\infty}\sigma((\partial Dg^{(L)})_{n;n}(\omega;\omega))
        =\sigma((\partial f)_{n;n}(\omega;\omega))
    \end{align*}
    for any $n\in\mathbb{N}$ and $\omega\in R\mathcal{D}_{0}(\mathbb{C})$. Thus, we have obtained condition (2).

    (2)$\Rightarrow$(3): Using condition (2) and Lemma \ref{lemcyclicdivergencedivergence}, we have
    \begin{align*}
        ((D\circ D^*_{\theta})f)_{n}(\omega)
        &=
        ((\partial^*_{\theta}\circ\sigma\circ\partial)f)_{n}(\omega)+f_{n}(\omega)\\
        &=
        \sigma((\partial f)_{n;n}(\omega;\omega))\#z(\theta)_{n}(\omega)+f_{n}(\omega)\\
        &=
        (\partial f)_{n;n}(\omega;\omega)\#z(\theta)_{n}(\omega)+f_{n}(\omega)\\
        &=((\partial^*_{\theta}\circ\partial+\mathrm{id}_{A(R\mathcal{D}_{0}(M_{k}))})f)_{n}(\omega)\\
        &=(L_{\theta}f)_{n}(\omega)
    \end{align*}
    for any $n\in\mathbb{N}$ and $\omega\in R\mathcal{D}_{0}(\mathbb{C})$. Thus, we have obtained condition (3).

    (3)$\Rightarrow$(1): By \cite[Theorem 13.8]{v10}, there exists a sequence $\{f^{(L)}\}_{L\in\mathbb{N}}$ of $\mathcal{Z}((\mathbb{C})^d)$ such that $f^{(L)}_{n}\to f_{n}$ in $\|\cdot\|_{n,R'}$ for every $n\in\mathbb{N}$ and $0<R'<R$. In particular, we can take a $f^{(L)}$, $L\in\mathbb{N}$, as follows.
    \begin{align*}
        f^{(L)}
        =a(0)1_{A(R\mathcal{D}_{0}(\mathbb{C}))}
        +
        \sum_{l=1}^{L}
        a(l)z(\theta)^l,
    \end{align*}
    where $a(l)\in\mathbb{C}$ and $\theta\in(\mathbb{C})^d\simeq\mathbb{C}$ with $\theta(1)=1$. Note that this series converges uniformly and absolutely on compact subsets of $\mathcal{D}_{0}(\mathbb{C})_{n}$ for each $n\in\mathbb{N}$.
    Then, by Corollary \ref{corcontifmcg}, we have $((D\circ D^*_{\theta})f^{(L)})_{n}\to((D\circ D^*_{\theta})f)_{n}$ in $\|\cdot\|_{n,R'}$ for every $n\in\mathbb{N}$ and $0<R'<R$. By Corollary \ref{corinjinc}, for each $L\in\mathbb{N}$ there exists an element $g^{(L)}$ of $\mathcal{Z}((\mathbb{C})^d)$ such that $D^*_{\theta}f^{(L)}=N_{\theta}g^{(L)}$. In fact, we can take a $g^{(L)}$ as follows.
    \begin{align*}
        g^{(L)}
        =
        a(0)z(\theta)
        +
        \sum_{l=1}^{L}\frac{a(l)}{l+1}z(\theta)^{l+1}
        =\sum_{l=1}^{L}\frac{a(l-1)}{l}z(\theta)^l.
    \end{align*}
    Since $\{f^{(L)}_{n}\}_{L\in\mathbb{N}}$ converges absolutely on compact subsets of $R\mathcal{D}_{0}(\mathbb{C})$ for every $n\in\mathbb{N}$, so is $\{g^{(L)}_{n}\}_{L\in\mathbb{N}}$ for every $n\in\mathbb{N}$ (we denote its limit by $g=(g_{n})_{n\in\mathbb{N}}\in A(R\mathcal{D}_{0}(\mathbb{C}))$). Here, note that $\{g^{(L)}\}_{L\in\mathbb{N}}$ also converges to $g$ uniformly on compact subsets of $R\mathcal{D}_{0}(\mathbb{C})_{n}$ for each $n\in\mathbb{N}$, since the series
    \begin{align*}
        \sum_{l\geq1}\left|\frac{a(l-1)}{l}\right|r^l
    \end{align*}
    converges for any $0<r<R$ by the absolutely convergence of $g^{(L)}$.
    Applying \cite[Lemma 3.12]{ms19} to the following second equality, we have
    \begin{align*}
        (D\circ D^*_{\theta})f^{(L)}=(D\circ N^*_{\theta})g^{(L)}=(L_{\theta}\circ D)g^{(L)}.
    \end{align*}
    Thus, using condition (3), we have
    \begin{align*}
        (L_{\theta}f)_{n}=
        \lim_{L\to\infty}((D\circ D^*_{\theta})f^{(L)})_{n}=\lim_{L\to\infty}(L_{\theta}Dg^{(L)})_{n}=(L_{\theta}Dg)_{n},
    \end{align*}
    where the above limit means the convergence in $\|\cdot\|_{n,R'}$ for every $n\in\mathbb{N}$ and $0<R'<R$. By the injectivity of $L_{\theta}$ (see Theorem \ref{thminjltheta}), we have $f=Dg$. Hence, we are done.
\end{proof}

\begin{Thm}\label{thmpoincarefree}
    Assume $R>0$. For any $f\in A(R\mathcal{D}_{0}(\mathbb{C});R\mathcal{D}_{0}(\mathbb{C}))$, the following conditions are equivalent:
    \begin{enumerate}
        \item there exists a $g\in A(R\mathcal{D}_{0}(\mathbb{C}))$ such that $\partial g=f$.
        \item $(\partial\otimes\mathrm{id})f=(\mathrm{id}\otimes\partial)f$.
        \item $(\partial\circ\partial^*_{\theta})f=N_{\theta,2}f$.
    \end{enumerate}
\end{Thm}
\begin{proof}
    The proof is also done in a similar fashion to \cite[Theorem 4.5]{ms19}.

    (1)$\Rightarrow$(2): See \cite[Lemma 7.6]{v04}.

    (2)$\Rightarrow$(3): Using Lemma \ref{dividen}, we have
    \begin{align*}
        (\partial\circ\partial^*_{\theta})f
        &=(\partial_{\theta}^*\otimes\mathrm{id})(\mathrm{id}\otimes\partial)f+(\mathrm{id}\otimes\partial_{\theta}^*)(\partial\otimes\mathrm{id})f+f\\
        &=(\partial_{\theta}^*\otimes\mathrm{id})(\partial\otimes\mathrm{id})f+(\mathrm{id}\otimes\partial_{\theta}^*)(\mathrm{id}\otimes\partial)f+f
        =N_{\theta,2}f.
    \end{align*}
    Hence, we have obtained condition (3).

    (3)$\Rightarrow$(1): By Appendix \ref{appendixpolyapproximate}, there exists a sequence $\{f^{(L)}\}_{L\in\mathbb{N}}$ of $\mathcal{Z}((\mathbb{C})^d)^{\otimes2}$ such that $f^{(L)}_{n_{1};n_{2}}\to f_{n_{1};n_{2}}$ in $\|\cdot\|_{n_{1};n_{2},R'}$ for every pair $n_{1},n_{2}\in\mathbb{N}$ and $0<R'<R$. In particular, we can take a $f^{(L)}$, $L\in\mathbb{N}$, as follows.
    \begin{align*}
        f^{(L)}
        =
        a(0)1_{A(R\mathcal{D}_{0}(\mathbb{C});R\mathcal{D}_{0}(\mathbb{C}))}
        +
        \sum_{l\geq1}a(l)z(\theta)^l\otimes z(\theta)^l,
    \end{align*}
    where $a(l)\in\mathbb{C}$ and $\theta\in(\mathbb{C})^d\simeq\mathbb{C}$ with $\theta(1)=1$. Note that this series converges uniformly and absolutely on compact subsets of $R\mathcal{D}_{0}(\mathbb{C})_{n_{1}}\oplus R\mathcal{D}_{0}(\mathbb{C})_{n_{2}}$ for each pair $n_{1},n_{2}\in\mathbb{N}$. Using Corollary \ref{corcontifmfdq} and Lemma \ref{lemcontisharp}, we have $((\partial\circ\partial^*_{\theta})f^{(L)})_{n_{1};n_{2}}\to((\partial\circ\partial^*_{\theta})f)_{n_{1};n_{2}}$ in $\|\cdot\|_{n_{1};n_{2},R'}$ for each pair $n_{1},n_{2}\in\mathbb{N}$ and $0<R'<R$. By Corollary \ref{corinjinc}, for each $L\in\mathbb{N}$ there exists an element $g^{(L)}$ of $\mathcal{Z}((\mathbb{C})^d)$ such that $\partial^{*}_{\theta}f^{(L)}=N_{\theta}g^{(L)}$. In fact, we can find a $g^{(L)}$, $L\in\mathbb{N}$, as follows.
    \begin{align*}
        g^{(L)}
        =a(0)z(\theta)
        +
        \sum_{l=1}^{L}
        \frac{a(l)}{2l+1}z(\theta)^{2l+1}
        =
        \sum_{l=1}^{L}
        \frac{a(l-1)}{2l-1}z(\theta)^{2l-1}.
    \end{align*}
    Since $\{f^{(L)}_{n_{1};n_{2}}\}_{L\in\mathbb{N}}$ converges to $f_{n_{1};n_{2}}$ uniformly and absolutely on compact subsets of $R\mathcal{D}_{0}(\mathbb{C})_{n_{1}}\oplus R\mathcal{D}_{0}(\mathbb{C})_{n_{2}}$ for each pair $n_{1},n_{2}\in\mathbb{N}$, so is $\{g^{(L)}_{n_{1};n_{2}}\}_{L\in\mathbb{N}}$ for each pair $n_{1},n_{2}\in\mathbb{N}$ (we denote its limit by $g=(g_{n_{1};n_{2}})_{n_{1},n_{2}\in\mathbb{N}}\in A(R\mathcal{D}_{0}(\mathbb{C}))$).
    Also, note that $\{g^{(L)}_{n}\}_{L\in\mathbb{N}}$ converges to $g_{n}$ uniformly on compact subsets of $R\mathcal{D}_{0}(\mathbb{C})_{n}$ for each $n\in\mathbb{N}$, since the series
    \begin{align*}
        \sum_{l\geq1}
        \left|\frac{a(l-1)}{2l-1}\right|r^{2l-1}
    \end{align*}
    converges for any $0<r<R$ by the absolutely convergence of $g^{(L)}$.
    Using Lemma \ref{lemnumtheta}, we have
    \begin{align*}
        (\partial\circ\partial^*_{\theta})f^{(L)}
        =(\partial\circ N_{\theta})g^{(L)}
        =(N_{\theta,2}\circ\partial)g^{(L)}.
    \end{align*}
    Thus, using condition (3), we have
    \begin{align*}
        (N_{\theta,2}f)_{n_{1};n_{2}}
        =\lim_{L\to\infty}((\partial\circ\partial^*_{\theta})f^{(L)})_{n_{1};n_{2}}
        =\lim_{L\to\infty}(N_{\theta,2}\partial g^{(L)})_{n_{1};n_{2}}
        =(N_{\theta,2}\partial g)_{n_{1};n_{2}},
    \end{align*}
    where the above limit means the convergence in $\|\cdot\|_{n_{1};n_{2},R'}$ for each pair $n_{1},n_{2}\in\mathbb{N}$ and $0<R'<R$. By the injectivity of $N_{\theta,2}$ (see Theorem \ref{thminjnumtheta2}), we have $f=\partial g$. Hence, we are done.
\end{proof}


\section{The kernel of cyclic derivative}
\subsection{On $Gr(B)$}
The following lemma immediately follows from Corollary \ref{cyclic}:
\begin{Lem}
    Let $\Omega$ be an open fully matricial $B$-set of $Gr(B)$ and $f\in A(\Omega)$. Then, $f\in\ker(\widetilde{D})$ if and only if $\mathrm{ran}\left(\left(\widetilde{\nabla}_{n,n}f_{n+n}\right)(\pi;\pi)\right)\subset[M_{n}(\mathbb{C}),M_{n}(\mathbb{C})]$ holds for every $n\in\mathbb{N}$ and any $\pi\in\Omega_{n}$.
\end{Lem}
\begin{proof}
    Remark that $\{A\in M_{n}(\mathbb{C})\,|\,\mathrm{Tr}_{n}(A)=0\}=[M_{n}(\mathbb{C}),M_{n}(\mathbb{C})]$ (see \cite{am57}).
\end{proof}

\subsection{On $\mathcal{Z}(B^d)$}
We will determine the kernel of $D|_{\mathcal{Z}(B^d)}$. Firstly, consider $(B\langle X\rangle,\mu,\partial_{X:B})$, where $B$ is a unital algebra over $\mathbb{C}$ and $B\langle X\rangle:=B*_{\mathbb{C}}\mathbb{C}\langle X\rangle$. Also, set $\delta_{X:B}:=\mu\circ\sigma\circ\partial_{X:B}$ and define the $B$-symmetrization operator $C:=(1\otimes X)\#\delta_{X:B}[\cdot]:B\langle X\rangle\to B\langle X\rangle$, that is,
\begin{equation*}
    C[b_{0}Xb_{1}\cdots Xb_{n}]
    =\sum_{i=0}^{n-1}b_{i+1}X\cdots Xb_{n}b_{0}X\cdots Xb_{i}X
\end{equation*}
for each monomial $b_{0}Xb_{1}\cdots Xb_{n}$.

\begin{Rem}\label{remexpressionpoly}
    If $p\in B\langle X\rangle$ is homogeneous (that is, $N[p]=mp$, where $N$ is the number operator for $B\langle X\rangle$ constructed in a similar way to $N_{\theta}$), then we can write
    \begin{align*}
        p=\sum_{i=1}^{k}b^{(i)}_{0}Xb^{(i)}_{1}\cdots Xb^{(i)}_{m},\quad b^{(i)}_{j}\in B.
    \end{align*}
    \begin{proof}
    Assume $p\not=0$. If there exist a family $\{m_{i}\}_{i=1}^{k}$ of natural numbers with $m_{i}<m_{i+1}$ and a family $\{b^{(j,i)}_{p}\,|\,\,1\leq i\leq k,\,n(i)\in\mathbb{N},\,1\leq j\leq n(i),\,1\leq p\leq m_{i}\}$ of elements of $B$ such that 
    \begin{align*}
        p=\sum_{\substack{1\leq i\leq k}}
        \sum_{1\leq j\leq n(i)}
        b^{(j,i)}_{0}Xb^{(j,i)}_{1}\cdots Xb^{(j,i)}_{m_{i}},
    \end{align*}
    then 
    \begin{align*}
        mp=N[p]
        =
        \sum_{\substack{1\leq i\leq k}}
        \sum_{1\leq j\leq n(i)}
        m_{i}b^{(j,i)}_{0}Xb^{(j,i)}_{1}\cdots Xb^{(j,i)}_{m_{i}}.
    \end{align*}
   Hence, we have
    \begin{align*}
        \sum_{\substack{1\leq i\leq k}}
        (m-m_{i})
        \sum_{1\leq j\leq n(i)}
        b^{(j,i)}_{0}Xb^{(j,i)}_{1}\cdots Xb^{(j,i)}_{m_{i}}=0.
    \end{align*}
    
    If $m-m_{i}\not=0$, that is, $m_{i}\not=m$, then $\sum_{1\leq j\leq n(i)}
        b^{(j,i)}_{0}Xb^{(j,i)}_{1}\cdots Xb^{(j,i)}_{m_{i}}=0$ (we can see this by the repeated use of the number operator $N$). This implies that there exists a unique $i_{0}\in\{1,2,\dots,k\}$ such that $m_{i_{0}}=m$ by the assumption of $p\not=0$.
    Thus, we have
    \begin{align*}
        p=\sum_{1\leq j\leq n(i_{0})}b^{(j,i_{0})}_{0}Xb^{(j,i_{0})}_{1}\cdots Xb^{(j,i_{0})}_{m}.
    \end{align*}
    Hence, we are done.
    \end{proof}
\end{Rem}

\begin{Thm}\label{ker}
    We have $\ker(\delta_{X:B})=B+[B\langle X\rangle,B\langle X\rangle]=\ker(C)$.
\end{Thm}
\begin{proof}
    By definition, it is clear that $\ker(\delta_{X:B})\subset\ker(C)$. We also have $B+[B\langle X\rangle,B\langle X\rangle]\subset\ker(\delta_{X:B})$, because the following identity holds:
    \begin{equation*}
        \delta_{X:B}[pq]=(\sigma\circ\partial_{X:B})[p]\#q+(\sigma\circ\partial_{X:B})[q]\#p 
    \end{equation*}
    for any $p,q\in B\langle X\rangle$. 
    
    We will show that $\ker(C)\subset B+[B\langle X\rangle,B\langle X\rangle]$. Here, it suffices to show the statement for any homogeneous elements of $\ker(C)$, since it is clear that $p\in\ker(C)$ if and only if each homogeneous term of $p$ is in $\ker(C)$ (this is easily seen by the use of discussion in Remark \ref{remexpressionpoly}). 
    
    Let $m\in\mathbb{N}\cup\{0\}$ be the degree of $p\in\ker(C)\setminus\{0\}$.
    If we write $p=\sum_{i=1}^{k}b^{(i)}_{0}Xb^{(i)}_{1}\cdots Xb^{(i)}_{m}$, then
    \begin{align}\label{(1)}
        \sum_{i=1}^{k}b^{(i)}_{m}b^{(i)}_{0}X\cdots Xb^{(i)}_{m-1}X=-\sum_{i=1}^{k}\sum_{j=0}^{m-2}b^{(i)}_{j+1}X\cdots Xb^{(i)}_{m}b^{(i)}_{0}X\cdots Xb^{(i)}_{j}X
    \end{align}
    since $C[p]=0$. Also, we observe that
    \begin{align*}
        p&=\sum_{i=1}^{k}[b^{(i)}_{0}X,b^{(i)}_{1}X\cdots Xb^{(i)}_{m}]
        +\sum_{i=1}^{k}b^{(i)}_{1}Xb^{(i)}_{2}\cdots b^{(i)}_{m}b^{(i)}_{0}X\\
        &=\sum_{i=1}^{k}[b^{(i)}_{0}X,b^{(i)}_{1}X\cdots Xb^{(i)}_{m}]
        +\sum_{i=1}^{k}[b^{(i)}_{1}X,b^{(i)}_{2}X\cdots Xb^{(i)}_{m}b^{(i)}_{0}X]
        +\sum_{i=1}^{k}b^{(i)}_{2}X\cdots b^{(i)}_{m}b^{(i)}_{0}Xb^{(i)}_{1}X\\
        &=\sum_{i=1}^{k}[b^{(i)}_{0}X,b^{(i)}_{1}X\cdots Xb^{(i)}_{m}]
        +\sum_{j=1}^{2}\sum_{i=1}^{k}[b^{(i)}_{j}X,b^{(i)}_{j+1}X\cdots Xb^{(i)}_{m}b^{(i)}_{0}X\cdots b^{(i)}_{j-1}X]\\
        &\quad
        +\sum_{i=1}^{k}b^{(i)}_{3}X\cdots b^{(i)}_{m}b^{(i)}_{0}Xb^{(i)}_{1}Xb^{(i)}_{2}X\\
        &\quad\vdots\\
        &=\sum_{i=1}^{k}[b^{(i)}_{0}X,b^{(i)}_{1}X\cdots Xb^{(i)}_{m}]
        +\sum_{j=1}^{m-1}\sum_{i=1}^{k}[b^{(i)}_{j}X,b^{(i)}_{j+1}X\cdots Xb^{(i)}_{m}b^{(i)}_{0}X\cdots b^{(i)}_{j-1}X]\\
        &\quad+\sum_{i=1}^{k}b^{(i)}_{m}b^{(i)}_{0}X\cdots b^{(i)}_{m-1}X.
    \end{align*}
    With this observation and equality (\ref{(1)}), we have
    \begin{align*}
        mp
        &=\sum_{i=1}^{k}[b^{(i)}_{0}X,b^{(i)}_{1}X\cdots Xb^{(i)}_{m}]
        +\sum_{j=1}^{m-1}\sum_{i=1}^{k}[b^{(i)}_{j}X,b^{(i)}_{j+1}X\cdots Xb^{(i)}_{m}b^{(i)}_{0}X\cdots b^{(i)}_{j-1}X]\\
        &\quad+(m-1)p-\sum_{i=1}^{k}\sum_{j=0}^{m-2}b^{(i)}_{j+1}X\cdots Xb^{(i)}_{m}b^{(i)}_{0}X\cdots Xb^{(i)}_{j}X\\
        &=\sum_{i=1}^{k}[b^{(i)}_{0}X,b^{(i)}_{1}X\cdots Xb^{(i)}_{m}]
        +\sum_{j=1}^{m-1}\sum_{i=1}^{k}[b^{(i)}_{j}X,b^{(i)}_{j+1}X\cdots Xb^{(i)}_{m}b^{(i)}_{0}X\cdots b^{(i)}_{j-1}X]\\
        &\quad+\sum_{i=1}^{k}\sum_{j=0}^{m-2}\left(b^{(i)}_{0}Xb^{(i)}_{1}\cdots Xb^{(i)}_{m}-b^{(i)}_{j+1}X\cdots Xb^{(i)}_{m}b^{(i)}_{0}X\cdots Xb^{(i)}_{j}X\right)\\
        &=\sum_{i=1}^{k}[b^{(i)}_{0}X,b^{(i)}_{1}X\cdots Xb^{(i)}_{m}]
        +\sum_{j=1}^{m-1}\sum_{i=1}^{k}[b^{(i)}_{j}X,b^{(i)}_{j+1}X\cdots Xb^{(i)}_{m}b^{(i)}_{0}X\cdots b^{(i)}_{j-1}X]\\
        &\quad+\sum_{i=1}^{k}\sum_{j=0}^{m-2}[b^{(i)}_{0}Xb^{(i)}_{1}\cdots Xb^{(i)}_{j}X,b^{(i)}_{j+1}X\cdots Xb^{(i)}_{m}].
    \end{align*}
    Thus, we obtain that
    \begin{align*}
        p
        &=\frac{1}{m}\sum_{i=1}^{k}\Bigl([b^{(i)}_{0}X,b^{(i)}_{1}X\cdots Xb^{(i)}_{m}]
        +\sum_{j=1}^{m-1}[b^{(i)}_{j}X,b^{(i)}_{j+1}X\cdots Xb^{(i)}_{m}b^{(i)}_{0}X\cdots b^{(i)}_{j-1}X]\\
        &\quad\quad\quad+\sum_{j=0}^{m-2}[b^{(i)}_{0}Xb^{(i)}_{1}\cdots Xb^{(i)}_{j}X,b^{(i)}_{j+1}X\cdots Xb^{(i)}_{m}]\Bigr)\in[B\langle X\rangle,B\langle X\rangle],
    \end{align*}
    that is, $\ker(C)\subset B+[B\langle X\rangle,B\langle X\rangle]$.
\end{proof}

Here, let us return to consider $(\mathcal{Z}(B^d),\partial|_{\mathcal{Z}(B^d)},D|_{\mathcal{Z}(B^d)})$. With letting $C_{\theta}:=D_{\theta}^*\circ D:A(M(B))\to A(M(B))$, $C_{\theta}|_{\mathcal{Z}(B^d)}$ is the symmetrization operator, that is, 
\begin{equation*}
    C_{\theta}[c_{0}z(\theta)c_{1}\cdots z(\theta)c_{n}]
    =\sum_{i=0}^{n}c_{i+1}z(\theta)\cdots c_{n}z(\theta)c_{i}z(\theta)
\end{equation*}
for any $c_{0}z(\theta)\cdots z(\theta)c_{n}\in(\mathcal{Z}(1^{\perp}))\langle z(\theta)\rangle=\mathcal{Z}(B^d)$. Hence, Theorem \ref{ker} shows the following fact:
\begin{Cor}
    Both the kernels of $D|_{\mathcal{Z}(B^d)}$ and $C_{\theta}|_{\mathcal{Z}(B^d)}$ are exactly $\mathcal{Z}(1^{\perp})+[\mathcal{Z}(B^d),\mathcal{Z}(B^d)]$.
\end{Cor}

\section{The exact sequence for $D|_{\mathcal{Z}(B^d)}$}
Recall the Poincar\'{e} lemma and a fact about the kernel of cyclic gradient $\delta=(\delta_{1},\dots,\delta_{n})$ due to Voiculescu.

\begin{Thm}(\cite[Theorem 1]{v00a})
    Let $p_{1},\dots,p_{n}\in\mathbb{C}\langle X_{1},\dots,X_{n}\rangle$. The following conditions are equivalent:
    \begin{enumerate}
        \item there exists a $q\in\mathbb{C}\langle X_{1},\dots,X_{n}\rangle$ such that $\delta_{j}[q]=p_{j}$ for all $j=1,\dots,n$.
        \item $(p_{1},p_{2},\dots,p_{n})\in\ker(\theta)$, where $\theta$ is the linear map from $\mathbb{C}\langle X_{1},\dots,X_{n}\rangle^{n}$ to $\mathbb{C}\langle X_{1},\dots,X_{n}\rangle$ defined by $\theta(p_{1},p_{2},\dots,p_{n})=\sum_{j=1}^{n}[X_{j},p_{j}]$.
        \item $\sum_{j=1}^{n}X_{j}p_{j}\in\mathrm{ran}(C)$, where $C$ is the symmetrization operator.
        \item $\delta_{k}[\sum_{j=1}^{n}X_{j}p_{j}]=(N+\mathrm{id})[p_{k}]$ for any $k=1,\dots,n$.
    \end{enumerate}
\end{Thm}
\begin{Thm}(\cite[Theorem 2]{v00a})
    We have
    \begin{equation*}
        \ker(\delta)=\mathbb{C}1+\sum_{j=1}^{n}[X_{j},\mathbb{C}\langle X_{1},
    \dots,X_{n}\rangle]=\mathbb{C}1+[\mathbb{C}\langle X_{1},
    \dots,X_{n}\rangle,\mathbb{C}\langle X_{1},
    \dots,X_{n}\rangle]=\ker(C).
    \end{equation*}
\end{Thm}

Hence, Voiculescu established the exact sequence:
    \begin{equation*}
        0\to\mathbb{C}1+[\mathbb{C}\langle X_{1},
    \dots,X_{n}\rangle,\mathbb{C}\langle X_{1},
    \dots,X_{n}\rangle]
    \to\mathbb{C}\langle X_{1},\dots,X_{n}\rangle\xrightarrow{\delta}\mathbb{C}\langle X_{1},\dots,X_{n}\rangle^{n}\xrightarrow{\theta}\mathbb{C}\langle X_{1},\dots,X_{n}\rangle.
    \end{equation*}

Here, we can prove the following theorem:

\begin{Thm}
    For any $q\in B\langle X\rangle$, the following conditions are equivalent:
    \begin{enumerate}
        \item there exists a $p\in B\langle X\rangle$ such that $\delta_{X:B}[p]=q$.
        \item $q\in\ker(\Theta)$, where $\Theta$ is a linear map from $B\langle X\rangle$ to $B\langle X\rangle$ such that $\Theta=\mathrm{id}_{B\langle X\rangle}-\rho$ with $\rho[b_{0}Xb_{1}\cdots Xb_{n}]=b_{1}Xb_{2}\cdots Xb_{n}Xb_{0}$.
    \end{enumerate}
\end{Thm}
\begin{proof}
    (1)$\Rightarrow$(2): Assume condition (1) (namely, there exists a $p\in B\langle X\rangle$ such that $\delta_{X:B}[p]=q$). It suffices to show the statement for each homogeneous element $p$, since $\delta_{X:B}[p]=q$ if and only if $\delta_{X:B}[p(m+1)]=q(m)$ for any homogeneous terms $p(m+1)$ and $q(m)$ of $p$ and $q$, respectively. 
    
    If $p=\sum_{i}b^{(i)}_{0}Xb^{(i)}_{1}\cdots Xb^{(i)}_{n}$, then 
    \begin{align*}
        q&=\delta_{X:B}[p]
        =\sum_{i}\sum_{j=1}^{n_{i}}
        (b^{(i)}_{j+1})'X\cdots (b^{(i)}_{n_{i}})'X(b^{(i)}_{1})'X\cdots X(b^{(i)}_{j})',
    \end{align*}
    where $(b^{(i)}_{j})'=b^{(i)}_{j}$, $1\leq j\leq n_{i}-1$, and $(b^{(i)}_{n_{i}})'=b^{(i)}_{n_{i}}b^{(i)}_{0}$. Here, we have
    \begin{align*}
        \,&\Theta\left[\sum_{j=1}^{n_{i}}(b^{(i)}_{j+1})'X\cdots (b^{(i)}_{n_{i}})'X(b^{(i)}_{1})'X\cdots X(b^{(i)}_{j})'\right]\\
        &=(b^{(i)}_{2})'X(b^{(i)}_{3})'\cdots X(b^{(i)}_{1})'-(b^{(i)}_{3})'X\cdots X(b^{(i)}_{n_{i}})'X(b^{(i)}_{1})'X(b^{(i)}_{2})'\\
        &\quad+(b^{(i)}_{3})'X\cdots X(b^{(i)}_{n_{i}})'X(b^{(i)}_{1})'X(b^{(i)}_{2})'-(b^{(i)}_{4})'X\cdots X(b^{(i)}_{n_{i}})'X(b^{(i)}_{1})'X(b^{(i)}_{2})'X(b^{(i)}_{3})'\\
        &\quad\quad\quad\vdots\\
        &\quad+(b^{(i)}_{n_{i}})'X(b^{(i)}_{1})'\cdots X(b^{(i)}_{n_{i}-1})'-(b^{(i)}_{1})'X(b^{(i)}_{2})'\cdots X(b^{(i)}_{n_{i}})'\\
        &\quad+(b^{(i)}_{1})'X(b^{(i)}_{2})'\cdots X(b^{(i)}_{n_{i}})'
        -(b^{(i)}_{2})'X(b^{(i)}_{3})'\cdots X(b^{(i)}_{1})'\\
        &=0.
    \end{align*}
    Thus, we have $\Theta[q]=0$. This implies condition (2).

    (2)$\Rightarrow$(1): Assume condition (2). Set $q=\sum_{i}\alpha_{i}b^{(i)}_{0}Xb^{(i)}_{1}\cdots Xb^{(i)}_{n_{i}}$, where $\{b^{(i)}_{0}Xb^{(i)}_{1}\cdots Xb^{(i)}_{n_{i}}\}_{i}$ are linearly independent and $\alpha_{i}\in\mathbb{C}\setminus\{0\}$ (if the $b^{(i)}_{j}$ are taken to be a basis of $B$ as vector space for each). By condition (2), we have 
    \begin{equation*}
        \sum_{i}\alpha_{i}b^{(i)}_{0}Xb^{(i)}_{1}\cdots Xb^{(i)}_{n_{i}}
        =\sum_{i}\alpha_{i}b^{(i)}_{1}Xb^{(i)}_{2}\cdots Xb^{(i)}_{n_{i}}Xb^{(i)}_{0}.
    \end{equation*}
    By the uniqueness of expression, we have $\{b^{(i)}_{0}Xb^{(i)}_{1}\cdots Xb^{(i)}_{n_{i}}\}_{i}=\{b^{(i)}_{1}X\cdots Xb^{(i)}_{n_{i}}Xb^{(i)}_{0}\}_{i}$. This implies that all $b_{1}Xb_{2}\cdots Xb_{n}Xb_{0},\dots,b_{n}Xb_{0}\cdots Xb_{n-1}$ fall into $\{b^{(i)}_{0}Xb^{(i)}_{1}\cdots Xb^{(i)}_{n_{i}}\}_{i}$ for any $b_{0}Xb_{1}\cdots Xb_{n}\in \{b^{(i)}_{0}Xb^{(i)}_{1}\cdots Xb^{(i)}_{n_{i}}\}_{i}$. Then, it is easy to see that coefficients $\alpha_{i}$ are constant on each cycle, that is, $q$ is in the form of $q=\sum_{j}\beta_{j}c_{j}$, where $\beta_{j}\in\mathbb{C}$ and $c_{j}$ is a sum of all circular permutation (with respect to non-commutative coefficient $B$) of some monomial $m_{j}\in B\langle X\rangle$. For example, if the $m_{j}$ is $b_{0}Xb_{1}\cdots Xb_{n}$, then $c_{j}=b_{0}Xb_{1}\cdots Xb_{n}+b_{1}Xb_{2}\cdots b_{n}Xb_{0}+\cdots+b_{n}Xb_{0}\cdots b_{n-2}Xb_{n-1}$.
    This means that $\partial_{X:B}[q]=\sigma(\partial_{X:B}[q])$, which is equivalent to condition (1) by Theorem \ref{poic} for $B\langle X\rangle$ (see Remark \ref{bofx}).
\end{proof}
\begin{Rem}\label{remthetavoi}
    The above theorem holds if we can replace $\Theta$ for the operator $\xi\circ\theta$, where $\theta$ is Voiculescu's one, that is, $\theta(b_{0}Xb_{1}\cdots Xb_{n})=[X,b_{0}Xb_{1}\cdots Xb_{n}]$, and $\xi$ is given by $\xi(b_{0}Xb_{1}\cdots Xb_{n})=Xb_{1}\cdots Xb_{n}b_{0}$. The proof is similar to the above discussion.
\end{Rem}

Hence, we have the following corollary:
\begin{Cor}
    The following sequence is exact:
    \begin{equation*}
        0\to B+[B\langle X\rangle,B\langle X\rangle]\to B\langle X\rangle\xrightarrow{\delta_{X:B}}B\langle X\rangle\xrightarrow{\Theta}B\langle X\rangle.
    \end{equation*}
\end{Cor}

Therefore, if we define a linear map $\Theta_{\mathcal{Z}(B^d)}$ from $\mathcal{Z}(B^d)=(\mathcal{Z}(1^{\perp}))\langle z(\theta)\rangle$ to $\mathcal{Z}(B^d)$ similarly to $\Theta$, then we also have the following facts:

\begin{Cor}
    For any $q\in\mathcal{Z}(B^d)$, the following conditions are equivalent:
    \begin{enumerate}
        \item there exists a $p\in\mathcal{Z}(B^d)$ such that $D[p]=q$.
        \item $q\in\ker(\Theta_{\mathcal{Z}(B^d)})$.
    \end{enumerate}
\end{Cor}
\begin{Cor}
    The following sequence is exact:
    \begin{equation*}
        0\to\mathcal{Z}(1^{\perp})+[\mathcal{Z}(B^d),\mathcal{Z}(B^d)]\to\mathcal{Z}(B^d)\xrightarrow{D|_{\mathcal{Z}(B^d)}}\mathcal{Z}(B^d)\xrightarrow{\Theta_{\mathcal{Z}(B^d)}}\mathcal{Z}(B^d).
    \end{equation*}
\end{Cor}

\section{Translation into the context of \cite{kvv14}}

We have studied Voiculescu's fully matricial function theory so far. However, there are other ``non-commutative function'' theories due to various authors from various points of view. In particular, the theory of \cite{kvv14} seems to be rather exhaustive. The purpose of this section is to position Voiculescu's fully matricial function theory in the context of \cite{kvv14} and clarify the relation between the present work and two previous works \cite{kvsv20} and \cite{a20} due to Kaliuzhnyi-Verbovetskyi, Stevenson and Vinnikov and to Augat, respectively. 

In subsections \ref{subsectionncsetncfunc}--\ref{subsectionanalyticity} below, we will review some basic materials from \cite{kvv14} for the reader's convenience.

\subsection{NC sets and nc functions}\label{subsectionncsetncfunc}
Let $\mathcal{R}$ be a commutative ring and $\mathcal{M}$ be a module over $\mathcal{R}$. We write $M_{n}(\mathcal{M})=M_{n}(\mathcal{R})\otimes_{\mathcal{R}}\mathcal{M}$ for any $n\in\mathbb{N}$.

\begin{Def}(\cite[section 2.1]{kvv14})\label{defncfunction}
    Set $\mathcal{M}_{\mathrm{nc}}=\bigsqcup_{n\in\mathbb{N}}M_{n}(\mathcal{M})$, which we call the \textit{non-commutative space (nc space) over $\mathcal{M}$}. A subset $\Omega$ of $\mathcal{M}_{\mathrm{nc}}$ is said to be a \textit{non-commutative set (nc set)} if, with letting $\Omega_{n}=\Omega\cap M_{n}(\mathcal{M})$, the following condition holds:
    \begin{equation*}
        X\in\Omega_{n},\quad Y\in\Omega_{m}
        \quad\Rightarrow\quad
        X\oplus Y\in\Omega_{n+m}.
    \end{equation*}
    
    Let $\Omega\subset\mathcal{M}_{\mathrm{nc}}$ be a non-commutative set and $\mathcal{N}$ be another module over $\mathcal{R}$. Then, a map $f:\Omega\to\mathcal{N}_{\mathrm{nc}}$ is said to be a \textit{non-commutative $\mathcal{N}$-valued function (nc $\mathcal{N}$-valued function)} if the following conditions hold:
    \begin{enumerate}
        \item $f(\Omega_{n})\subset M_{n}(\mathcal{N})$ for any $n\in\mathbb{N}$.
        \item $f(X\oplus Y)=f(X)\oplus f(Y)$ for any $n,m\in\mathbb{N}$, $X\in\Omega_{n}$ and $Y\in\Omega_{m}$.
        \item If $SXS^{-1}\in\Omega_{n}$ for $X\in\Omega_{n}$ and $S\in M_{n}(\mathcal{R})$, then $f(SXS^{-1})=Sf(X)S^{-1}$.
    \end{enumerate}
\end{Def}

\begin{Rem}\label{fullyVSnoncommutative}
    Let $\Omega=(\Omega_{n})_{n\geq1}$ be a fully matricial set. Remark that $\bigsqcup_{n\in\mathbb{N}}\Omega_{n}$ can be seen as a nc set. Then, any fully matricial function $f=(f_{n})_{n\in\mathbb{N}}$ on $\Omega$ can naturally be regarded as a nc function on $\bigsqcup_{n\in\mathbb{N}}\Omega_{n}$ by letting $f(X):=f_{n}(X)$ if $X\in\Omega_{n}$. Thus, it is natural to identify $\Omega=(\Omega_{n})_{n\geq1} = \bigsqcup_{n\in\mathbb{N}}\Omega_{n}$.
\end{Rem}

\begin{Rem}
    The set of conditions to be a fully matricial set is stronger than that to be a nc set (see Definition \ref{deffmBs}). In fact, a nc set need not satisfy the similarity preserving property, i.e., condition (3) in Definition \ref{deffmBs}. Also, condition (2) there means that not only
    \begin{equation*}
        X\in\Omega_{n},\quad Y\in\Omega_{m}
        \quad\Rightarrow\quad
        X\oplus Y\in\Omega_{n+m}.
    \end{equation*}
    but also
    \begin{equation*}
        X\in\Omega_{n},\quad Y\in\Omega_{m}
        \quad\Leftarrow\quad
        X\oplus Y\in\Omega_{n+m}
    \end{equation*}
    for any $X\in M_{n}(\mathcal{M})$ and $Y\in M_{m}(\mathcal{M})$.
\end{Rem}

The notion of multivariable nc functions is also formulated as follows.

\begin{Def}(\cite[section 3.1]{kvv14})\label{defhigerncfunc}
    Let $\mathcal{M}^{(j)}$ and $\mathcal{N}^{(j)}$, $j=0,1,\dots,k$, be modules over $\mathcal{R}$. Let $\Omega^{(j)}$ be a nc set over $\mathcal{M}^{(j)}$ for each $j=0,1,\dots,k$.
    Then, a map $f$ from $\Omega^{(0)}\times\Omega^{(1)}\times\cdots\times\Omega^{(k)}$ to 
    \begin{align*}
    &\mathrm{Hom}_{\mathcal{R}}(\mathcal{N}^{(1)}_{\mathrm{nc}}\otimes_{\mathcal{R}}\mathcal{N}^{(2)}_{\mathrm{nc}}\otimes_{\mathcal{R}}\cdots\otimes_{\mathcal{R}}\mathcal{N}^{(k)}_{\mathrm{nc}},\mathcal{N}^{(0)}_{\mathrm{nc}}):=\\
    &\bigsqcup_{n_{0},n_{1},\dots,n_{k}\in\mathbb{N}}\mathrm{Hom}_{\mathcal{R}}(M_{n_{0},n_{1}}(\mathcal{N}^{(1)})\otimes_{\mathcal{R}}M_{n_{1},n_{2}}(\mathcal{N}^{(2)})\otimes_{\mathcal{R}}\cdots\otimes_{\mathcal{R}}M_{n_{k-1},n_{k}}(\mathcal{N}^{(k)}),M_{n_{0},n_{k}}(\mathcal{N}^{(0)}))
    \end{align*} 
    is said to be a \textit{nc function of order $k$} if the following conditions hold:
    \begin{enumerate}
        \item We have 
        \begin{align*}
            \,&f(\Omega^{(0)}_{n_{0}}\times\Omega^{(1)}_{n_{1}}\times\cdots\times\Omega^{(k)}_{n_{k}})\\
            &\subset \mathrm{Hom}_{\mathcal{R}}(M_{n_{0},n_{1}}(\mathcal{N}^{(1)})\otimes_{\mathcal{R}}M_{n_{1},n_{2}}(\mathcal{N}^{(2)})\otimes_{\mathcal{R}}\cdots\otimes_{\mathcal{R}}M_{n_{k-1},n_{k}}(\mathcal{N}^{(k)}),M_{n_{0},n_{k}}(\mathcal{N}^{(0)}))
        \end{align*}
        for any $n_{0},n_{1},\dots,n_{k}\in\mathbb{N}$.
        \item For any $n_{j},n'_{j},n''_{j}\in\mathbb{N}$, any $X^{(j)}\in\Omega^{(j)}_{n_{j}}$, $X'^{(j)}\in\Omega^{(j)}_{n'_{j}}$, $X''^{(j)}\in\Omega^{(j)}_{n''_{j}}$ and 
        $Z^{(j)}\in M_{n_{j-1},n_{j}}(\mathcal{N}^{(j)})$, $j=0,1,\dots,k$, the following conditions hold:
        \begin{enumerate}
            \item[(2-$0$)] For any $Z'^{(1)}\in M_{n'_{0},n_{1}}(\mathcal{N}^{(1)})$ and $Z''^{(1)}\in M_{n''_{0},n_{1}}(\mathcal{N}^{(1)})$,
            \begin{align*}
                \,&
                f(X'^{(0)}\oplus X''^{(0)};X^{(1)};\dots;X^{(k)})
                \left(
                \left[
                \begin{smallmatrix}
                    Z'^{(1)}\\
                    Z''^{(1)}
                \end{smallmatrix}
                \right]
                ;Z^{(1)};\dots;Z^{(k)}
                \right)\\
                &=
                \begin{bmatrix}
                    f(X'^{(0)};X^{(1)};\dots;X^{(k)})(Z'^{(1)};Z^{(2)};\dots;Z^{(k)})\\
                    f(X''^{(0)};X^{(1)};\dots;X^{(k)})(Z''^{(1)};Z^{(2)};\dots;Z^{(k)})
                \end{bmatrix}. 
            \end{align*}
            \item[(2-$j$)] For any $Z'^{(j)}\in M_{n_{j-1},n'_{j}}(\mathcal{N}^{(j)})$, $Z''^{(j)}\in M_{n_{j-1},n''_{j}}(\mathcal{N}^{(j)})$, $Z'^{(j+1)}\in M_{n'_{j},n_{j+1}}(\mathcal{N}^{(j+1)})$, $Z''^{(j+1)}\in M_{n''_{j},n_{j+1}}(\mathcal{N}^{(j+1)})$,
            \begin{align*}
                \,&
                f(X^{(0)};\dots;X^{(j-1)};X'^{(j)}\oplus X''^{(j)};X^{(j+1)};\dots;X^{(k)})\\
                &\hspace{1cm}
                \left(
                Z^{(1)};\dots;Z^{(j-1)};
                \left[
                \begin{smallmatrix}
                    Z'^{(j)}&Z''^{(j)}
                \end{smallmatrix}
                \right]
                ;
                \left[
                \begin{smallmatrix}
                    Z'^{(j+1)}\\
                    Z''^{(j+1)}
                \end{smallmatrix}
                \right]
                ;Z^{(j+2)};\dots;Z^{(k)}
                \right)\\
                &=
                f(X^{(0)};\dots;X^{(j-1)};X'^{(j)};X^{(j+1)};\dots;X^{(k)})\\
                &\hspace{1cm}
                \left(
                Z^{(1)};\dots;Z^{(j-1)};Z'^{(j)}
                ;Z'^{(j+1)}
                ;Z^{(j+2)};\dots;Z^{(k)}
                \right)\\
                &\quad+
                f(X^{(0)};\dots;X^{(j-1)};X''^{(j)};X^{(j+1)};\dots;X^{(k)})\\
                &\hspace{1cm}
                \left(
                Z^{(1)};\dots;Z^{(j-1)};Z''^{(j)}
                ;Z''^{(j+1)}
                ;Z^{(j+2)};\dots;Z^{(k)}
                \right).
            \end{align*}
            \item[(2-$k$)] For any $Z'^{(k)}\in M_{n_{k-1},n'_{k}}(\mathcal{N}^{(k)})$ and $Z''^{(k)}\in M_{n_{k-1},n''_{k}}(\mathcal{N}^{(k)})$
            \begin{align*}
                \,&f(X^{(0)};\dots;X^{(k-1)};X'^{(k)}\oplus X''^{(k)})
                \left(
                Z^{(1)};\dots;Z^{(k-1)};
                \left[
                \begin{matrix}
                    Z'^{(k)}&Z''^{(k)}
                \end{matrix}
                \right]
                \right)\\
                &=
                \mathrm{Row}
                \Bigl[
                f(X^{(0)};\dots;X^{(k-1)};X'^{(k)})
                \left(
                Z^{(1)};\dots;Z^{(k-1)};Z'^{(k)}
                \right),\\
                &\hspace{3cm}
                f(X^{(0)};\dots;X^{(k-1)};X''^{(k)})
                \left(
                Z^{(1)};\dots;Z^{(k-1)};Z''^{(k)}
                \right)
                \Bigr],
            \end{align*}
            where $\mathrm{Row}[\alpha,\beta]=
            \begin{bmatrix}
                \alpha&\beta
            \end{bmatrix}$.
        \end{enumerate}
        \item For any $n_{j}\in\mathbb{N}$, $X^{(j)}\in\Omega^{(j)}_{n_{j}}$ and 
        $Z^{(j)}\in M_{n_{j-1},n_{j}}(\mathcal{N}^{(j)})$, $j=0,1,\dots,k$, the following conditions hold:
        \begin{enumerate}
            \item[(3-$0$)] For any $S_{0}\in GL_{n_{0}}(\mathcal{R})$,
            \begin{align*}
                \,&
                f(S_{0}X^{(0)}S_{0}^{-1};X^{(1)};\dots;X^{(k)})
                (S_{0}Z^{(1)};Z^{(2)};\dots;Z^{(k)})\\
                &\quad
                =
                S_{0}f(X^{(0)};X^{(1)};\dots;X^{(k)})
                (Z^{(1)};Z^{(2)};\dots;Z^{(k)}).
            \end{align*}
            \item[(3-$j$)] For any $S_{j}\in GL_{n_{j}}(\mathcal{R})$,
            \begin{align*}
                \,&
                f(X^{(0)};\dots;S_{j}X^{(j)}S_{j}^{-1};\dots;X^{(k)})\\
                &\hspace{1cm}
                (Z^{(1)};\dots;Z^{(j)}S_{j}^{-1};S_{j}Z^{(j+1)};Z^{(j+2)};\dots;Z^{(k)})\\
                &\quad
                =
                f(X^{(0)};X^{(1)};\dots;X^{(k)})
                (Z^{(1)};Z^{(2)};\dots;Z^{(k)}).
            \end{align*}
            \item[(3-$k$)] For any $S_{k}\in GL_{n_{k}}(\mathcal{R})$,
            \begin{align*}
                \,&
                f(X^{(0)};\dots;X^{(k-1)};S_{k}X^{(k)}S_{k}^{-1})
                (Z^{(1)};\dots;Z^{(k)}S_{k}^{-1})\\
                &=
                f(X^{(0)};X^{(1)};\dots;X^{(k)})
                (Z^{(1)};Z^{(2)};\dots;Z^{(k)})S_{k}^{-1}.
            \end{align*}
        \end{enumerate}
    \end{enumerate}
    We denote by $\mathcal{T}^{k}(\Omega^{(0)},\dots,\Omega^{(k)};\mathcal{N}^{(0)}_{\mathrm{nc}},\dots,\mathcal{N}^{(k)}_{\mathrm{nc}})$ the space of all above nc functions of order $k$. (Hence, the space $\mathcal{T}^{(0)}(\Omega;\mathcal{N}_{nc})$ of all nc functions of order $0$ from $\Omega$ to $\mathcal{N}_{\mathrm{nc}}$ is the same as that of all nc functions from $\Omega$ to $\mathcal{N}_{\mathrm{nc}}$.)
\end{Def}

\subsection{The (right) nc difference-differential operator $\Delta$}\label{subsectionncdcdfop}
Here we will review the (right) nc difference-differential operator $\Delta$ for nc functions and its higher order version. See \cite[Chapters 2 and 3]{kvv14} for their details. 

In what follows, let $\mathcal{M}$ and $\mathcal{N}$ be modules over a commutative ring $\mathcal{R}$.
Let $\Omega$ be a nc set over $\mathcal{M}$ and $f$ be a nc function from $\Omega$ to $\mathcal{N}_{\mathrm{nc}}$. If $X\in\Omega_{n}$, $Y\in\Omega_{m}$ and $Z\in M_{n,m}(\mathcal{M})$ enjoy 
$\begin{bmatrix}
    X&Z\\
    0&Y
\end{bmatrix}
\in\Omega_{n+m}$, 
then there exists a unique $(\Delta f)(X,Y)(Z)\in M_{n,m}(\mathcal{N})$ such that
\begin{align*}
    f\left(
    \begin{bmatrix}
        X&Z\\
        0&Y
    \end{bmatrix}
    \right)
    =
    \begin{bmatrix}
        f(X)&(\Delta f)(X,Y)(Z)\\
        0&f(Y)
    \end{bmatrix} 
\end{align*}
and that
\begin{align*}
    (\Delta f)(X,Y)(rZ)=r(\Delta f)(X,Y)(Z)\quad
    \mbox{if}\quad
    \begin{bmatrix}
        X&rZ\\
        0&Y
    \end{bmatrix}
    \in \Omega_{n+m}\quad\mbox{with}\quad r\in\mathcal{R}
\end{align*}
(see \cite[Proposition 2.2]{kvv14}).

In general, we cannot define $(\Delta f)(X,Y)(Z)\in M_{n,m}(\mathcal{N})$ for any $X\in\Omega_{n}$, $Y\in\Omega_{m}$ and $Z\in M_{n,m}(\mathcal{M})$, because 
$
\begin{bmatrix}
    X&Z\\
    0&Y
\end{bmatrix}\in\Omega_{n+m}$ is not required for any $X\in\Omega_{n}$, $Y\in\Omega_{m}$ and $Z\in M_{n,m}(\mathcal{M})$. 
Here is a concept concerning nc sets. 

\begin{Def}(\cite[section 2.1]{kvv14})
    Let $\Omega$ be a nc set over $\mathcal{M}$. Then, $\Omega$ is said to be \textit{right admissible} if, for any $X\in\Omega_{n}$, $Y\in\Omega_{m}$ and $Z\in M_{n,m}(\mathcal{M})$, there exists an $r\in GL_{1}(\mathcal{R})$ with 
    $
    \begin{bmatrix}
        X&rZ\\
        0&Y
    \end{bmatrix}
    \in\Omega_{n+m}$.
\end{Def}

\begin{Rem}
    Every open fully matricial set is right admissible (see \cite[Proposition 6.2]{v04}).
\end{Rem}

Assume that $\Omega$ is a right admissible nc set over $\mathcal{M}$ and $f$ be a nc function from $\Omega$ to $\mathcal{N}_{\mathrm{nc}}$. For any $X\in\Omega_{n}$, $Y\in\Omega_{m}$ and $Z\in M_{n,m}(\mathcal{M})$ there exists an $r\in GL_{1}(\mathcal{R})$ with
$
\begin{bmatrix}
    X&rZ\\
    0&Y
\end{bmatrix}
\in\Omega_{n+m}$. Then we can define $(\Delta f)(X,Y)(Z) \in M_{n,m}(\mathcal{N})$ by
\begin{align*}
    (\Delta f)(X,Y)(Z):=r^{-1}(\Delta f)(X,Y)(rZ).
\end{align*}
Then, the mapping $Z \mapsto (\Delta f)(X,Y)(Z)$ is $\mathcal{R}$-homogeneous and additive, that is, 
    \begin{align*}
        (\Delta f)(X,Y)(Z_{1}+rZ_{2})=(\Delta f)(X,Y)(Z_{1})+r(\Delta f)(X,Y)(Z_{2})
    \end{align*}
for any $r\in\mathcal{R}$ and $Z_{1},Z_{2}\in M_{n,m}(\mathcal{M})$. This $\Delta$ is called the \textit{(right) nc difference-differential operator}. See \cite[Propositions 2.4 and 2.6]{kvv14} for details.

We will next review the higher order nc difference-differential operator 
\begin{align*}
    \,&\,_{k,j}\Delta
    :
    \mathcal{T}^{(k)}(\Omega^{(0)},\Omega^{(1)},\dots,\Omega^{(k)};\mathcal{N}^{(0)}_{\mathrm{nc}},\mathcal{N}^{(1)}_{\mathrm{nc}},\dots,\mathcal{N}^{(k)})\\
    &\hspace{1.0cm}\to\mathcal{T}^{(k+1)}(\Omega^{(0)},\Omega^{(1)},\dots,\Omega^{(j)},\Omega^{(j)},\dots,\Omega^{(k)};\mathcal{N}^{(0)}_{\mathrm{nc}},\mathcal{N}^{(1)}_{\mathrm{nc}},\dots,\mathcal{N}^{(j)}_{\mathrm{nc}},\mathcal{M}^{(j)}_{\mathrm{nc}},\mathcal{N}^{(j+1)}_{\mathrm{nc}},\dots,\mathcal{N}^{(k)}_{\mathrm{nc}})
\end{align*}
for any $k\in\mathbb{N}$ and $j=0,1,\dots,k$ from \cite[section 3.2]{kvv14}. Let $\mathcal{M}^{(j)}$, $\mathcal{N}^{(j)}$, $j=0,1,\dots,k+1$, be modules over $\mathcal{R}$. 
Let $\Omega^{(i)}$ be a nc set over $\mathcal{M}^{(i)}$ for each $i=0,1,\dots,k+1$. Assume that $\Omega^{(0)}$ is right admissible. Choose an arbitrary $f\in \mathcal{T}^{(k)}(\Omega^{(0)},\dots,\Omega^{(k)};\mathcal{N}^{(0)}_{\mathrm{nc}},\dots,\mathcal{N}^{(k)}_{\mathrm{nc}})$. Then, for any $X'^{(0)}\in\Omega^{(0)}_{n'_{0}}$, $X''^{(0)}\in\Omega^{(0)}_{n''_{0}}$, $X^{(i)}\in\Omega^{(i)}_{n_{i}}$, $i=1,2,\dots,k$, $Z'^{(1)}\in M_{n'_{0},n_{1}}(\mathcal{N}^{(1)})$, $Z''^{(1)}\in M_{n''_{0},n_{1}}(\mathcal{N}^{(1)})$, $Z^{(j)}\in M_{n_{j-1},n_{j}}(\mathcal{N}^{(j)})$, $j=2,3,\dots,k$ and $Z\in M_{n'_{0},n''_{0}}(\mathcal{M})$ with 
$
\begin{bmatrix}
    X'^{(0)}&rZ\\
    0&X''^{(0)}
\end{bmatrix}\in\Omega^{(0)}_{n'_{0}+n''_{0}}
$ for some $r\in GL_{1}(\mathcal{R})$, there exists a unique 
\begin{equation*}
    (_{k,0}\Delta f)
    \left(X'^{(0)};X''^{(0)};X^{(1)};\dots;X^{(k)}\right)
    \left(rZ;Z''^{(1)};Z^{(2)};\dots;Z^{(k)}
    \right)\in M_{n'_{0},n_{k}}(\mathcal{N}),
\end{equation*}
which is independent of $Z'^{(1)}$,
such that
\begin{align*}
   \,&
   f
   \left(
   \begin{bmatrix}
      X'^{(0)}&rZ\\
      0&X''^{(0)}
   \end{bmatrix};X^{(1)};\dots;X^{(k)}\right)
   \left(
   \begin{bmatrix}
      Z'^{(1)}\\
      Z''^{(1)}
   \end{bmatrix}
   ;Z^{(2)};\dots;Z^{(k)}
   \right)\\
   &
   =\mathrm{Col}\Biggl[
   f \left(X'^{(0)};X^{(1)};\dots;X^{(k)}\right)\left(Z'^{(1)};Z^{(2)};\dots;Z^{(k)}\right)\\
   &\hspace{1cm}+
   (_{k,0}\Delta f)
   \left(X'^{(0)};X''^{(0)};X^{(1)};\dots;X^{(k)}\right)
   \left(rZ;Z''^{(1)};Z^{(2)};\dots;Z^{(k)}
   \right),\\
   &\hspace{1cm}f\left(X''^{(0)};X^{(1)};\dots;X^{(k)}\right)\left(Z''^{(1)};Z^{(2)};\dots;Z^{(k)}\right)
   \Biggr],
\end{align*}
where $\mathrm{Col}[\alpha,\beta]=
\begin{bmatrix}
    \alpha\\
    \beta
\end{bmatrix}$. Letting 
\begin{align*}
    \,&
    (_{k,0}\Delta f)\left(X'^{(0)};X''^{(0)};X^{(1)};\dots;X^{(k)}\right)\left(Z;Z''^{(1)};Z^{(2)};\dots;Z^{(k)}\right)\\
    &\quad:=r^{-1}(_{k,0}\Delta f)\left(X'^{(0)};X''^{(0)};X^{(1)};\dots;X^{(k)}\right)\left(rZ;Z''^{(1)};Z^{(2)};\dots;Z^{(k)}\right), 
\end{align*}
we have
\begin{align*}
    \,_{k,0}\Delta f\in\mathcal{T}^{(k+1)}(\Omega^{(0)},\Omega^{(0)},\Omega^{(1)},\dots,\Omega^{(k)};\mathcal{N}^{(0)}_{\mathrm{nc}},\mathcal{M}^{(0)}_{\mathrm{nc}},\mathcal{N}^{(1)}_{\mathrm{nc}},\dots,\mathcal{N}^{(k)}_{\mathrm{nc}}).
\end{align*}

Assume that $\Omega^{(j)}$ is right admissible for some $j\in\{1,\dots,k-1\}$. Choose an arbitrary $f$ of $\mathcal{T}^{(k)}(\Omega^{(0)},\dots,\Omega^{(k)};\mathcal{N}^{(0)}_{\mathrm{nc}},\dots,\mathcal{N}^{(k)}_{\mathrm{nc}})$. Then, for any $X^{(i)}\in\Omega^{(i)}_{n_{i}}$, $i=1,\dots,j-1,j+1,\dots,k$, $X'^{(j)}\in\Omega^{(j)}_{n'_{j}}$, $X''^{(j)}\in\Omega^{(j)}_{n''_{j}}$, $X^{(j+1)}\in\Omega^{(j+1)}_{n_{j+1}}$, $Z^{(m)}\in M_{n_{m-1},n_{m}}(\mathcal{N}^{(m)})$, $m=1,\dots,j-1,j+2,\dots,k$, $Z'^{(j)}\in M_{n_{j-1},n'_{j}}(\mathcal{N}^{(j)})$, $Z''^{(j)}\in M_{n_{j-1},n''_{j}}(\mathcal{N}^{(j)})$, $Z'^{(j+1)}\in M_{n'_{j},n_{j+1}}(\mathcal{N}^{(j+1)})$, $Z''^{(j+1)}\in M_{n''_{j},n_{j+1}}(\mathcal{N}^{(j+1)})$, $Z\in M_{n'_{j},n''_{j}}(\mathcal{M}^{(j)})$ with 
$\begin{bmatrix}
    X'^{(j)}&rZ\\
    0&X''^{(j)}
\end{bmatrix}\in\Omega^{(j)}_{n'_{j}+n''_{j}}$ for some $r\in GL_{1}(\mathcal{R})$, there exists a unique 
\begin{align*}
    \,&
    (_{k,j}\Delta f)
    \left(
    X^{(0)};\dots;X^{(j-1)};X'^{(j)};X''^{(j)};X^{(j+1)};\dots;X^{(k)}
    \right)\\
    &\quad
    \left(
    Z^{(1)};\dots;Z^{(j-1)};Z'^{(j)};rZ;Z''^{(j+1)};Z^{(j+2)};\dots;Z^{(k)}
    \right)\in M_{n_{0},n_{k}}(\mathcal{N}^{(0)}),
\end{align*}
which is independent of $Z''^{(j)}$ and $Z'^{(j+1)}$, such that 
\begin{align*}
    \,&
    f
    \left(
    X^{(0)};\dots;X^{(j-1)};
    \begin{bmatrix}
        X'^{(j)}&rZ\\
        0&X''^{(j)}
    \end{bmatrix}
    ;X^{(j+1)};\dots;X^{(k)}
    \right)\\
    &\hspace{1cm}
    \left(
    Z^{(1)};\dots;Z^{(j-1)};
    \begin{bmatrix}
        Z'^{(j)}&Z''^{(j)}
    \end{bmatrix}
    ;
    \begin{bmatrix}
        Z'^{(j+1)}\\
        Z''^{(j+1)}
    \end{bmatrix}
    ;Z^{(j+2)};\dots;Z^{(k)}
    \right)\\
    &=
    f(X^{(0)};\dots;X^{(j-1)};X'^{(j)};X^{(j+1)};\dots;X^{(k)})(Z^{(1)};\dots;Z^{(j-1)};Z'^{(j)};Z'^{(j+1)};Z^{(j+2)};\dots;Z^{(k)})\\
    &\quad
    +(_{k,j}\Delta f)
    \left(
    X^{(0)};\dots;X^{(j-1)};X'^{(j)};X''^{(j)};X^{(j+1)};\dots;X^{(k)}
    \right)\\
    &\hspace{3cm}
    \left(
    Z^{(1)};\dots;Z^{(j-1)};Z'^{(j)};rZ;Z''^{(j+1)};Z^{(j+2)};\dots;Z^{(k)}
    \right)\\
    &\quad
    +f(X^{(0)};\dots;X^{(j-1)};X''^{(j)};X^{(j+1)};\dots;X^{(k)})
    (Z^{(1)};\dots;Z^{(j-2)};Z''^{(j)};Z''^{(j+1)};Z^{(j+2)};\dots;Z^{(k)}).
\end{align*}
Here, we define 
\begin{align*}
    \,&
    (_{k,j}\Delta f)(X^{(0)};\dots;X^{(j-1)};X'^{(j)};X''^{(j+1)};X^{(j+2)};\dots;X^{(k)})\\
    &\hspace{2cm}
    \left(
    Z^{(1)};\dots;Z^{(j-1)};Z'^{(j)};Z;Z''^{(j+1)};Z^{(j+2)};\dots;Z^{(k)}
    \right)\\
    &=r^{-1}
    (_{k,j}\Delta f)(X^{(0)};\dots;X^{(j-1)};X'^{(j)};X''^{(j+1)};X^{(j+2)};\dots;X^{(k)})\\
    &\hspace{3cm}
    \left(
    Z^{(1)};\dots;Z^{(j-1)};Z'^{(j)};rZ;Z''^{(j+1)};Z^{(j+2)};\dots;Z^{(k)}
    \right).
\end{align*}
Then, we have
\begin{align*}
    \,_{k,j}\Delta f\in\mathcal{T}^{(k+1)}(\Omega^{(0)},\dots,\Omega^{(j)},\Omega^{(j)},\dots,\Omega^{(k)};\mathcal{N}^{(0)}_{\mathrm{nc}},\mathcal{N}^{(1)}_{\mathrm{nc}},\dots,\mathcal{N}^{(j)}_{\mathrm{nc}},\mathcal{M}^{(j)}_{\mathrm{nc}},\mathcal{N}^{(j+1)}_{\mathrm{nc}},\dots,\mathcal{N}^{(k)}_{\mathrm{nc}}).
\end{align*}

Assume that $\Omega^{(k)}$ is right admissible. Choose an arbitrary $f\in\mathcal{T}^{(k)}(\Omega^{(0)},\dots,\Omega^{(k)};\mathcal{N}^{(0)}_{\mathrm{nc}},\dots,\mathcal{N}^{(k)}_{\mathrm{nc}})$. Then, for any $X^{(i)}\in\Omega^{(i)}_{n_{i}}$, $i=0,1,\dots,k-1$, $X'^{(k)}\in\Omega^{(k)}_{n'_{k}}$, $X''^{(k)}\in\Omega^{(k)}_{n''_{k}}$, $Z^{(j)}\in M_{n_{j-1},n_{j}}(\mathcal{N}^{(j)})$, $j=1,2,\dots,k-1$ $Z'^{(k)}\in M_{n_{k-1},n'_{k}}(\mathcal{N}^{(k)})$, $Z''^{(k)}\in M_{n_{k-1},n''_{k}}(\mathcal{N}^{(k)})$, $Z\in M_{n'_{k},n''_{k}}(\mathcal{M})$ with 
\[
\begin{bmatrix}
    X'^{(k)}&rZ\\
    0&X''^{(k)}
\end{bmatrix}\in\Omega_{n'_{k}+n''_{k}}
\mbox{ for some }r\in GL_{1}(\mathcal{R}),
\]
there exists a unique 
\begin{align*}
    (\,_{k,k}\Delta f)
    \left(
    X^{(0)};\dots;X^{(k-1)};X'^{(k)};X''^{(k)}
    \right)
    \left(
    Z^{(1)};\dots;Z^{(k-1)};Z'^{(k)};rZ
    \right)\in M_{n_{0},n_{k}}(\mathcal{N}^{(0)}),
\end{align*}
which is independent of $Z''^{(k)}$, such that
\begin{align*}
    \,&
    f\left(
    X^{(0)};\dots;X^{(k-1)};
    \begin{bmatrix}
        X'^{(k)}&rZ\\
        0&X''^{(k)}
    \end{bmatrix}
    \right)
    \left(
    Z^{(1)};\dots;Z^{(k-1)};
    \begin{bmatrix}
        Z'^{(k)}&Z''^{(k)}
    \end{bmatrix}   
    \right)\\
    &=
    \mathrm{Row}
    \Biggl[
    f(X^{(0)};\dots;X^{(k-1)};X'^{(k)})(Z^{(1)};\dots;Z^{(k-1)};Z'^{(k)}),\\
    &\hspace{2cm}(\,_{k,k}\Delta f)
    \left(
    X^{(0)};\dots;X^{(k-1)};X'^{(k)};X''^{(k)}
    \right)
    \left(
    Z^{(1)};\dots;Z^{(k-1)};Z'^{(k)};rZ
    \right)\\
    &\hspace{4.5cm}+f(X^{(0)};\dots;X^{(k-1)};X''^{(k)})(Z^{(1)};\dots;Z^{(k-1)};Z''^{(k)})
    \Biggr],
\end{align*}
where $\mathrm{Row}[\alpha,\beta]=
\begin{bmatrix}
    \alpha&\beta
\end{bmatrix}$. Here, we define
\begin{align*}
    \,&
    (\,_{k,k}\Delta f)
    \left(
    X^{(0)};\dots;X^{(k-1)};X'^{(k)};X''^{(k)}
    \right)
    \left(
    Z^{(1)};\dots;Z^{(k-1)};Z'^{(k)};Z
    \right)\\
    &:=
    r^{-1}
    (\,_{k,k}\Delta f)
    \left(
    X^{(0)};\dots;X^{(k-1)};X'^{(k)};X''^{(k)}
    \right)
    \left(
    Z^{(1)};\dots;Z^{(k-1)};Z'^{(k)};rZ
    \right).
\end{align*}
Then, we have
\begin{align*}
    \,_{k,k}\Delta f\in \mathcal{T}^{(k+1)}(\Omega^{(0)},\dots,\Omega^{(k-1)},\Omega^{(k)},\Omega^{(k)};\mathcal{N}^{(0)}_{\mathrm{nc}},\dots,\mathcal{N}^{(k)}_{\mathrm{nc}},\mathcal{M}^{(k)}_{\mathrm{nc}}).
\end{align*}
See \cite[Propositions 3.7-3.9, Theorem 3.10 and Remark 3.18]{kvv14} for details.

\subsection{The analyticity of nc functions}\label{subsectionanalyticity}
Voiculescu usually assumed the analyticity assumption in his studies on fully matricial functions. In fact, he used analyticity in order to construct his derivation-comultiplication $\partial$. However, it is known, in the context of \cite{kvv14}, that analyticity is equivalent to just continuity, or even local boundedness. See \cite[Chapter 7]{kvv14}.

Here are two definitions. 

\begin{Def}(\cite[section 7.1]{kvv14})
    Let $\mathcal{W}$ be a Banach space over $\mathbb{C}$. An \textit{admissible system of matrix norms over $\mathcal{W}$} is a sequence of norms $\|\cdot\|_{n}$ on $M_{n}(\mathcal{W})$ for $n\in\mathbb{N}$ such that
    \begin{enumerate}
        \item For any $n,m\in\mathbb{N}$ there exist constants $C_{1}(n,m),C'_{1}(n,m)>0$ such that 
        \begin{align*}
            C_{1}(n,m)^{-1}\mathrm{max}\{\|X\|_{n},\|Y\|_{m}\}\leq\|X\oplus Y\|_{n+m}\leq C'_{1}(n,m)\mathrm{max}\{\|X\|_{n},\|Y\|_{m}\}
        \end{align*}
        for all $X\in M_{n}(\mathcal{W})$ and $Y\in M_{m}(\mathcal{W})$.
        \item For any $n\in\mathbb{N}$ there exists $C_{2}(n)>0$ such that 
        \begin{align*}
            \|SXT\|_{n}\leq C_{2}(n)\|S\|\|X\|_{n}\|T\|
        \end{align*}
        for all $X\in M_{n}(W)$ and $S,T\in M_{n}(\mathbb{C})$.
    \end{enumerate}
\end{Def}

\begin{Def}(\cite[section 7.1]{kvv14})
    Let $\mathcal{V}$ and $\mathcal{W}$ be Banach spaces over $\mathbb{C}$ with admissible systems of matrix norms over $\mathcal{V}$ and over $\mathcal{W}$, respectively. Let $\Omega$ be an open nc set over $\mathcal{V}$ (that is, $\Omega_{n}=\Omega\cap M_{n}(\mathcal{V})$ is open for each $n\in\mathbb{N}$). Then, a nc function $f$ from $\Omega$ to $\mathcal{W}_{\mathrm{nc}}$ is \textit{locally bounded} if $f|_{\Omega_{n}}$ is locally bounded for each $n\in\mathbb{N}$, that is, for any $Y\in\Omega_{n}$ there exists $\delta_{n}>0$ such that $f$ is bounded on $\{X\in \Omega_{n}\,|\,\|X-Y\|_{n}<\delta_{n}\}$.
\end{Def}

The next fact given in \cite{kvv14} seems important (but missing) in Voiculescu's studies.

\begin{Thm}(\cite[Corollary 7.6]{kvv14})
    Let $\mathcal{V}$ and $\mathcal{W}$ be Banach spaces over $\mathbb{C}$ with admissible systems of matrix norms over $\mathcal{V}$ and over $\mathcal{W}$, respectively. Let $\Omega$ be an open nc set over $\mathcal{V}$. Then, a nc function $f$ from $\Omega$ to $\mathcal{N}$ is locally bounded if and only if $f$ is continuous (that is, $f|_{\Omega_{n}}$ is continuous for each $n\in\mathbb{N}$) if and only if $f$ is Fr\'{e}chet differentiable (that is, $f_{\Omega_{n}}$ is Fr\'{e}chet differentiable for each $n\in\mathbb{N}$) if and only if $f$ is analytic (that is, $f|_{\Omega_{n}}$ is analytic for each $n\in\mathbb{N}$). 
\end{Thm}

\begin{Rem}
    In Voiculescu's studies on analytic fully matricial functions, he usually assumed that $\mathcal{V}$ is an operator system (or a unital C*-algebra) and $\mathcal{W}$ is $\mathbb{C}$. Thus, with the help of Remark \ref{fullyVSnoncommutative}, all the statements on analytic fully matricial $\mathbb{C}$-valued functions in Voiculescu's studies as well as in the previous sections can be applied to any just locally bounded ones.
\end{Rem}

\subsection{Voiculescu's fully matricial function theory in the context of \cite{kvv14}}\label{subsectionVderivationNC}
In this subsection, we will interpret Voiculescu's fully matricial function theory in the context of \cite{kvv14}. 

\subsubsection{NC functions vs fully matricial functions}
Consider the case when $\mathcal{R}=\mathbb{C}$, that is, $\mathcal{M}$, $\mathcal{M}^{(j)}$, $\mathcal{N}$, $\mathcal{N}^{(j)}$ are vector spaces over $\mathbb{C}$. We have
\begin{align*}
    \,&
    \mathrm{Hom}_{\mathbb{C}}
    \left(
    M_{n_{0},n_{1}}(\mathcal{N}^{(1)})\otimes M_{n_{1},n_{2}}(\mathcal{N}^{(2)})\otimes\cdots\otimes M_{n_{k-1},n_{k}}(\mathcal{N}^{(k)})
    ,
    M_{n_{0},n_{k}}(\mathcal{N}^{(0)})
    \right)\\
    &\simeq
    \mathrm{Hom}_{\mathbb{C}}
    \left(
    M_{n_{0},n_{1}}(\mathbb{C})\otimes M_{n_{1},n_{2}}(\mathbb{C})\otimes\cdots\otimes M_{n_{k-1},n_{k}}(\mathbb{C})\otimes\mathcal{N}^{\otimes},M_{n_{0},n_{k}}(\mathbb{C})\otimes\mathcal{N}^{(0)}
    \right),
\end{align*}
where $\mathcal{N}^{\otimes}=\mathcal{N}^{(1)}\otimes\mathcal{N}^{\otimes2}\otimes\cdots\otimes\mathcal{N}^{(k)}$. If $\mathcal{N}^{(j)}\simeq\mathbb{C}$ for any $j=1,2,\dots,k$, that is, $\mathcal{N}^{\otimes}\simeq\mathbb{C}$, then we have the isomorphism
\begin{align*}
    \,&
    \iota_{n_{0};\dots;n_{k}}
    :
    \mathrm{Hom}_{\mathbb{C}}
    \left(
    M_{n_{0},n_{1}}(\mathbb{C})\otimes M_{n_{1},n_{2}}(\mathbb{C})\otimes\cdots\otimes M_{n_{k-1},n_{k}}(\mathbb{C}),M_{n_{0},n_{k}}(\mathbb{C})\otimes\mathcal{N}^{(0)}
    \right)\\
    &\hspace{1cm}\quad\to
    M_{n_{0}}(\mathbb{C})\otimes M_{n_{1}}(\mathbb{C})\otimes\cdots\otimes M_{n_{k}}(\mathbb{C})\otimes \mathcal{N}^{(0)}
\end{align*}
given by
\begin{align*}
    (\iota_{n_{0};\dots;n_{k}})^{-1}(A^{(0)}\otimes\cdots\otimes A^{(k)}\otimes V)(Z^{(1)};\dots;Z^{(k)})
    =A^{(0)}Z^{(1)}A^{(1)}\cdots A^{(k-1)}Z^{(k)}A^{(k)}\otimes V
\end{align*}
for any $A^{(j)}\in M_{n_{j}}(\mathbb{C})$, $j=0,1,\dots,k$, any $Z^{(i)}\in M_{n_{i-1},n_{i}}(\mathbb{C})$, $i=1,2,\dots,k$, and any $V\in\mathcal{N}^{(0)}$.
Hence, we can regard each element $f\in\mathcal{T}^{(k)}(\Omega^{(0)},\Omega^{(1)},\dots,\Omega^{(k)};\mathcal{N}^{(0)},\mathbb{C}_{\mathrm{nc}},\dots,\mathbb{C}_{\mathrm{nc}})$ as a map
\begin{align*}
    \,&
    \Omega^{(0)}_{n_{0}}\times\Omega^{(1)}_{n_{1}}\times\cdots\times\Omega^{(k)}_{n_{k}}
    \to 
    M_{n_{0}}(\mathbb{C})\otimes M_{n_{1}}(\mathbb{C})\otimes\cdots\otimes M_{n_{k}}(\mathbb{C})\otimes \mathcal{N}^{(0)},
\end{align*}
via the above isomorphisms, for each $n_{0},n_{1},\dots,n_{k}\in\mathbb{N}$. We denote by $\rho^{[k]}(f)_{n_{0};n_{1};\dots;n_{k}}$ the resulting map, that is,
\begin{align*}
    \rho^{[k]}(f)_{n_{0};n_{1};\dots;n_{k}}=\iota_{n_{0};\dots;n_{k}}\circ f.
\end{align*}

\begin{Def}
    Let $\Xi^{(j)}$ be a nc set over $\mathcal{M}^{(j)}$ for each $j=1,\dots,k$. A family of functions $f=(f_{n_{1};\dots;n_{k}})_{n_{1},\dots,n_{k}\in\mathbb{N}}$ is a \textit{$k$-variable tensorial nc $\mathcal{N}^{(0)}$-valued function} on $\Xi^{(1)}\times\cdots\times\Xi^{(k)}$ if $f$ satisfies the following conditions: 
    \begin{enumerate}
        \item $f_{n_{1};\dots;n_{k}}(\Xi^{(1)}_{n_{1}};\dots;\Xi^{(k)}_{n_{k}})
        \subset
        M_{n_{1}}(\mathbb{C})\otimes\cdots\otimes M_{n_{k}}(\mathbb{C})\otimes\mathcal{N}^{(0)}$ for each $n_{1},\dots,n_{k}\in\mathbb{N}$.
        \item For any $j=1,2,\dots,k$, $n_{1},\dots,n_{j-1},n_{j+1},\dots,n_{k}\in\mathbb{N}$ and $X^{(i)}\in\Xi^{(i)}_{n_{i}}$, $i=1,2,\dots,j-1,j+1,\dots,k$, 
        \begin{align*}
            \left(
            f_{n_{1};\dots;n_{j-1};n;n_{j+1};\dots;n_{k}}
            (X^{(1)};\dots;X^{(j-1)};\cdot;X^{(j+1)};\dots;X^{(k)})
            \right)_{n\in\mathbb{N}}
        \end{align*}
        is a nc 
        \begin{align*}
            M_{n_{1}}(\mathbb{C})\otimes\cdots\otimes M_{n_{j-1}}(\mathbb{C})\otimes M_{n_{j+1}}(\mathbb{C})\otimes\cdots\otimes M_{n_{k}}(\mathbb{C})\otimes\mathcal{N}^{(0)}. 
        \end{align*}
        -valued function on $\Xi^{(j)}$.
    \end{enumerate}
    We denote by $\mathcal{F}^{(k)}(\Xi^{(1)};\dots;\Xi^{(k)}:\mathcal{N}^{(0)})$ the space of all $k$-variable tensorial nc $\mathcal{N}^{(0)}$-valued functions on $\Xi^{(1)}\times\cdots\times\Xi^{(k)}$. 
\end{Def}
\begin{Rem}
    A $1$-variable tensorial nc function is nothing but a nc function (of order $0$). Also, if $\Xi^{(j)}$, $j=1,2,\dots,k$, are fully matricial, then the above definition is exactly the definition of Voiculescu's $k$-variable fully matricial functions (see Definition \ref{defkariablefuuly}). In particular, if $\Xi$ is an open fully matricial $G$-set for some operator system $G$ and $\mathcal{N}=\mathbb{C}$, then $A(\Omega;\dots;\Omega)\subset\mathcal{F}^{(k)}(\Omega;\dots;\Omega:\mathbb{C})$ for any $k\in\mathbb{N}$.
\end{Rem}

\begin{Prop}\label{propcorrespondence}
    The above $\rho^{[k]}(f)=(\rho(f)_{n_{0};n_{1};\dots;n_{k}})_{n_{0},n_{1};\dots,n_{k}\in\mathbb{N}}$ is a $(k+1)$-variable tensorial nc $\mathcal{N}^{(0)}$-valued function. Moreover, $\rho^{[k]}$ defines a linear isomorphism 
    \[
    \mathcal{T}^{(k)}(\Omega^{(0)},\Omega^{(1)},\dots,\Omega^{(k)};\mathcal{N}^{(0)},\mathbb{C}_{\mathrm{nc}},\dots,\mathbb{C}_{\mathrm{nc}})
    \to
    \mathcal{F}^{(k+1)}(\Omega^{(0)};\Omega^{(1)};\dots;\Omega^{(k)}:\mathcal{N}^{(0)}).
    \]
\end{Prop}
\begin{proof}
    Choose an arbitrary $f\in\mathcal{T}^{(k)}(\Omega^{(0)},\Omega^{(1)},\dots,\Omega^{(k)};\mathcal{N}^{(0)},\mathbb{C}_{\mathrm{nc}},\dots,\mathbb{C}_{\mathrm{nc}})$ and fix arbitrary $X^{(i)}\in\Omega^{(i)}_{n_{i}}$, $i=0,1,2,\dots,k$.  
We will confirm that
\begin{align*}
    \left(
    \rho^{[k]}(f)_{n_{0};\dots;n_{j-1};n;n_{j+1}\dots;n_{k}}(X^{(0)};\dots;X^{(j-1)};\cdot;X^{(j+1)};\dots;X^{(k)})
    \right)_{n\in\mathbb{N}}
\end{align*}
is a $M_{n_{0}}(\mathbb{C})\otimes\cdots\otimes M_{n_{j-1}}(\mathbb{C})\otimes M_{n_{j+1}}(\mathbb{C})\otimes M_{n_{k}}(\mathbb{C})\otimes\mathcal{N}^{(0)}$-valued nc function on $\Omega^{(j)}$ for any $j=1,2,\dots,k-1$. (The confirmations for
\begin{align*}
    \left(
    \rho^{[k]}(f)_{n;n_{1};\dots;n_{k}}(\cdot;X^{(1)};\dots;X^{(k)})
    \right)_{n\in\mathbb{N}}
\end{align*}
and
\begin{align*}
    \left(
    \rho^{[k]}(f)_{n_{0};\dots;n_{k-1};n}(X^{0};\dots;X^{(k-1)};\cdot)
    \right)_{n\in\mathbb{N}}
\end{align*}
are similar to the discussion below.) It is clear that 
\begin{align*}
    \,&
    \rho^{[k]}(f)_{n_{0};\dots;n_{j-1};n;n_{j+1;}\dots;n_{k}}(X^{(0)};\dots;X^{(j-1)};\Omega^{(j)};X^{(j+1)};\dots;X^{(k)})\\
    &\quad\subset M_{n}(\mathbb{C})\otimes M_{n_{1}}(\mathbb{C})\otimes\cdots\otimes M_{n_{k}}(\mathbb{C})\otimes\mathcal{N}^{(0)}
\end{align*}
 by construction.

\underline{Direct sum preserving property}: Choose arbitrary $X'\in\Omega^{(j)}_{n'}$ and $X''\in\Omega^{(j)}_{n''}$. By definition, the $k$-linear map $f(X^{(0)};\dots;X^{(j-1)};X'\oplus X'';X^{(j+1)};\dots;X^{(k)})$ from $M_{n_{0},n_{1}}(\mathbb{C})\times\cdots\times M_{n_{j-2},n_{j-1}}(\mathbb{C})\times M_{n_{j-1},n'+n''}(\mathbb{C})\times M_{n'+n'',n_{j+1}}(\mathbb{C})\times M_{n_{j+1},n_{j+2}}(\mathbb{C})\times\cdots\times M_{n_{k-1},n_{k}}(\mathbb{C})$ to $M_{n_{0},n_{k}}(\mathbb{C})\otimes\mathcal{N}^{(0)}$ satisfies
\begin{align*}
    \,&
    f(X^{(0)};\dots;X^{(j-1)};X'\oplus X'';X^{(j+1)};\dots;X^{(k)})\\
    &\hspace{2cm}
    \left(
    Z^{(1)};\dots;Z^{(j-1)};
    \begin{bmatrix}
        Z'^{(j)}&Z''^{(j)}
    \end{bmatrix}
    ;
    \begin{bmatrix}
        Z'^{(j+1)}\\
        Z''^{(j+1)}
    \end{bmatrix}
    ;
    Z^{(j+2)};\dots;Z^{(k)}
    \right)\\
    &=
    f(X^{(0)};\dots;X^{(j-1)};X';X^{(j+1)};\dots;X^{(k)})
    (Z^{(1)};\dots;Z^{(j-1)};Z'^{(j)};Z'^{(j+1)};Z^{(j+2)};\dots;Z^{(k)})\\
    &\quad+
    f(X^{(0)};\dots;X^{(j-1)};X'';X^{(j+1)};\dots;X^{(k)})
    (Z^{(1)};\dots;Z^{(j-1)};Z''^{(j)};Z''^{(j+1)};Z^{(j+2)};\dots;Z^{(k)})
\end{align*}
for any $Z^{(i)}\in M_{n_{i-1},n_{i}}(\mathbb{C})$, $Z'^{(j)}\in M_{n_{j-1},n'_{j}}(\mathbb{C})$, $Z''^{(j)}\in M_{n_{j-1},n''_{j}}(\mathbb{C})$, $Z'^{(j+1)}\in M_{n'_{j},n_{j+1}}(\mathbb{C})$, $Z''^{(j+1)}\in M_{n''_{j},n_{j+1}}(\mathbb{C})$, $i=1,2,\dots,j-1,j+2,\dots,k$.
Via the isomorphism 
    \begin{align*}
    \,&
    \mathrm{Hom}_{\mathbb{C}}
    \left(
    M_{n'+n'',n_{1}}(\mathbb{C})\otimes M_{n_{1},n_{2}}(\mathbb{C})\otimes\cdots\otimes M_{n_{k-1},n_{k}}(\mathbb{C}),M_{n'+n'',n_{k}}(\mathbb{C})\otimes\mathcal{N}^{(0)}
    \right)\\
    &\simeq
    M_{n'+n''}(\mathbb{C})\otimes M_{n_{1}}(\mathbb{C})\otimes\cdots\otimes M_{n_{k}}(\mathbb{C})\otimes \mathcal{N}^{(0)},
\end{align*}
there exist elements of $\mathcal{N}^{(0)}$
\begin{align*}
    V(s(0),t(0);\dots;s(j-1),t(j-1);s,t;s(j+1),t(j+1);\dots;s(k),t(k)),
\end{align*}
\begin{align*}
    V'(s(0),t(0);\dots;s(j-1),t(j-1);s',t';s(j+1),t(j+1);\dots;s(k),t(k))
\end{align*}
and
\begin{align*}
    V''(s(0),t(0);\dots;s(j-1),t(j-1);s'',t'';s(j+1),t(j+1);\dots;s(k),t(k)),
\end{align*}
where $1\leq s(p),t(p)\leq n_{p}$, $1\leq s,t\leq n'+n''$, $1\leq s',t'\leq n'$, $1\leq s'',t''\leq n''$, $p=0,1,\dots,j-1,j+1,\dots,k$,
such that
\begin{align*}
    \,&
    f(X^{(0)};\dots;X^{(j-1)};X'\oplus X'';X^{(j+1)};\dots;X^{(k)})\\
    &=
    \sum_{\substack{1\leq s,t\leq n'+n''\\1\leq s(p),t(p)\leq n_{p}\\p\not=j}}e^{(n_{0})}_{s(0),t(0)}\otimes\cdots\otimes e^{(n_{j-1})}_{s(j-1),t(j-1)}\otimes e^{(n'+n'')}_{s,t}\otimes e^{(n_{j+1})}_{s(j+1),t(j+1)}\otimes\cdots\otimes e^{(k)}_{s(k),t(k)}\\
    &\hspace{2cm}\otimes V(s(0),t(0);\dots;s(j-1),t(j-1);s,t;s(j+1),t(j+1);\dots;s(k),t(k)),
\end{align*}
\begin{align*}
    \,&
    f(X^{(0)};\dots;X^{(j-1)};X';X^{(j+1)};\dots;X^{(k)})\\
    &=
    \sum_{\substack{1\leq s',t'\leq n'\\1\leq s(p),t(p)\leq n_{p}\\p\not=j}}e^{(n_{0})}_{s(0),t(0)}\otimes\cdots\otimes e^{(n_{j-1})}_{s(j-1),t(j-1)}\otimes e^{(n')}_{s',t'}\otimes e^{(n_{j+1})}_{s(j+1),t(j+1)}\otimes\cdots\otimes e^{(k)}_{s(k),t(k)}\\
    &\hspace{2cm}\otimes V'(s(0),t(0);\dots;s(j-1),t(j-1);s',t';s(j+1),t(j+1);\dots;s(k),t(k))
\end{align*}
and 
\begin{align*}
    \,&
    f(X^{(0)};\dots;X^{(j-1)};X'';X^{(j+1)};\dots;X^{(k)})\\
    &=
    \sum_{\substack{1\leq s'',t''\leq n''\\1\leq s(p),t(p)\leq n_{p}\\p\not=j}}e^{(n_{0})}_{s(0),t(0)}\otimes\cdots\otimes e^{(n_{j-1})}_{s(j-1),t(j-1)}\otimes e^{(n'')}_{s'',t''}\otimes e^{(n_{j+1})}_{s(j+1),t(j+1)}\otimes\cdots\otimes e^{(k)}_{s(k),t(k)}\\
    &\hspace{2cm}\otimes V''(s(0),t(0);\dots;s(j-1),t(j-1);s'',t'';s(j+1),t(j+1);\dots;s(k),t(k)).
\end{align*}
Plugging in matrix units for  $Z'^{(j)},Z''^{(j)},Z'^{(j+1)},Z''^{(j+1)},Z^{(p)}$, $p=1,2,\dots,j-1,j+2,\dots,k$, we see that
\begin{align*}
    \,&
    V(s(0),t(0);\dots;s,t;\dots;s(k),t(k))\\
    &=
    \begin{cases}
        V'(s(0),t(0);\dots;s,t;\dots;s(k),t(k))\mbox{ if }1\leq s,t\leq n',\\
        V''(s(0),t(0);\dots;s-n',t-n';\dots;s(k),t(k))\mbox{ if }n'+1\leq s,t\leq n'+n'',\\
        0\mbox{ otherwise}.
    \end{cases}
\end{align*}
It follows that
\begin{align*}
    \,&
    \rho^{[k]}(f)_{n_{0};\dots;n_{j-1};n'+n'';n_{j+1};\dots;n_{k}}
    (X^{(0)};\dots;X^{(j-1)};X'\oplus X'';X^{(j+1)};\dots;X^{(k)})\\
    &=
    \sum_{\substack{1\leq s',t'\leq n'\\1\leq s(p),t(p)\leq n_{p}\\p\not=j}}e^{(n_{0})}_{s(0),t(0)}\otimes\cdots\otimes e^{(n_{j-1})}_{s(j-1),t(j-1)}\otimes (e^{(n')}_{s',t'}\oplus 0_{n''})\otimes e^{(n_{j+1})}_{s(j+1),t(j+1)}\otimes\cdots\otimes e^{(k)}_{s(k),t(k)}\\
    &\hspace{2cm}\otimes V'(s(0),t(0);\dots;s(j-1),t(j-1);s',t';s(j+1),t(j+1);\dots;s(k),t(k))\\
    &\quad+
    \sum_{\substack{1\leq s'',t''\leq n''\\1\leq s(p),t(p)\leq n_{p}\\p\not=j}}e^{(n_{0})}_{s(0),t(0)}\otimes\cdots\otimes e^{(n_{j-1})}_{s(j-1),t(j-1)}\otimes (0_{n'}\oplus e^{(n'')}_{s'',t''})\otimes e^{(n_{j+1})}_{s(j+1),t(j+1)}\otimes\cdots\otimes e^{(k)}_{s(k),t(k)}\\
    &\hspace{2cm}\otimes V''(s(0),t(0);\dots;s(j-1),t(j-1);s'',t'';s(j+1),t(j+1);\dots;s(k),t(k))\\
    &\simeq
    \Biggl(
    \sum_{\substack{1\leq s',t'\leq n'\\1\leq s(p),t(p)\leq n_{p}\\p\not=j}}e^{(n_{0})}_{s(0),t(0)}\otimes\cdots\otimes e^{(n_{j-1})}_{s(j-1),t(j-1)}\otimes e^{(n')}_{s',t'}\otimes e^{(n_{j+1})}_{s(j+1),t(j+1)}\otimes\cdots\otimes e^{(k)}_{s(k),t(k)}\\
    &\hspace{2cm}\otimes V'(s(0),t(0);\dots;s(j-1),t(j-1);s',t';s(j+1),t(j+1);\dots;s(k),t(k))
    \Biggr)\\
    &\quad
    \oplus
    \Biggl(
    \sum_{\substack{1\leq s'',t''\leq n''\\1\leq s(p),t(p)\leq n_{p}\\p\not=j}}e^{(n_{0})}_{s(0),t(0)}\otimes\cdots\otimes e^{(n_{j-1})}_{s(j-1),t(j-1)}\otimes e^{(n'')}_{s'',t''}\otimes e^{(n_{j+1})}_{s(j+1),t(j+1)}\otimes\cdots\otimes e^{(k)}_{s(k),t(k)}\\
    &\hspace{2cm}\otimes V''(s(0),t(0);\dots;s(j-1),t(j-1);s'',t'';s(j+1),t(j+1);\dots;s(k),t(k))
    \Biggr)\\
    &=
    \rho^{[k]}(f)_{n_{0};\dots;n_{j-1};n';n_{j+1};\dots;n_{k}}(X^{(0)};\dots;X^{(j-1)};X';X^{(j+1)};\dots;X^{(k)})\\
    &\quad
    \oplus
    \rho^{[k]}(f)_{n_{0};\dots;n_{j-1};n;n_{j+1};\dots;n_{k}}
    (X^{(0)};\dots;X^{(j-1)};X'';X^{(j+1)};\dots;X^{(k)}).
\end{align*}
Hence, we are done.

\underline{Similarity preserving property}: Choose arbitrary $X\in\Omega^{(j)}_{n}$ and $S\in GL_{n}(\mathcal{R})$. Write
\begin{align*}
    \,&
    \rho^{[k]}(f)_{n_{0};\dots;n_{j-1};n;n_{j+1};\dots;n_{k}}
    (X^{(0)};\dots;X^{(j-1)};X;X^{(j+1)};\dots;X^{(k)})\\
    &=
    \sum_{\substack{1\leq s,t\leq n\\1\leq s(p),t(p)\leq n_{p}\\p\not=j}}e^{(n_{0})}_{s(0),t(0)}\otimes\cdots\otimes e^{(n_{j-1})}_{s(j-1),t(j-1)}\otimes e^{(n)}_{s,t}\otimes e^{(n_{j+1})}_{s(j+1),t(j+1)}\otimes\cdots\otimes e^{(k)}_{s(k),t(k)}\\
    &\hspace{2cm}\otimes V(s(0),t(0);\dots;s(j-1),t(j-1);s,t;s(j+1),t(j+1);\dots;s(k),t(k))
\end{align*}
as above. By definition we have 
\begin{align*}
    \,&
    f(X^{(0)};\dots;X^{(j-1)};SXS^{-1};X^{(j+1)};\dots;X^{(k)})
    (Z^{(1)};\dots;Z^{(j-1)};Z^{(j)};Z^{(j+1)};\dots;Z^{(k)})\\
    &=f(X^{(0)};\dots;X^{(j-1)};X;X^{(j+1)};\dots;X^{(k)})
    (Z^{(1)};\dots;Z^{(j-1)};Z^{(j)}S;S^{-1}Z^{(j+1)};\dots;Z^{(k)}).
\end{align*}
It follows that 
\begin{align*}
    \,&
    \rho^{[k]}(f)_{n_{0};\dots;n_{j-1};n;n_{j+1};\dots;n_{k}}
    (X^{(0)};\dots;X^{(j-1)};SXS^{-1};X^{(j+1)};\dots;X^{(k)})\\
    &=
    \sum_{\substack{1\leq s,t\leq n\\1\leq s(p),t(p)\leq n_{p}\\p\not=j}}e^{(n_{0})}_{s(0),t(0)}\otimes\cdots\otimes e^{(n_{j-1})}_{s(j-1),t(j-1)}\otimes Se^{(n)}_{s,t}S^{-1}\otimes e^{(n_{j+1})}_{s(j+1),t(j+1)}\otimes\cdots\otimes e^{(k)}_{s(k),t(k)}\\
    &\hspace{2cm}\otimes V(s(0),t(0);\dots;s(j-1),t(j-1);s,t;s(j+1),t(j+1);\dots;s(k),t(k))\\
    &=
    (\mathrm{Ad}(S)\otimes\mathrm{id}_{\mathcal{U}})
    \left(
    \rho^{[k]}(f)_{n_{0};\dots;n_{j-1};n;n_{j+1};\dots;n_{k}}
    (X^{(0)};\dots;X^{(j-1)};X;X^{(j+1)};\dots;X^{(k)})
    \right),
\end{align*}
where $\mathcal{U}=M_{n_{0}}(\mathbb{C})\otimes\cdots\otimes M_{n_{j-1}}(\mathbb{C})\otimes M_{n_{j+1}}(\mathbb{C})\otimes\cdots\otimes M_{n_{k}}(\mathbb{C})\otimes\mathcal{N}^{(0)}$.

Conversely, choose an arbitrary $F\in\mathcal{F}^{(k+1)}(\Omega^{(0)};\Omega^{(1)};\dots;\Omega^{(k+1)}:\mathcal{N}^{(0)})$. 
Letting 
\begin{align*}
\,&
f(X^{(0)};\dots;X^{(k)}):=(\iota_{n_{0};\dots;n_{k}})^{-1}(F_{n_{0};\dots;n_{k}}(X^{(0)};\dots;X^{(k)}))\\
&\hspace{2cm}\in\mathrm{Hom}_{\mathbb{C}}(M_{n_{0},n_{1}}(\mathbb{C})\otimes\cdots\otimes M_{n_{k-1},n_{k}}(\mathbb{C}),M_{n_{0},n_{k}}(\mathbb{C})\otimes\mathcal{N}^{(0)})
\end{align*}
we can show that $f\in\mathcal{T}^{(k)}(\Omega^{(0)},\Omega^{(1)},\dots,\Omega^{(k)};\mathcal{N}^{(0)},\mathbb{C}_{\mathrm{nc}},\dots,\mathbb{C}_{\mathrm{nc}})$ in a similar fashion to the above discussion. Here, it is clear that $\rho^{[k]}(f)=F$. Thus, the bijectivity of $\rho^{[k]}$ is obvious by construction.
\end{proof}

\subsubsection{The nc difference-differential operator $\Delta$ vs Voiculescu's derivation-comultiplication $\partial$}
To understand Voiculescu's derivation-comultiplication in the context of \cite{kvv14}, we have to recall the notion of \textit{directional} nc difference-differential operator given in \cite{kvv14}.

\begin{Def}(\cite[section 2.6]{kvv14})
     Let $\mathcal{M}$ and $\mathcal{N}$ be modules over a commutative ring $\mathcal{R}$ and $\Omega$ be a right admissible nc set over $\mathcal{M}$. For a fixed $a\in\mathcal{M}$, we define
     \begin{equation*}
         (\Delta_{a} f)
         \left(X,Y\right)(R)
         :=(\Delta f)(X,Y)
         \left(Ra^{\oplus m}\right)\in M_{n,m}(\mathcal{N})
     \end{equation*}
     for any $f\in\mathcal{T}^{(0)}(\Omega;\mathcal{N}_{\mathrm{nc}})$, $X\in\Omega_{n}$, $Y\in\Omega_{m}$ and $R\in M_{n,m}(\mathcal{R})$.
\end{Def}

\begin{Def}(\cite[section 3.5]{kvv14})
    Let $\mathcal{M}^{(j)}$ and $\mathcal{N}^{(j)}$, $j=0,1,\dots,k$, be modules over a commutative ring $\mathcal{R}$, and $\Omega^{(j)}$ be a nc set over $\mathcal{M}^{(j)}$ for each $j=0,1,\dots,k$. Choose an arbitrary $f\in\mathcal{T}^{(k)}(\Omega^{(0)},\Omega^{(1)},\dots,\Omega^{(k)};\mathcal{N}^{(0)}_{\mathrm{nc}},\mathcal{N}^{(1)}_{\mathrm{nc}},\dots,\mathcal{N}^{(k)}_{\mathrm{nc}})$.
    \begin{enumerate}
        \item[($0$)] Assume that $\Omega^{(0)}$ is right admissible and fix an $a\in\mathcal{M}^{(0)}$. We define
        \begin{align*}
            \,&
            (_{k,0}\Delta_{a}f)
            \left(
            X'^{(0)};X''^{(0)};X^{(1)};\dots;X^{(k)}
            \right)
            \left(
            R;Z''^{(1)};Z^{(2)};\dots;Z^{(k)}
            \right)\\
            &:=
            (_{k,0}\Delta f)
            \left(
            X'^{(0)};X''^{(0)};X^{(1)};\dots;X^{(k)}
            \right)
            \left(
            Ra^{\oplus n''_{0}};Z''^{(1)};Z^{(2)};\dots;Z^{(k)}
            \right)\in M_{n'_{0},n_{k}}(\mathcal{N}^{(0)})
        \end{align*}
        for any $X'^{(0)}\in\Omega^{(0)}_{n'_{0}}$, $X''^{(0)}\in\Omega^{(0)}_{n''_{0}}$, $X^{(i)}\in\Omega^{(i)}_{n_{i}}$, $R\in M_{n'_{0},n''_{0}}(\mathcal{R})$, $Z''^{(1)}\in M_{n''_{0},n_{1}}(\mathcal{N}^{(1)})$, $Z^{(j)}\in M_{n_{j-1},n_{j}}(\mathcal{N}^{(j)})$ for $i=1,2,\dots,k$ and $j=2,3,\dots,k$.
        \item[($j$)] Assume that $\Omega^{(j)}$ is right admissible for some $j=1,2,\dots,k-1$ and fix an $a\in\mathcal{M}^{(j)}$. We define
        \begin{align*}
            \,&
            (_{k,j}\Delta_{a}f)
            \left(
            X^{(0)};\dots;X^{(j-1)};X'^{(j)};X''^{(j)};X^{(j+1)};\dots;X^{(k)}
            \right)\\
            &\hspace{2cm}
            \left(
            Z^{(1)};\dots;Z^{(j-1)};Z'^{(j)};R;Z''^{(j+1)}
            ;Z^{(j+2)};\dots;Z^{(k)}
            \right)\\
            &:=
            (_{k,j}\Delta f)
            \left(
            X^{(0)};\dots;X^{(j-1)};X'^{(j)};X''^{(j)};X^{(j+1)};\dots;X^{(k)}
            \right)\\
            &\hspace{2cm}
            \left(
            Z^{(1)};\dots;Z^{(j-1)};Z'^{(j)};Ra^{\oplus n''_{j}};Z''^{(j+1)}
            ;Z^{(j+2)};\dots;Z^{(k)}
            \right)\in M_{n_{0},n_{k}}(\mathcal{N}^{(0)})
        \end{align*}
        for any $X^{(i)}\in\Omega^{(i)}_{n_{i}}$, $X'^{(j)}\in\Omega^{(j)}_{n'_{j}}$, $X''^{(j)}\in\Omega^{(j)}_{n''_{j}}$, $R\in M_{n'_{j},n''_{j}}(\mathcal{R})$, $Z^{(m)}\in M_{n_{m-1},n_{m}}(\mathcal{N}^{(m)})$, $Z'^{(j)}\in M_{n_{j-1},n'_{j}}(\mathcal{N}^{(j)})$, $Z''^{(j+1)}\in M_{n''_{j},n_{j+1}}(\mathcal{N}^{(j+1)})$ for $i=0,1,\dots,j-1,j+1,\dots,k$ and $m=1,2,\dots,j-1,j+2,\dots,k$.
        \item[($k$)] Assume that $\Omega^{(k)}$ is right admissible and fix $a\in\mathcal{M}^{(k)}$. We define 
        \begin{align*}
            \,&
            (_{k,k}\Delta_{a}f)
            \left(X^{(0)};\dots;X^{(k-1)};X'^{(k)};X''^{(k)}\right)
            \left(
            Z^{(1)};\dots;Z'^{(k)};R
            \right)\\
            &:=
            (_{k,k}\Delta f)
            \left(X^{(0)};\dots;X^{(k-1)};X'^{(k)};X''^{(k)}\right)
            \left(
            Z^{(1)};\dots;Z'^{(k)};Ra^{\oplus n''_{k}}
            \right)\in M_{n_{0},n''_{k}}(\mathcal{N}^{(0)})
        \end{align*}
        for any $X^{(i)}\in\Omega^{(i)}_{n_{i}}$, $X'^{(k)}\in\Omega^{(k)}_{n'_{k}}$, $X''^{(k)}\in\Omega^{(k)}_{n''_{k}}$, $R\in M_{n'_{k},n''_{k}}(\mathcal{R})$, $Z^{(j)}\in M_{n_{j-1},n_{j}}(\mathcal{N}^{(j)})$, $Z'^{(k)}\in M_{n_{k-1},n'_{k}}(\mathcal{N}^{(k)})$ for $i=0,1,\dots,k-1$ and $j=1,2,\dots,k-1$.
    \end{enumerate}
\end{Def}

\begin{Rem}\label{remdeltadnd}
    In the case when $\mathcal{M}=\mathbb{C}$, the directional nc difference-differential operator $\Delta_{1}$ is exactly same as the (non-directional) nc difference-differential operator $\Delta$.
\end{Rem}

Let us consider the case when $\mathcal{R}=\mathbb{C}$, that is, $\mathcal{M}$, $\mathcal{M}^{(j)}$, $\mathcal{N}$, $\mathcal{N}^{(j)}$ are vector spaces over $\mathbb{C}$. Moreover, assume that $\mathcal{N}^{(i)}=\mathbb{C}$ for any $i=1,2,\dots,k$. Hence, all the spaces of nc functions that will appear in the discussion bellow are in the form of
\begin{align*}
    \mathcal{T}^{(k)}(\Omega^{(0)},\dots,\Omega^{(k)};\mathcal{N}_{\mathrm{nc}},\underbrace{\mathbb{C}_{\mathrm{nc}},\dots,\mathbb{C}_{\mathrm{nc}}}_{\mbox{$k$-times}}).
\end{align*}
For the ease of notation, we write
\begin{equation*}
    \mathcal{T}^{(k)}(\Omega^{(0)},\dots,\Omega^{(k)};\mathcal{N}_{\mathrm{nc}})
    =
    \mathcal{T}^{(k)}(\Omega^{(0)},\dots,\Omega^{(k)};\mathcal{N}_{\mathrm{nc}},\underbrace{\mathbb{C}_{\mathrm{nc}},\dots,\mathbb{C}_{\mathrm{nc}}}_{\mbox{$k$-times}}).
\end{equation*}
In what follows, we will give the counterpart of Voiculescu's derivation-comoultiplcation $\partial$ in the context of \cite{kvv14}.

\begin{Def}
    Let $\Omega$ be a right admissible nc set over $\mathcal{M}$. Also, let us fix $a\in\mathcal{M}$. Then, we define a map $\partial_{a}$ from $\mathcal{F}^{(1)}(\Omega:\mathcal{N})=\mathcal{T}^{(0)}(\Omega;\mathcal{N}_{\mathrm{nc}})$ to $\mathcal{F}^{(2)}(\Omega;\Omega:\mathcal{N})$ as follows.
    \begin{align*}
        \partial_{a}=\rho^{[1]}\circ\Delta_{a},
    \end{align*}
    where $\rho^{[1]}$ is the linear isomorphism in Proposition \ref{propcorrespondence}, that is,
    \begin{equation*}
        \begin{CD}
            \mathcal{T}^{(0)}(\Omega;\mathcal{N}_{\mathrm{nc}})
            @>{\Delta_{a}}>>\mathcal{T}^{(1)}(\Omega,\Omega;\mathcal{N}_{\mathrm{nc}})\\
            @| @VV{\rho^{[1]}}V\\
            \mathcal{F}^{(1)}(\Omega:\mathcal{N})@>>{\partial_{a}}>\mathcal{F}^{(2)}(\Omega;\Omega:\mathcal{N})
        \end{CD}.
    \end{equation*}
\end{Def}

\begin{Def}
    Let $\Omega^{(i)}$ be a nc set over $\mathcal{M}^{(i)}$ for each $i=1,\dots,k$. Assume that $\Omega^{(j)}$ is right admissible for some $1\leq j\leq k$ and fix an $a\in\mathcal{M}^{(j)}$. Then, we define a map $\,_{k,j}\partial_{a}$ (or $\mathrm{id}^{\otimes(j-1)}\otimes\partial_{a}\otimes\mathrm{id}^{\otimes(k-j)}$) from $\mathcal{F}^{(k)}(\Omega^{(1)};\dots;\Omega^{(k)}:\mathcal{N})$ to $\mathcal{F}^{(k+1)}(\Omega^{(1)};\dots;\Omega^{(j)};\Omega^{(j)};\dots;\Omega^{(k)}:\mathcal{N})$ as follows.
    \begin{equation*}
        \,_{k,j}\partial_{a}
        =\rho^{[k]}\circ\,_{k-1,j-1}\Delta_{a}\circ(\rho^{[k-1]})^{-1},
    \end{equation*}
    that is,
    \begin{equation*}
        \begin{CD}
            \mathcal{T}^{(k-1)}(\Omega^{(1)},\dots,\Omega^{(k)};\mathcal{N}_{\mathrm{nc}})@>{\,_{k-1,j-1}\Delta_{a}}>>\mathcal{T}^{(k)}(\Omega^{(1)},\dots,\Omega^{(j)},\Omega^{(j)},\dots,\Omega^{(k)};\mathcal{N}_{\mathrm{nc}})\\
            @V{\rho^{[k-1]}}VV @VV{\rho^{[k]}}V\\
            \mathcal{F}^{(k)}(\Omega^{(1)};\dots;\Omega^{(k)}:\mathcal{N})
            @>>{\,_{k,j}\partial_{a}}>\mathcal{F}^{(k+1)}(\Omega^{(1)};\dots;\Omega^{(j)};\Omega^{(j)};\dots;\Omega^{(k)}:\mathcal{N})
        \end{CD}.
    \end{equation*}
\end{Def}

Let us see that Voiculescu's derivation-comultiplication $\partial$ is a special case of the above $\partial_{a}$ and $\,_{k,j}\partial_{a}$. To do so, let $\mathcal{M}$ be an operator system $G$ (that is, $G$ is a $*$-closed subspace of some unital C*-algebra with unit $1$) and $\mathcal{N}=\mathbb{C}$. Also, take the unit element $1$ as fixed element $a\in G$. Then, we have the following:
\begin{Prop}
    Let $\Omega$ be an open fully matricial $G$-set (i.e., a right admissible nc set over $G$). Then, the map $\partial_{1}$ from $\mathcal{F}^{(1)}(\Omega:\mathbb{C})$ to $\mathcal{F}^{(2)}(\Omega;\Omega:\mathbb{C})$ is the extension of Voiculescu's derivation $\partial$ from $A(\Omega)$ to $A(\Omega;\Omega)$ (see Definition \ref{defkariablefuuly}).
    (Equivalently, the operator $\Delta_{1}$ is the extension of the operator $\nabla$ in \cite{v04}.)
\end{Prop}
\begin{proof}
    This is clear by the constructions of $\partial_{1}$ and $\partial$ (see \cite[section 7]{v04}).
\end{proof}

Moreover, Voiculescu's 2nd order derivations $\partial\otimes\mathrm{id}$ and $\mathrm{id}\otimes\partial$ admit a similar fact to the above proposition.

\begin{Prop}
    Let $\Omega$ be an open fully matricial $G$-set. Then, the maps $\,_{2,1}\partial_{1}$ and $\,_{2,2}\partial_{1}$ from $\mathcal{F}^{(2)}(\Omega;\Omega:\mathbb{C})$ to $\mathcal{F}^{(3)}(\Omega;\Omega;\Omega:\mathbb{C})$ are the extensions of Voiculescu's 2nd order derivations $\partial\otimes\mathrm{id}$ and $\mathrm{id}\otimes\partial$ from $A(\Omega;\Omega)$ to $A(\Omega;\Omega;\Omega)$, respectively.
\end{Prop}
\begin{proof}
    We will confirm the statement for only $\partial\otimes\mathrm{id}$. (The confirmation for $\mathrm{id}\otimes\partial$ is similar to the following discussion.) Choose an arbitrary $f\in A(\Omega;\Omega)\subset\mathcal{F}^{(2)}(\Omega;\Omega:\mathbb{C})$. We can easily see that, for any $X\in\Omega_{n}$, $X'\in\Omega_{n'}$, $X''\in\Omega_{n''}$, $\gamma\in M_{n,n'}(\mathbb{C})$ and any basis $\{B_{i}\}_{1\leq i\leq (n'')^2}$ of $M_{n''}(\mathbb{C})$, there exists a family $\{A_{i}\}_{1\leq i\leq (n'')^2}$ of elements of $M_{n,n'}(\mathbb{C})$ such that
    \begin{align*}
        \,&
        f_{n+n';n''}
        \left(
        \begin{bmatrix}
            X&\gamma\otimes1\\
            0&X'
        \end{bmatrix};X''
        \right)
        =f_{n+n';n''}(X\oplus X';X'')
        +\sum_{1\leq i\leq (n'')^2}
        \begin{bmatrix}
            0_{n}&A_{i}\\
            0&0_{n'}
        \end{bmatrix}\otimes B_{i}
    \end{align*}
    and that
    \begin{align*}
        \left.
        \frac{d}{d\epsilon}
        f_{n+n';n''}
        \left(
        \begin{bmatrix}
            X&\epsilon(\gamma\otimes1)\\
            0&X'
        \end{bmatrix};X''
        \right)
        \right|_{\epsilon=0}
        =\sum_{1\leq i\leq (n'')^2}
        \begin{bmatrix}
            0_{n}&A_{i}\\
            0&0_{n'}
        \end{bmatrix}\otimes B_{i}.
    \end{align*}
    By the property of G\^{a}teaux derivative, we have a linear map
    \begin{align*}
        \,&
        ((\nabla\otimes\mathrm{id})f)_{n;n';n''}(X;X';X'')(-)\\
        &=
        \left(
        \begin{bmatrix}
            I_{n}&0_{n,n'}
        \end{bmatrix}\otimes I_{n''}
        \right)
        \cdot
        \left.
        \frac{d}{d\epsilon}
        f_{n+n';n''}
        \left(
        \begin{bmatrix}
            X&\epsilon(-\otimes1)\\
            0&X'
        \end{bmatrix};X''
        \right)
        \right|_{\epsilon=0}
        \cdot
        \left(
        \begin{bmatrix}
            0_{n,n'}\\
            I_{n'}
        \end{bmatrix}\otimes I_{n''}
        \right)
    \end{align*}
    from $M_{n,n'}(\mathbb{C})$ to $M_{n,n'}(\mathbb{C})\otimes M_{n''}(\mathbb{C})$. Here, it is easy to see that
    \begin{align*}
        \,&
        ((\nabla\otimes\mathrm{id})f)_{n;n';n''}(X;X';X'')(\gamma)\\
        &=
        \sum_{\substack{1\leq a,b\leq n\\1\leq c,d\leq n'\\1\leq e,f\leq n''}}
        \left(
        \left.
        \frac{d}{d\epsilon}
        f_{n+n';n''}
        \left(
        \begin{bmatrix}
            X&\epsilon(e^{(n,n')}_{b,c}\otimes1)\\
            0&X'
        \end{bmatrix}
        ;X''
        \right)
        \right|_{\epsilon=0}
        \right)_{(a,n+d)(e,f)}e^{(n)}_{a,b}\gamma e^{(n')}_{c,d}\otimes e^{(n'')}_{e,f}
    \end{align*}
    for any $\gamma\in M_{n,n'}(\mathbb{C})$. Hence, we can also regard $((\nabla\otimes\mathrm{id})f)_{n;n';n''}(X;X';X'')$ as a bilinear map from $M_{n,n'}(\mathbb{C})\times M_{n',n''}(\mathbb{C})$ to $M_{n,n''}(\mathbb{C})$ such that
    \begin{align*}
        \,&
        ((\nabla\otimes\mathrm{id})f)_{n;n';n''}(X;X';X'')(\gamma';\gamma'')\\
        &=
        \sum_{\substack{1\leq a,b\leq n\\1\leq c,d\leq n'\\1\leq e,f\leq n''}}
        \left(
        \left.
        \frac{d}{d\epsilon}
        f_{n+n';n''}
        \left(
        \begin{bmatrix}
            X&\epsilon(e^{(n,n')}_{b,c}\otimes1)\\
            0&X'
        \end{bmatrix}
        ;X''
        \right)
        \right|_{\epsilon=0}
        \right)_{(a,n+d)(e,f)}e^{(n)}_{a,b}\gamma' e^{(n')}_{c,d}\gamma'' e^{(n'')}_{e,f}
    \end{align*}
    for any $\gamma'\in M_{n,n'}(\mathbb{C})$ and $\gamma''\in M_{n',n''}(\mathbb{C})$.
    Via the natural isomorphism 
    \[
    \mathrm{Hom}_{\mathbb{C}}(M_{n,n'}(\mathbb{C})\otimes M_{n',n''}(\mathbb{C}),M_{n,n''}(\mathbb{C}))\simeq M_{n}(\mathbb{C})\otimes M_{n'}(\mathbb{C})\otimes M_{n''}(\mathbb{C}),
    \]
    $((\nabla\otimes\mathrm{id})f)_{n;n';n''}(X;X';X'')$ is exactly Voiculescu's 2nd order derivation $((\partial\otimes\mathrm{id})f)_{n;n';n''}(X;X';X'')$ (see \cite[subsection 7.10]{v04}). Since $(\partial\otimes\mathrm{id})f\in A(\Omega;\Omega;\Omega)\subset\mathcal{F}^{(3)}(\Omega;\Omega;\Omega:\mathbb{C})$, we have $(\nabla\otimes\mathrm{id})f\in \mathcal{T}^{(2)}(\Omega,\Omega,\Omega;\mathbb{C}_{\mathrm{nc}})$. Moreover, 
    we observe that
    \begin{align*}
        \,&
        f_{n+n';n''}
        \left(
        \begin{bmatrix}
            X&\gamma\otimes1\\
            0&X'
        \end{bmatrix};X''
        \right)\\
        &=f_{n+n';n''}(X\oplus X';X'')\\
        &\quad+
        \sum_{\substack{1\leq a,b\leq n\\1\leq c,d\leq n'\\1\leq e,f\leq n''}}
        \left(
        \left.
        \frac{d}{d\epsilon}
        f_{n+n';n''}
        \left(
        \begin{bmatrix}
            X&\epsilon(e^{(n,n')}_{b,c}\otimes1)\\
            0&X'
        \end{bmatrix}
        ;X''
        \right)
        \right|_{\epsilon=0}
        \right)_{(a,n+d)(e,f)}
        \begin{bmatrix}
            0_{n}&e^{(n)}_{a,b}\gamma e^{(n')}_{c,d}\\
            0&0_{n'}
        \end{bmatrix}
        \otimes e^{(n'')}_{e,f}.
    \end{align*}
    Thus, we have
    \begin{align*}
        \,&
        (\iota_{n+n';n''})^{-1}
        \left(
        f_{n+n';n''}
        \left(
        \begin{bmatrix}
            X&\gamma\otimes1\\
            0&X'
        \end{bmatrix};X''
        \right)
        \right)
        \left(
        \begin{bmatrix}
            Z\\
            Z'
        \end{bmatrix}
        \right)\\
        &=
        (\iota_{n+n';n''})^{-1}(f_{n+n';n''}(X\oplus X';X''))
        \left(
        \begin{bmatrix}
            Z\\
            Z'
        \end{bmatrix}
        \right)\\
        &\quad+
        \sum_{\substack{1\leq a,b\leq n\\1\leq c,d\leq n'\\1\leq e,f\leq n''}}
        \left(
        \left.
        \frac{d}{d\epsilon}
        f_{n+n';n''}
        \left(
        \begin{bmatrix}
            X&\epsilon(e^{(n,n')}_{b,c}\otimes1)\\
            0&X'
        \end{bmatrix}
        ;X''
        \right)
        \right|_{\epsilon=0}
        \right)_{(a,n+d)(e,f)}
        \begin{bmatrix}
            0_{n}&e^{(n)}_{a,b}\gamma e^{(n')}_{c,d}\\
            0&0_{n'}
        \end{bmatrix}
        \begin{bmatrix}
            Z\\
            Z'
        \end{bmatrix}
        e^{(n'')}_{e,f}\\
        &=
        \begin{bmatrix}
            (\iota_{n;n''})^{-1}(f_{n;n''}(X;X''))(Z)\\
            (\iota_{n';n''})^{-1}(f_{n';n''}(X';X''))(Z')
        \end{bmatrix}\\
        &\quad+
        \sum_{\substack{1\leq a,b\leq n\\1\leq c,d\leq n'\\1\leq e,f\leq n''}}
        \left(
        \left.
        \frac{d}{d\epsilon}
        f_{n+n';n''}
        \left(
        \begin{bmatrix}
            X&\epsilon(e^{(n,n')}_{b,c}\otimes1)\\
            0&X'
        \end{bmatrix}
        ;X''
        \right)
        \right|_{\epsilon=0}
        \right)_{(a,n+d)(e,f)}
        \begin{bmatrix}
            e^{(n)}_{a,b}\gamma e^{(n')}_{c,d}Z'e^{(n'')}_{e,f}\\
            0
        \end{bmatrix}\\
        &=
        \begin{bmatrix}
            (\iota_{n;n''})^{-1}(f_{n;n''}(X;X''))(Z)+((\nabla\otimes\mathrm{id})f)(X;X';X'')(\gamma;Z')\\
            (\iota_{n';n''})^{-1}(f_{n';n''}(X';X''))(Z')
        \end{bmatrix}
    \end{align*}
    for any $\gamma\in M_{n,n'}(\mathbb{C})$, $Z\in M_{n,n''}(\mathbb{C})$ and $Z''\in M_{n',n''}(\mathbb{C})$. Since this equality holds if we replace $((\nabla\otimes\mathrm{id})f)(X;X';X'')$ for $(\,_{1,0}\Delta_{1}f)(X;X';X'')$, it follows that $(\nabla\otimes\mathrm{id})f=\,_{1,0}\Delta_{1}f$.  Therefore, we obtain that $(\partial\otimes\mathrm{id})f=\,_{2,1}\partial_{1}f$.
\end{proof}

\begin{Rem}
    Voiculescu constructed the derivation-comultiplication $\partial$ only for analytic fully matricial functions. He proved the linearity of $\partial$ by the use of the property of G\^{a}teaux derivative of analytic functions on Banach spaces. Nevertheless, we can show the linearity of $\partial$ only by the ``algebraic'' condition to be a nc function (see \cite[Propositions 2.4, 2.6, 3.7-3.9, Theorem 3.10 and Remark 3.18]{kvv14}). 
\end{Rem}

\subsection{Summary}
So far, we have interpreted Voiculescu's theory in the framework of \cite{kvv14}. As a consequence, Voiculescu's theory correspond to the lower sequence and the theory of \cite{kvv14} does, as a special case, to the upper sequence in the following diagram:
\begin{equation*}
    \begin{CD}
        \mathcal{T}^{(0)}@>{\Delta_{a}}>>\mathcal{T}^{(1)}@>{\,_{1,j-1}\Delta_{a}}>>\cdots@>{\,_{k-2,j}\Delta_{a}}>>\mathcal{T}^{(k-1)}@>{\,_{k-1,j-1}\Delta_{a}}>>\mathcal{T}^{(k)}@>{\,_{k,j-1}\Delta_{a}}>>\cdots\\
        @| @V{\rho^{[1]}}VV @. @V{\rho^{[k-1]}}VV @V{\rho^{[k]}}VV\\
        \mathcal{F}^{(1)}@>>{\partial_{a}}>\mathcal{F}^{(2)}@>>{\,_{2,j}\partial_{a}}>\cdots@>>{\,_{k-1,j}\partial_{a}}>\mathcal{F}^{(k)}@>>{\,_{k,j}\partial_{a}}>\mathcal{F}^{(k+1)}@>>{\,_{k+1,j}\partial_{a}}>\cdots
    \end{CD}.
\end{equation*}
Here, each $\mathcal{T}^{(k)}$ is the space of $(k+1)$-variable nc functions that take values in the space of $k$-linear maps. On the other hand, each $\mathcal{F}^{(k)}$ is the space of $k$-variable nc functions that take values in the $k$-times tensor products of matrices. Then, each $\rho^{[k]}$ is a linear isomorphism from $\mathcal{T}^{(k)}$ to $\mathcal{F}^{(k+1)}$, which is defined via the isomorphism between the space of multilinear maps and the tensor products of square matrices. Also, for a fixed element $a$, each $\Delta_{a}$ and $\,_{k,j}\Delta_{a}$ are special values of $\Delta$ and $\,_{k,j}\Delta$. 

Since each $\rho^{[k]}$ is an isomorphism, we can easily go and come between the above two sequences. Therefore, the framework of \cite{kvv14} includes that of Voiculescu's fully matricial function theory.

\subsection{The relation between the present work and two previous works \cite{kvsv20,a20}}
The work \cite{kvsv20} determined the range of (higher order) nc difference-differential operator for any nc functions (of order $k\in\mathbb{N}$). The equivalence between conditions (1) and (2) in Theorem \ref{thmpoincarefree} is included in \cite[Theorem 3.5]{kvsv20}. But, the proof of Theorem \ref{thmpoincarefree} is shorter than that of \cite[Theorem 3.5]{kvsv20} (in the case of $k=0$), based on GDQ ring structure, for certain classes of nc analytic functions. This means that Mai and Speicher's algebraic approach \cite{ms19} can actually work even in some topological situations. Also, Theorem \ref{thmpoincarefree} finds condition (3) as a new one. 

The work \cite{a20} also determined the range of nc difference-differential operator for analytic nc functions in the framework of \textit{free vector field} (see \cite[Definition 2.1, Theorems 2.4 and 3.6]{a20}). Remark that the nc difference-differential operator, which was used in \cite{a20}, is restricted to the diagonal set of a domain, i.e., $\Delta f(X,X)$, where $f$ is an analytic nc function and $X$ is in the domain of $f$. Moreover, the discussion in \cite{a20} is different from that of \cite{kvsv20}, which is more accessible to non-exparts of free analysis.

On the other hand, (any counterpart of) the concept of cyclic derivatives (for nc functions) has not appeared in the context of \cite{kvv14} so far.

\section*{Acknowledgements}
We acknowledge our supervisor Professor Yoshimichi Ueda for his encouragements and editorial supports to this paper. We also acknowledge Dr.\ Tobias Mai for drawing our attention to various works including \cite{kvsv20,a20}, Dr.\ Akihiro Miyakawa for his advice on Remark \ref{remthetavoi} and Mr.\ Kenta Kojin for explaining us various aspects on non-commutative function theories.  

\begin{appendix}
\section{The series expansion of multivariable fully (stably) matricial analytic functions at the origin}\label{appendixpolyapproximate}
Voiculescu established the series expansion in the polynomial sub-bialgebra $\mathcal{Z}(B^d)$ at the origin of ``one-variable'' fully (resp. stably) matricial analytic functions on open affine fully (resp. stably) matricial $B$-sets that contain the origin (see \cite[section 13]{v10}). Let us call it \textit{Voiculescu's series expansion}. In fact, we can also use his arguments to extend the series expansion result to ``multivariable'' fully (resp. stably) matricial analytic functions. In this section, we record the result in the case of only $2$-variable fully (resp. stably) matricial functions with a few comment on necessary arguments (corresponding arguments for general cases are similar to the following). 

Assume that $B$ is finite-dimensional. Let $\Omega$ be an open affine fully matricial $B$-set such that $\Omega_{n}\ni 0_{n}$, where $0_{n}$ is the zero element of $M_{n}(B)$. Also, let $f$ be a $2$-variable fully matricial analytic function in $A(\Omega;\Omega)$. Then, $f$ has the series expansion at the origin as follows.
\begin{align*}
    \,&f_{n_{1};n_{2}}
    (\beta^{(1)};\beta^{(2)})=
    f_{n_{1};n_{2}}(0_{n_{1}};0_{n_{2}})
    +
    \sum_{l\geq1}
    \frac{1}{l!}
    (d^lf_{n_{1};n_{2}})(0_{n_{1}};0_{n_{2}})
    [
    \underbrace{(\beta^{(1)},\beta^{(2)});\dots;(\beta^{(1)},\beta^{(2)})}_{l\mbox{-times}}
    ],
\end{align*}
where $d^lf$ is the $l$-times Fr\'{e}chet derivative of $f$ and it is assumed that $\{(z\beta^{(1)},z\beta^{(2)})\,|\,z\in\mathbb{C},|z|<1+\epsilon\}\subset \Omega_{n_{1}}\oplus\Omega_{n_{2}}$, $\epsilon>0$. This series converges uniformly and absolutely on suitable compact subsets of $\Omega_{n_{1}}\times\Omega_{n_{2}}$ (see \cite[Theorem 7.11]{m86}). Also, $d^lf_{n_{1};n_{2}}(0_{n_{1}};0_{n_{2}})$ is an $l$-linear symmetric map from $(M_{n_{1}}(B)\oplus M_{n_{2}}(B))^{\times l}$ to $M_{n_{1}}(\mathbb{C})\otimes M_{n_{2}}(\mathbb{C})$. Then, we can see the counterpart of \cite[Lemma 13.2]{v10}.

\begin{Lem}
    Let us set 
    \begin{align*}
        \,&F_{n_{1};n_{2},l}(\beta)
        =
        (d^lf_{n_{1};n_{2}})(0_{n_{1}};0_{n_{2}})
        [\underbrace{\beta;\dots;\beta}_{\mbox{$l$-times}}]
    \end{align*}
    for any $\beta\in M_{n_{1}}(B)\times M_{n_{2}}(B)$. Then, $(F_{n_{1};n_{2},l})_{n_{1},n_{2}\in\mathbb{N}}$ is a $2$-variable fully matricial analytic function on $M(B)$.
\end{Lem}

By the universality of tensor products, there exist linear maps $\Phi_{n_{1};n_{2},l}$ from $M_{n_{1}}(\mathbb{C})^{\otimes l}\otimes M_{n_{2}}(\mathbb{C})^{\otimes l}\otimes B^{\otimes 2l}$ to $M_{n_{1}}(\mathbb{C})\otimes M_{n_{2}}(\mathbb{C})$ such that
\begin{align*}
    \,&F_{n_{1};n_{2},l}(A^{(1)}\otimes b^{(1)};A^{(2)}\otimes b^{(2)})
    =
    \Phi_{n_{1};n_{2},l}((A^{(1)})^{\otimes l}\otimes (A^{(2)})^{\otimes l}\otimes (b^{(1)})^{\otimes l}\otimes (b^{(2)})^{\otimes l})
\end{align*}
and that
\begin{align*}
    \,&\Phi_{n_{1};n_{2},l}(A^{(1)}_{1}\otimes\cdots\otimes A^{(1)}_{l}\otimes A^{(2)}_{1}\otimes\cdots\otimes A^{(2)}_{l}\otimes b^{(1)}_{1}\otimes\cdots\otimes b^{(1)}_{l}\otimes b^{(2)}_{1}\otimes\cdots\otimes b^{(2)}_{l})\\
    &=
    \Phi_{n_{1};n_{2},l}(A^{(1)}_{\sigma(1)}\otimes\cdots\otimes A^{(1)}_{\sigma(l)}\otimes A^{(2)}_{\sigma(1)}\otimes\cdots\otimes A^{(2)}_{\sigma(l)}\otimes b^{(1)}_{\sigma(1)}\otimes\cdots\otimes b^{(1)}_{\sigma(l)}\otimes b^{(2)}_{\sigma(1)}\otimes\cdots\otimes b^{(2)}_{\sigma(l)})
\end{align*}
for any $A^{(1)},A^{(1)}_{j}\in M_{n_{1}}(\mathbb{C})$, $A^{(2)},A^{(2)}_{j}\in M_{n_{2}}(\mathbb{C})$, $b^{(1)},b^{(1)}_{j},b^{(2)},b^{(2)}_{j}\in B$ and $\sigma\in\mathfrak{S}_{l}$. This $\Phi_{n_{1};n_{2},l}$ is the counterpart of $\Phi_{n,k}$ in \cite[subsection 13.3]{v10}. When we establish an analogue of \cite[Lemma 13.4]{v10}, we set the counterpart $\Psi_{n_{1};n_{2},l}$ of $\Psi_{n,k}$ as follows. $\Psi_{n_{1};n_{2},l}$ is the linear functional from $(M_{n_{1}}(\mathbb{C})\otimes M_{n_{2}}(\mathbb{C}))^{\otimes(l+1)}\otimes B^{\otimes 2l}$ to $\mathbb{C}$ defined by
\begin{align*}
    \,&\Psi_{n_{1};n_{2},l}((A^{(1)}_{1}\otimes A^{(2)}_{1})\otimes\cdots\otimes(A^{(1)}_{l+1}\otimes A^{(2)}_{l+1})\otimes b^{(1)}_{1}\otimes\cdots\otimes b^{(1)}_{l}\otimes b^{(2)}_{1}\otimes\cdots\otimes b^{(2)}_{l})\\
    &=
    (\mathrm{Tr}_{n_{1}}\otimes\mathrm{Tr}_{n_{2}})
    \Bigl[
    \Phi_{n_{1};n_{2}}
    \Bigl(A^{(1)}_{1}\otimes\cdots\otimes A^{(1)}_{l}\otimes A^{(2)}_{1}\otimes\cdots\otimes A^{(2)}_{l}\\
    &\hspace{5cm}\otimes b^{(1)}_{1}\otimes\cdots\otimes b^{(1)}_{l}\otimes b^{(2)}_{1}\otimes\cdots\otimes b^{(2)}_{l}\Bigr)
    (A^{(1)}_{l+1}\otimes A^{(2)}_{l+1})
    \Bigr]
\end{align*}
for any $A^{(1)},A^{(1)}_{j}\in M_{n_{1}}(\mathbb{C})$, $A^{(2)},A^{(2)}_{j}\in M_{n_{2}}(\mathbb{C})$, $b^{(1)},b^{(1)}_{j},b^{(2)},b^{(2)}_{j}\in B$, where $\mathrm{Tr}_{n}$ is the non-normalized trace of $n$ by $n$ matrices. 

In order to establish \cite[Lemma 13.4]{v10}, Voiculescu applied the Schur-Weyl duality to $M_{n}(\mathbb{C})^{\otimes(k+1)}$. In our case, we have to apply it to $(M_{n_{1}}(\mathbb{C})\otimes M_{n_{2}}(\mathbb{C}))^{\otimes(l+1)}$. (In this task, we need the similarity preserving property of fully matricial $B$-sets and the finite-dimensionality of $B$.) The remainders wil be done in the same line as \cite[subsections 13.5-13.8]{v10}. The desired result is the following, a counterpart of \cite[Theorem 13.8]{v10}: Recall that $B$ is finite-dimensional. Let $N\in\mathbb{N}$ be the dimension of $B$ and $\{\varphi_{1},\dots,\varphi_{N}\}$ be a basis of $B^d$.

\begin{Thm}\label{thmseriesexpansiontwo}
    Let $\Omega$ be an open affine fully matricial $B$-set containing the origin and $f\in A(\Omega;\Omega)$. Then, for each $l\in\mathbb{N}$ there exist complex numbers $\{a_{l}\left(i(1);\dots;i(l),j(1);\dots;j(l)\right)\,|\,l\in\mathbb{N},1\leq i(p)\leq N,1\leq j(p)\leq N,1\leq p\leq l\}$ such that
    \begin{align*}
        \,&(d^lf_{n_{1};n_{2}})(0_{n_{1};n_{2}})[(\beta^{(1)}_{1},\beta^{(2)}_{1});\dots;(\beta^{(1)}_{l},\beta^{(2)}_{l})]\\
        &=
        \sum_{\theta\in\mathfrak{S}_{l}}
        \sum_{\substack{1\leq i(p),j(p)\leq N\\1\leq p\leq l}}
        a_{l}\left(i(1);\dots;i(l),j(1);\dots;j(l)\right)\\
        &\hspace{4cm}\times
        z(\varphi_{i(1)})_{n_{1}}(\beta^{(1)}_{1})\cdots z(\varphi_{i(l)})_{n_{1}}(\beta^{(1)}_{l})\otimes
        z(\varphi_{j(1)})_{n_{2}}(\beta^{(2)}_{1})\cdots z(\varphi_{j(l)})_{n_{2}}(\beta^{(2)}_{l}).
    \end{align*}
    Equivalently, we have
    \begin{align*}
        \,&
        \frac{1}{l!}F_{n_{1};n_{2},l}
        =
        \sum_{\substack{1\leq i(p),j(p)\leq N\\1\leq p\leq l}}
        a_{l}\left(i(1);\dots;i(l),j(1);\dots;j(l)\right)
        z(\varphi_{i(1)})\cdots z(\varphi_{i(l)})\otimes
        z(\varphi_{j(1)})\cdots z(\varphi_{j(l)}).
    \end{align*}
    The formula for $d^lf_{n_{1};n_{2}}(0_{n_{1};0_{n_{2}}})$ uniquely determines the numbers $a_{l}$ if $n_{1}n_{2}\geq l+1$. The series expansion of $f$ at the origin is
    \begin{align*}
        f&
        =(f_{n_{1};n_{2}}(0_{n_{1}};0_{n_{2}}))_{n_{1},n_{2}\in\mathbb{N}}\\
        &\quad
        +
        \sum_{l\geq1}
        \sum_{\substack{1\leq i(p),j(p)\leq N\\1\leq p\leq l}}
        a_{l}\left(i(1);\dots;i(l),j(1);\dots;j(l)\right)
        z(\varphi_{i(1)})\cdots z(\varphi_{i(l)})\otimes
        z(\varphi_{j(1)})\cdots z(\varphi_{j(l)}).
    \end{align*}
\end{Thm}

\begin{Rem}
   If $\Omega$ is an open stably matricial $B$-set and $f$ be a $2$-variable stably matricial analytic function, then we apply the above theorem to the fully matricial extensions of $\Omega$ and $f$ (see \cite[Proposition 11.1 and Corollary 11.1]{v10}), and hence its series expansion restricted to $\Omega$ gives Voiculescu's series expansion of $f$. 
\end{Rem}

\section{The injectivity of $L_{\theta}$ and $N_{\theta,2}$ in the case of $B=M_{k}(\mathbb{C})$}\label{appendixinjenctions}
We have shown that $L_{\theta}|_{\mathcal{Z}(B^d)}$ and $N_{\theta,2}|_{\mathcal{Z}(B^d)^{\otimes2}}$ are injective (see Corollary \ref{corinjinc}). In this section, we will show their injectivity as operators on $A(R\mathcal{D}_{0}(M_{k}(\mathbb{C})))$ and $A(R\mathcal{D}_{0}(M_{k}(\mathbb{C}));R\mathcal{D}_{0}(M_{k}(\mathbb{C})))$, respectively, for any $R>0$. Let $\varphi_{1}(=:\theta),\varphi_{2},\dots,\varphi_{k^2}$ be a basis of $M_{k}(\mathbb{C})^d$ such that 
\begin{enumerate}
    \item $\{\alpha_{j}\}_{j=2}^{k^2}$ is an orthonormal system of the subspace of all traceless matrices of $M_{k}(\mathbb{C})$ with respect to the Hilbert-Schmidt inner product.
    \item $\varphi_{1}(X)=\mathrm{Tr}_{k}(Xk^{-1}I_{k})$, $\varphi_{j}(X)=\mathrm{Tr}_{k}(X\alpha_{j}^{t})$, $j=2,3,\dots,k^2$ for any $X\in M_{k}(\mathbb{C})$.
\end{enumerate}
Note that $I_{k}$ is orthogonal to $\{\alpha_{j}\}_{j=2}^{k^2}$ with respect to the Hilbert-Schmidt inner product, and hence $\{k^{-1}I_{k}\}\cup\{\alpha_{j}\}_{j=2}^{k^2}$ is an orthogonal basis of $M_{k}(\mathbb{C})$ with respect to the Hilbert-Schmidt inner product. Then, we have the following fact:
\begin{Thm}(\cite[Theorem 14.3]{v10})\label{thmvoiintegral}
    We denote by $\mathcal{U}(Nk)$ the unitary group of $M_{Nk}(\mathbb{C})$ and by $\mu_{Nk}$ the normalized Haar measure of $\mathcal{U}(Nk)$. Then, we have
    \begin{enumerate}
        \item If $n\not=m$, then
        \begin{align*}
            \int_{\mathcal{U}(Nk)}
            \frac{1}{N}\mathrm{Tr}_{N}
            \left[
            z(\varphi_{i(1)})_{N}(\omega)\cdots z(\varphi_{i(n)})_{N}(\omega)
            \left(
            z(\varphi_{j(1)})_{N}(\omega)\cdots z(\varphi_{j(m)})_{N}(\omega)
            \right)^*
            \right]
            \,d\mu_{Nk}(\omega)=0
        \end{align*}
        for $N\in\mathbb{N}$, where $1\leq i(p),j(q)\leq k^2$ for $1\leq p\leq n$ and $1\leq q\leq m$.
        \item 
        \begin{align*}
            \,&\lim_{N\to\infty}
            \int_{\mathcal{U}(Nk)}\frac{1}{N}\mathrm{Tr}_{N}
            \left[
            z(\varphi_{i(1)})_{N}(\omega)\cdots z(\varphi_{i(n)})_{N}(\omega)
            \left(
            z(\varphi_{j(1)})_{N}(\omega)\cdots z(\varphi_{j(m)})_{N}(\omega)
            \right)^*
            \right]
            \,d\mu_{Nk}(\omega)\\
            &=
            \delta_{n,m}
            \delta_{i(1),j(1)}\delta_{i(2),j(2)}\cdots\delta_{i(n),j(n)}k^{-\#[i(1);\dots;i(n)]},
        \end{align*}
        for $N\in\mathbb{N}$, where $1\leq i(p),j(q)\leq k^2$ for $1\leq p\leq n$ and $1\leq q\leq m$ and where $\#[i(1);\dots;i(n)]$ is the number of $1$ that appears in $i(1),i(2),\dots,i(n)$.
    \end{enumerate}
\end{Thm}

Let us show the injectivity of the grading operator $L_{\theta}$ with $\theta=\varphi_{1}$:
\begin{Thm}\label{thminjltheta}
    Assume $R>0$. Then, the grading operator 
    \begin{align*}
        L_{\theta}:A(R\mathcal{D}_{0}(M_{k}(\mathbb{C})))\to A(R\mathcal{D}_{0}(M_{k}(\mathbb{C})))
    \end{align*}
    with $\theta=\varphi_{1}$ is injective.
\end{Thm}
\begin{proof}
    Suppose $R>1$ and that $L_{\theta}f=0$ for some $f\in A(R\mathcal{D}_{0}(M_{k}(\mathbb{C})))$. Then, we have $f=-\partial f\#z(\theta)$ by definition. Using \cite[Theorem 13.8]{v10}, there exists a family $\{a_{l}(i(1);\dots;i(l))\in\mathbb{C}\,|\,l\in\mathbb{N},1\leq i(p)\leq k^2,1\leq p\leq l\}$ such that the series 
    \begin{align*}
        f_{1}(0)1_{A(R\mathcal{D}_{0}(M_{k}))}
        +
        \sum_{l\geq1}
        \sum_{\substack{1\leq i(p)\leq k^2\\1\leq p\leq l}}
        a_{l}(i(1);\dots;i(l))z(\varphi_{i(1)})\cdots z(\varphi_{i(l)})
    \end{align*}
    converges to $f$ in the uniform norm $\|\cdot\|_{n,R'}$ for each $n\in\mathbb{N}$ and $(1<)R'<R$ (cf. \cite[Theorem 7.11]{m86}). Here, by Corollary \ref{corcontifmfdq} and Lemma \ref{lemcontisharp}, we obtain that
    \begin{align*}
        \sum_{l\geq1}
        \sum_{\substack{1\leq i(p)\leq k^2\\1\leq p\leq l}}
        a_{l}(i(1);\dots;i(l))
        \partial\left[z(\varphi_{i(1)})\cdots z(\varphi_{i(l)})\right]\#z(\theta)
    \end{align*}
    converges to $\partial f\#z(\theta)$ uniformly in $\|\cdot\|_{n,R'}$ for each $n\in\mathbb{N}$ and $(1<)R'<R$. Using the uniform convergence of the series expansion of $f$ and Theorem \ref{thmvoiintegral}, we have
    \begin{align*}
        \,&a_{l}(j(1);\dots;j(l))k^{-\#[j(1);\dots;j(l)]}\\
        &=
        \lim_{N\to\infty}
        \int_{\mathcal{U}(Nk)}
        \frac{1}{N}\mathrm{Tr}_{N}
        \left[
        f_{N}(\omega)
        \left(
        z(\varphi_{j(1)})_{N}(\omega)\cdots z(\varphi_{j(l)})_{N}(\omega)
        \right)^*
        \right]
        \,d\mu_{Nk}(\omega).
    \end{align*}
    
    On the other hand, using the uniform convergence of the series expansion of $\partial f\#z(\theta)$ and by Theorem \ref{thmvoiintegral}, we have
    \begin{align*}
        \,&
        \lim_{N\to\infty}
        \int_{\mathcal{U}(Nk)}
        \frac{1}{N}\mathrm{Tr}_{N}
        \left[
        f_{N}(\omega)
        \left(
        z(\varphi_{j(1)})_{N}(\omega)\cdots z(\varphi_{j(l)})_{N}(\omega)
        \right)^*
        \right]
        \,d\mu_{Nk}(\omega)\\
        &=-\lim_{N\to\infty}
        \int_{\mathcal{U}(Nk)}
        \frac{1}{N}\mathrm{Tr}_{N}
        \left[
        (\partial f\#z(\theta))_{N}(\omega)
        \left(
        z(\varphi_{j(1)})_{N}(\omega)\cdots z(\varphi_{j(l)})_{N}(\omega)
        \right)^*
        \right]
        \,d\mu_{Nk}(\omega)\\
        &=
        -\lim_{N\to\infty}
        \sum_{\substack{1\leq i(1)\leq k^2\\1\leq p\leq l}}
        a_{l}(i(1);\dots;i(l))\\
        &\quad
        \int_{\mathcal{U}(Nk)}
        \frac{1}{N}\mathrm{Tr}_{N}
        \left[
        \left(\partial\left[z(\varphi_{i(1)})\cdots z(\varphi_{i(l)})\right]\#z(\theta)\right)_{N}(\omega)
        \left(
        z(\varphi_{j(1)})_{N}(\omega)\cdots z(\varphi_{j(l)})_{N}(\omega)
        \right)^*
        \right]
        \,d\mu_{Nk}(\omega).
    \end{align*}
    
    If there is not a $p\in\{1,2,\dots,l\}$ such that $j(p)=1$ (that is, $\varphi_{j(p)}\not=\theta$), then we have by Theorem \ref{thmvoiintegral}
        \begin{align*}
            \int_{\mathcal{U}(Nk)}
            \frac{1}{N}\mathrm{Tr}_{N}
            \left[
            \left(\partial\left[z(\varphi_{i(1)})\cdots z(\varphi_{i(l)})\right]\#z(\theta)\right)_{N}(\omega)
            \left(
            z(\varphi_{j(1)})_{N}(\omega)\cdots z(\varphi_{j(l)})_{N}(\omega)
            \right)^*
            \right]
            \,d\mu_{Nk}(\omega)
            \to0
        \end{align*}
        as $N\to\infty$. Thus, we obtain that $a_{l}(j(1);\dots;j(l))=0$ if there is not a $p\in\{1,2,\dots,l\}$ such that $j(p)=1$.
    
    If there is a $p\in\{1,2,\dots,l\}$ such that $j(p)=1$ (that is, $\varphi_{j(p)}=\theta$), then let $m\in\mathbb{N}$ be the number of $z(\theta)=z(\varphi_{1})$ that appears in $z(\varphi_{j(1)})\cdots z(\varphi_{j(l)})$. Also, let $1\leq s_{1}<s_{2}<\cdots s_{m}\leq l$ be numbers such that $j(s_{p})=1$ for $p=1,2,\dots,m$. Here, note that a monomial $Q$ that belongs to $\{z(\varphi_{i(1)})\cdots z(\varphi_{i(l)})\,|\,1\leq i(p)\leq k^2,1\leq p\leq l\}$ such that
        \begin{align*}
            \lim_{N\to\infty}\int_{\mathcal{U}(Nk)}N^{-1}\mathrm{Tr}_{N}\left[(\partial Q\#z(\theta))_{N}(z(\varphi_{j(1)})\cdots z(\varphi_{j(l)}))_{N}^*\right]\,d\mu_{Nk}\not=0
        \end{align*}
        is only
        \begin{align*}
            Q=z(\varphi_{j(1)})\cdots z(\varphi_{j(l)}),
        \end{align*}
        since $\partial[z(\varphi_{j})]=\varphi_{j}(1)1_{A(R\mathcal{D}_{0}(M_{k})(\mathbb{C}))}\otimes1_{A(R\mathcal{D}_{0}(M_{k}(\mathbb{C})))}=0$ for every $j=2,3,\dots,k^2$.
        Hence, we have
        \begin{align*}
            \,&a_{l}(j(1);\dots;j(l))\\
            &=-\lim_{N\to\infty}
            a_{l}(j(1);\dots;j(l))\\
            &\quad
            \int_{\mathcal{U}(Nk)}
            \frac{1}{N}\mathrm{Tr}_{N}
            \left[
            \left(\partial[z(\varphi_{j(1)})\cdots z(\varphi_{j(l)})]\#z(\theta)\right)_{N}(\omega)
            \left(
            z(\varphi_{j(1)})_{N}(\omega)\cdots z(\varphi_{j(l)})_{N}(\omega)
            \right)^*
            \right]
            \,d\mu_{Nk}(\omega)\\
            &=-m
            \lim_{N\to\infty}
            a_{l}(j(1);\dots;j(l))\\
            &\quad
            \int_{\mathcal{U}(Nk)}
            \frac{1}{N}\mathrm{Tr}_{N}
            \left[
            z(\varphi_{j(1)})_{N}(\omega)\cdots z(\varphi_{j(l)})_{N}(\omega)
            \left(
            z(\varphi_{j(1)})_{N}(\omega)\cdots z(\varphi_{j(l)})_{N}(\omega)
            \right)^*
            \right]
            \,d\mu_{Nk}(\omega)\\
            &=-m\cdot a_{l}(j(1);\dots;j(l))k^{-m},
        \end{align*}
        that is, $a_{l}(j(1);\dots;j(l))=0$.
    Therefore, all numbers $\{a_{l}(i(1);\dots;i(l))\}$ are $0$, that is, $f=0$. We are done.

    In the case of $0<R\leq1$, choose an arbitrary $f\in A(R\mathcal{D}_{0}(M_{k}(\mathbb{C})))$. We apply the above argument to $f(r\cdot)\in A(R'\mathcal{D}_{0}(M_{k}(\mathbb{C})))$ for every $0<r<R$, where $R'=r^{-1}>R^{-1}\geq1$.
\end{proof}

Using Theorem \ref{thmseriesexpansiontwo}, we can show the injectivity of $N_{\theta,2}$ with $\theta=\varphi_{1}$ in a similar fashion to the proof of Theorem \ref{thminjltheta}:

\begin{Thm}\label{thminjnumtheta2}
    Assume $R>0$. Then, the number operator 
    \begin{align*}
        N_{\theta,2}:A(R\mathcal{D}_{0}(M_{k}(\mathbb{C}));R\mathcal{D}_{0}(M_{k}(\mathbb{C})))\to A(R\mathcal{D}_{0}(M_{k}(\mathbb{C}));R\mathcal{D}_{0}(M_{k}(\mathbb{C})))
    \end{align*}
    with $\theta=\varphi_{1}$ is injective.
\end{Thm}

\begin{Rem}
    We can also show the injectivity of $L_{\theta}$ and $N_{\theta,2}$ with $\theta(X)=\mathrm{Tr}_{k}(X e^{(k)}_{i,i})$ for $X\in M_{k}$ in the same line as above theorems.
\end{Rem}
\end{appendix}
}

\end{document}